\documentclass[12pt,a4paper]{article}
\usepackage{authblk}
\usepackage{fancyhdr}
\usepackage{caption}
\usepackage{float}
\usepackage{url}
\usepackage{graphicx}
\usepackage{mathptmx}
\usepackage{latexsym}
\usepackage{amssymb}
\usepackage{amsmath}
\usepackage{amsthm}
\usepackage{amsfonts}
\usepackage{silence}
\newtheorem{theorem}{Theorem}[section]
\newtheorem{proposition}{Proposition}[section]
\newtheorem{corollary}{Corollary}[section]
\newtheorem{conjecture}{Conjecture}[section]
\newtheorem{example}{Example}[section]

\numberwithin{equation}{section}
\numberwithin{table}{section}
\title{Braid orbits and the Mathieu group $M_{23}$ as Galois group}
\author{Frank H\"afner}
\begin{document}
\maketitle
%
\abstract{
\noindent
At present, the inverse Galois problem over $\mathbb{Q}$ is unsolved for the Mathieu group $M_{23}$.
Here an overview of the current state in realizing $M_{23}$ as Galois group using the rigidity method
and the action of braids is given. Computing braid orbits for $M_{23}$ revealed new invariants of the
action of braids in addition to Fried's lifting invariant. These invariants can be used to construct
generic braid orbits and more Galois realizations over $\mathbb{Q}$ for the Mathieu group $M_{24}$,
but until now did not lead to success for realising $M_{23}$ as Galois group over $\mathbb{Q}$.
Thus $M_{23}/\mathbb{Q}$ remains open. Finally, heuristics for searching suitable class vectors with
regard to the realization of groups as Galois groups are given.}

\section{Introduction}\label{SectionIntro}

At present, the inverse Galois problem over $\mathbb{Q}$ is unsolved for the Mathieu group $M_{23}$,
i.e. no polynomial $f(X) \in \mathbb{Z}[X]$ with $Gal(f(X)) \cong M_{23}$ is known. The other $25$
sporadic simple groups have been realized as Galois groups over $\mathbb{Q}$ using the rigidity 
method in \cite[I]{RefMM2018} and the action of braids in \cite[III]{RefMM2018} for the 
Mathieu group $M_{24}$.
For details, see \cite[II, Theorem 9.9]{RefMM2018} and \cite[III, Theorem 7.12]{RefMM2018}.
The geometric Galois extensions over $\mathbb{Q}(t)$ resp. $\mathbb{Q}(v,t)$ in \cite{RefMM2018}
can be specialized to Galois extensions over $\mathbb{Q}$ by applying Hilbert's irreducibility criterion.
In \cite{RefHS1985} and \cite{RefHae1987}, geometric $M_{23}$-Galois extensions over $\mathbb{Q}(\sqrt{-23})(t)$
and $\mathbb{Q}(\sqrt{-7})(t)$ have been constructed from geometric $M_{24}$-Galois extensions
$N_1/\mathbb{Q}(\sqrt{-23})(t)$ and $N_2/\mathbb{Q}(\sqrt{-7})(t)$ by verification
that the fixed fields $N_1^{M_{23}}$ and $N_2^{M_{23}}$ are rational function fields. 
The $M_{24}$-Galois extensions have been achieved by using the rigidity method with
class vectors $(2A,3B,23A)$ and $(2B,3A,21A)$. No suitable class vectors of length $3$, that admit geometric Galois
extensions over $\mathbb{Q}(t)$ for $M_{23}$ have been found, see for example \cite{RefHS1984} and \cite{RefHae1987}.
\par
\noindent
Next, one can use class vectors of length $4$ and study the action of the Hurwitz braid group on
the set of generating systems. This approach was applied with success for $M_{24}$ and generated three
related geometric Galois extensions $N_3/\mathbb{Q}(v,t)$, $N_4/\mathbb{Q}(v,t)$ and
$N_5/\mathbb{Q}(v,t)$ with Galois group $M_{24}$.
The corresponding class vectors are $(12B,2A,2A,2A)$, $(12B,12B,2A,2A)$ and $(12B,12B,12B,12B)$,
see \cite{RefHae1987}, \cite{RefHae1991}, \cite{RefMat1991}, \cite{RefMM2018}
and section \ref{SectionClassVectorsOfDimension4InM24} below. 
The fixed fields $N_i^{M_{23}}$, $i=3,4,5$ have genus $g_3=0$, $g_4=7$ and $g_5=21$.
Thus, only the first one can be rational, but explicit computation of a generating polynomial for 
$N_3/\mathbb{Q}(v,t)$ in \cite{RefGra1996} shows, that $N_3^{M_{23}}$ is not a rational function field.
Using the class vectors $(14A,2A,2A,2A)$ and $(15A,2A,2A,2A)$ of $M_{23}$ leads to geometric Galois extensions
$N_6/\mathbb{Q}(\sqrt{-7})(v,t)$ and $N_7/\mathbb{Q}(\sqrt{-15})(v,t)$ with Galois group $M_{23}$,
see section \ref{SectionClassVectorsOfDimension4InM23} below.
\par
\noindent
In \cite{RefHae1991}, no suitable rational class vectors of length $4$ of $M_{23}$ or $M_{24}$
for realizing $M_{23}$ as Galois group over $\mathbb{Q}(t)$ have been found. The search was limited due
to computing resources but could be continued later with increasingly powerful machines.
Looking at class vectors of length $4$ of small groups gave some hints for further investigations.
Especially symmetric class vectors $(C,C,C,C)$ with a rational conjugacy class $C$ seem to lead to small
braid orbits that are suitable for Galois realizations.
In case of $M_{23}$, these class vectors are $(3A,3A,3A,3A)$, $(4A,4A,4A,4A)$, $(5A,5A,5A,5A)$, $(6A,6A,6A,6A)$
and $(8A,8A,8A,8A)$. Their properties with regard to realizing $M_{23}$ as Galois group over $\mathbb{Q}(v,t)$
will be shown in section \ref{SectionClassVectorsOfDimension4InM23}.
But first, new invariants of the action of braids in addition to Fried's lifting invariant
in \cite{RefFri2010} are introduced. The key observation here is the fact, that the braid group acts on
fixed points that stem from the action of one of its subgroups, see \cite[p. 4]{RefHae1991}.
Studying these fixed points leads to the construction of generic braid orbits.
This will be explained in sections \ref{SectionFixedPointsAndOrbitsOfSize2}, \ref{SectionSymmetriesAndOrbitsOfSize6},
\ref{SectionMoreBraidOrbits}, \ref{SectionFixedPointsAndOrbitsOfSize4}, \ref{SectionSmallOrbitsInDimension6} and
\ref{SectionDerivedClassvectors}.
Then we look at class vectors of length $m$ of $M_{23}$ and $M_{24}$ for $m=4,5,6$ in 
sections \ref{SectionClassVectorsOfDimension4InM23}, \ref{SectionClassVectorsOfDimension4InM24} and
\ref{SectionClassvectorsOfDimension5And6}.
In section \ref{SectionHeuristicsForSearching}, heuristics for searching suitable class vectors with regard
to the realization of groups as Galois groups are given.
Thus, while the understanding of the action of braids could be improved and new invariants of this action
have been found leading to generic braid orbits, the problem, whether $M_{23}$ occurs as Galois group
over $\mathbb{Q}$ still remains open.
\par
\noindent
Properties of finite simple groups are taken from \cite{RefAtlas1985}.
For group elements $\sigma, \tau$ we use $\sigma^{\tau} = \tau^{-1} \sigma \tau$.
If $\rho : G \rightarrow S_k$ is a permutation representation of a group $G$ into
the symmetric group $S_k$, we write $\rho^t(\gamma)$ for the permutation type of 
$\rho(\gamma) \in S_k$ with $\gamma \in G$.
In tables, we use '-', if no information exists, '?', if data could not be computed and '*' if data in
the annotated row is incomplete. The shortcut $x_y$ means $x$ repeated $y$ times.
All computations have been done either using software written by the author
or using the computer algebra system GAP \cite{RefGAP}.

\section{The action of braids}\label{SectionTheActionOfBraids}

For a finite group $G$, a generating system $\underline{\sigma} = (\sigma_1, \dots, \sigma_m) \in G^m$ satisfying
the relation $\sigma_1 \cdots \sigma_m = \iota$ is called a generating $m$-system of $G$.
The set of all generating $m$-systems of $G$ is denoted by
$$\Sigma_m(G) = \{ (\sigma_1, \dots, \sigma_m) \in G^m \mid \langle \sigma_1, \dots, \sigma_m \rangle = G,
\sigma_1 \cdots \sigma_m = \iota \}.$$
The group $Inn(G)$ of inner automorphisms of $G$ acts on $\Sigma_m(G)$ by simultaneous conjugation. 
The set of classes $[\underline{\sigma}] = [\sigma_1, \dots, \sigma_m]$ is denoted by
$$\Sigma^i_m(G) = \Sigma_m(G)/Inn(G).$$
For non-trivial conjugacy classes $C_1,\dots, C_m$ of $G$, $(C_1,\dots, C_m)$ is called a 
{\it class vector} of length $m$ of $G$ and the following sets can be defined:
$$\Sigma(C_1,\dots, C_m) = \{ (\sigma_1, \dots, \sigma_m) \in \Sigma_m(G) \mid \sigma_i \in C_i, i = 1, \dots, m\}$$
and
$$\Sigma^i(C_1,\dots, C_m) = \Sigma(C_1,\dots, C_m)/Inn(G).$$
A conjugacy class $C$ of $G$ is called {\it rational} if the values of the complex irreducible characters
of $G$ on $C$ are rational, i.e. if  
$$\mathbb{Q}_C = \mathbb{Q}(\{ \chi(C) \mid \chi \in Irr(G)\}) = \mathbb{Q}$$
and a class vector is called {\it rational} if all of its conjugacy classes $C_1,\dots, C_m$ are rational.
Further, a class vector is called {\it symmetric} if $C_1 = C_2 =\dots = C_m$.
Finally, the number of elements in $\Sigma^i(C_1,\dots, C_m)$ is denoted by $l^i(C_1,\dots, C_m)$.
\par
\noindent
\smallskip
Let $H_m = \langle \beta_2, \beta_3, \dots, \beta_m \rangle$ with the relations
\begin{equation} 
\label{eq:BraidRelation1}
\beta_i \beta_j = \beta_j \beta_i, 2 \leq i < j \leq m, j-i \neq 1
\end{equation}
\begin{equation} 
\label{eq:BraidRelation2}
\beta_i \beta_{i+1} \beta_i = \beta_{i+1} \beta_i \beta_{i+1}, 2 \leq i \leq m-1
\end{equation}
\begin{equation} 
\label{eq:BraidRelation3}
\beta_2 \cdots \beta_{m-1} \beta_m^2 \beta_{m-1} \cdots \beta_2 = \iota
\end{equation}
be the {\it full Hurwitz braid group}.
Then $H_m$ acts on $\Sigma^i_m(G)$ by
\begin{equation} 
\label{eq:BraidAction}
[\underline{\sigma}]^{\beta_i} = 
[\sigma_1, \dots, \sigma_{i-2}, \sigma_{i-1} \sigma_i \sigma^{-1}_{i-1}, \sigma_{i-1}, \sigma_{i+1}, \dots, \sigma_m]
\end{equation}
for $i = 2, \dots, m$, see \cite[III, 1.2, Theorem 1.6 ff.]{RefMM2018} and the action of $\beta_i^{-1}$ is given by
\begin{equation} 
\label{eq:BraidActionInvers}
[\underline{\sigma}]^{\beta_i^{-1}} = 
[\sigma_1, \dots, \sigma_{i-2}, \sigma_i, \sigma_i^{-1} \sigma_{i-1} \sigma_i, \sigma_{i+1}, \dots, \sigma_m]
\end{equation}
for $i = 2, \dots, m$.
\par
\noindent
For a class vector $(C_1,\dots, C_m)$ and $\pi \in S_m$, define $(C_1,\dots, C_m)^{\pi} = (C_{\pi(1)},\dots, C_{\pi(m)})$ and
$$(C_1,\dots, C_m)^{sy} = \{(C_1,\dots, C_m)^{\pi} \mid \pi \in S_m \}.$$
Then $H_m$ already acts on the set
$$\Sigma^i(C_1,\dots, C_m)^{sy} = \bigcup_{\pi \in S_m} {\Sigma^i(C_1,\dots, C_m)^{\pi}} \subseteq \Sigma^i_m(G).$$
Define $\rho_m$ to be the induced permutation representation of $H_m$ on the set $\Sigma^i(C_1,\dots, C_m)^{sy}$.
\par
\noindent
\begin{proposition}
\label{EpsilonHmActsCyclic}
Let $m \geq 3$, then the action of $\epsilon_m = \beta_2 \cdots \beta_m \in H_m$ 
on $[\underline{\sigma}]\in \Sigma^i(C_1,\dots, C_m)^{sy}$ is given by
$$[\underline{\sigma}]^{\epsilon_m} = [\sigma_2,\sigma_3,\dots,\sigma_m,\sigma_1].$$
\end{proposition}
\begin{proof}
We have 
$$[\underline{\sigma}]^{\epsilon_m} = [\underline{\sigma}]^{\beta_2 \cdots \beta_m} =$$
$$[\sigma_1 \sigma_2 \sigma_1^{-1},\sigma_1,\sigma_3,\dots,\sigma_m]^{\beta_3 \cdots \beta_m} =$$
$$[\sigma_1 \sigma_2 \sigma_1^{-1},\sigma_1 \sigma_3 \sigma_1^{-1},\sigma_1,\sigma_4,\dots,\sigma_m]^{\beta_4 \cdots \beta_m} =$$
$$\dots$$
$$[\sigma_1 \sigma_2 \sigma_1^{-1},\sigma_1 \sigma_3 \sigma_1^{-1},\dots,\sigma_1 \sigma_m \sigma_1^{-1},\sigma_1] =$$
$$[\sigma_2,\sigma_3,\dots,\sigma_m,\sigma_1].$$
\end{proof}
\noindent
The kernel of the natural projection
$$\pi_m : H_m \rightarrow S_m, \pi_m(\beta_i) = (i-1, i), i = 2, \dots, m$$
is called the {\it pure Hurwitz braid group} and denoted by $B_m = kernel(\pi_m)$.
We have
$$B_m = \langle \beta_{ij} \mid 1 \leq i < j \leq m \rangle$$ with
$$\beta_{ij} = \beta_{i+1}^{-1} \cdots \beta_{j-1}^{-1} \beta_j^2 \beta_{j-1} \cdots \beta_{i+1}, 1 \leq i < j \leq m,$$
$B_m \vartriangleleft H_m$ and $H_m/B_m \cong S_m$.
The group $B_m$ already acts on $\Sigma^i(C_1,\dots, C_m)$. See also \cite[III]{RefMM2018}.
Again we denote the induced permutation representation of this action with $\rho_m$ as
it is the restriction of $\rho_m$ from above to $B_m$.
For later usage, we state
\begin{proposition}
\label{PropertiesOfPermutationRepresentations}
Let $\rho : H_m \rightarrow S_k$ be any permutation representation of degree $k$, then
\begin{itemize}
\item[(a)]
$\rho(B_m) \leq A_k$
\item[(b)]
$\rho(H_m)/\rho(B_m) \cong U_m$ with $U_m \leq S_m$
\end{itemize}
\end{proposition}
\begin{proof}
We have $sign(\rho(\beta_{ij})) = 1$ for $1 \leq i < j \leq m$, thus $\rho(\beta_{ij}) \in A_m$ and (a) follows.
Part (b) comes from $H_m/B_m \cong S_m$.
\end{proof}
\noindent
For $m \geq 3$, the center $Z(H_m)$ of $H_m$ has order $2$ and is generated by 
$$\Delta_m = \epsilon_m^m = (\beta_2 \cdots \beta_m)^m,$$
see \cite[III, Corollary 1.8]{RefMM2018}.
\begin{proposition}
\label{CenterHmActsTrivial}
Let $m \geq 3$, then $Z(H_m) = \langle \Delta_m \rangle$ acts trivial on $\Sigma^i(C_1,\dots, C_m)$.
\end{proposition}
\begin{proof}
We have $[\underline{\sigma}]^{\Delta_m} = [\underline{\sigma}]^{\epsilon_m^m} = [\underline{\sigma}]$
by Proposition~\ref{EpsilonHmActsCyclic}.
\end{proof}
\noindent
The next proposition allows us to deduce $\pi_m(\beta) \in S_m$ for $\beta \in H_m$ from its action
on $[\underline{\sigma}]$.
\begin{proposition}
\label{DeducePiFromAction}
Let $m \geq 3$, $\beta \in H_m$ and $[\underline{\sigma}] \in \Sigma^i(C_1,\dots, C_m)$.
Then we have
$$[\underline{\sigma}]^{\beta} \in \Sigma^i(C_1,\dots, C_m)^{\pi_m(\beta)^{-1}}$$
with the natural projection $\pi_m$.
\end{proposition}
\begin{proof}
Proposition~\ref{DeducePiFromAction} is true for 
$\beta \in \{ \beta_2,\beta_3,\dots,\beta_m,\beta_2^{-1},\beta_3^{-1},\dots,\beta_m^{-1}  \}$
by (\ref{eq:BraidAction}) and (\ref{eq:BraidActionInvers}).
For $i,j \in \{ 2,3,\dots,m \}$ and $[\underline{\sigma}] \in \Sigma^i(C_1,\dots, C_m)$ we have
\begin{equation}
\pi_m(\beta_i \beta_j) = (i-1,i)(j-1,j) =
\begin{cases}
(i-1,i)(j-1,j) & |i-j| > 1 \\
(j-1,j,j+1) & i = j+1 \\
(i-1,i+1,i) & j = i+1 \\
\iota & i = j
\end{cases}
\end{equation}
and
\begin{equation}
[\underline{\sigma}]^{\beta_i \beta_j} \in
\begin{cases}
\Sigma^i(C_1,\dots, C_m)^{(j-1,j)(i-1,i)} & |i-j| > 1 \\
\Sigma^i(C_1,\dots, C_m)^{(j-1,j+1,j)} & i = j+1 \\
\Sigma^i(C_1,\dots, C_m)^{(i-1,i,i+1)} & j = i+1 \\
\Sigma^i(C_1,\dots, C_m)^{\iota} & i = j
\end{cases}
\end{equation}
Thus Proposition~\ref{DeducePiFromAction} is true for $\beta = \beta_i \beta_j$ and $i,j \in \{ 2,3,\dots,m \}$
and we get
$$[\underline{\sigma}]^{\beta_i \beta_j} \in \Sigma^i(C_1,\dots, C_m)^{\pi_m(\beta_j)\pi_m(\beta_i)}.$$
\par
\noindent
In an analogous manner we can verify Proposition~\ref{DeducePiFromAction} for the elements
$\beta_i^{-1} \beta_j, \beta_i \beta_j^{-1}, \beta_i^{-1} \beta_j^{-1}, i,j \in \{ 2,3,\dots,m \}$.
Now the proof can be extended to
$\beta = \beta_{k_1}^{e_1} \cdots \beta_{k_l}^{e_1}$ with $e_i = \pm 1$
using $\pi_m(\beta_i) = \pi_m(\beta_i)^{-1}$ and
$[\underline{\sigma}]^{\beta_i}, [\underline{\sigma}]^{\beta_i^{-1}} \in \Sigma^i(C_1,\dots, C_m)^{(i-1,i)}$
for $i=2,\dots,m$.
\end{proof}
\begin{example}
\label{DeducePiFromActionForEpsilon}
From Proposition~\ref{EpsilonHmActsCyclic} follows
$$[\underline{\sigma}]^{\epsilon_m} \in \Sigma^i(C_2,\dots, C_m,C_1) = 
\Sigma^i(C_1,\dots, C_m)^{(1,2,\dots,m)}$$
and by direct computation we get
$$\pi_m(\epsilon_m)^{-1} = ((1,2)(2,3)\dots(m-1,m))^{-1} = $$
$$(m-1,m)\dots(2,3)(1,2) = (1,2,\dots,m)$$
which confirms Proposition~\ref{DeducePiFromAction}.
\end{example}

\section{The action of braids in dimension 4}\label{SectionTheActionOfBraidsInDimension4}

For $m=4$, the full Hurwitz braid group
$$H_4 = \langle \beta_2, \beta_3, \beta_4 \mid \beta_2 \beta_4 = \beta_4 \beta_2, \beta_2 \beta_3 \beta_2 = 
\beta_3 \beta_2 \beta_3,$$
$$\beta_3 \beta_4 \beta_3 = \beta_4 \beta_3 \beta_4, \beta_2 \beta_3 {\beta_4}^2 \beta_3 \beta_2 = \iota \rangle$$
acts on $\Sigma^i(C_1,C_2,C_3,C_4)^{sy}$ by
$$[\underline{\sigma}]^{\beta_2} = [\sigma_1 \sigma_2 \sigma_1^{-1}, \sigma_1, \sigma_3, \sigma_4],$$
$$[\underline{\sigma}]^{\beta_3} = [\sigma_1, \sigma_2 \sigma_3 \sigma_2^{-1}, \sigma_2, \sigma_4],$$
$$[\underline{\sigma}]^{\beta_4} = [\sigma_1, \sigma_2, \sigma_3 \sigma_4 \sigma_3^{-1}, \sigma_3].$$
For the pure Hurwitz group 
$$B_4=\langle \beta_{12}, \beta_{13}, \beta_{14}, \beta_{23}, \beta_{24}, \beta_{34} \rangle$$
with the elements
$\beta_{12}=\beta_2^2$,
$\beta_{13}=\beta_2^{-1} \beta_3^2 \beta_2$,
$\beta_{14}=\beta_2^{-1} \beta_3^{-1} \beta_4^2 \beta_3 \beta_2$,
$\beta_{23}=\beta_3^2$,
$\beta_{24}=\beta_3^{-1} \beta_4^2 \beta_3$ and
$\beta_{34}=\beta_4^2$, we get
\begin{proposition}
\label{ActionOfPureHurwitzGroup}
$B_4$ acts on $\Sigma^i(C_1,C_2,C_3,C_4)$ by
$$[\underline{\sigma}]^{\beta^k_{12}} =
[(\sigma_3 \sigma_4)^{-k} \sigma_1 (\sigma_3 \sigma_4)^k, (\sigma_3 \sigma_4)^{-k} \sigma_2 (\sigma_3 \sigma_4)^k,
\sigma_3, \sigma_4],$$
$$[\underline{\sigma}]^{\beta^k_{13}} =
[(\sigma_2 \sigma_4)^{-k} \sigma_1 (\sigma_2 \sigma_4)^k, \sigma_2,
(\sigma_4 \sigma_2)^{-k} \sigma_3 (\sigma_4 \sigma_2)^k, \sigma_4],$$
$$[\underline{\sigma}]^{\beta^k_{14}} =
[(\sigma_2 \sigma_3)^{-k} \sigma_1 (\sigma_2 \sigma_3)^k, \sigma_2, \sigma_3,
(\sigma_2 \sigma_3)^{-k} \sigma_4 (\sigma_2 \sigma_3)^k],$$
$$
[\underline{\sigma}]^{\beta_{23}} = [\underline{\sigma}]^{\beta_{14}^{-1}},
[\underline{\sigma}]^{\beta_{24}} = [\underline{\sigma}]^{\beta_{14}^{-1} \beta_{12}^{-1}},
[\underline{\sigma}]^{\beta_{34}} = [\underline{\sigma}]^{\beta_{12}},
[\underline{\sigma}]^{\beta_{12} \beta_{13} \beta_{14}} = [\underline{\sigma}]
$$
for $k \in \mathbb{N}$.
\end{proposition}
\noindent
Proposition~\ref{ActionOfPureHurwitzGroup} can be verified by direct computation and induction over $k$.
\begin{corollary}
\label{ActionOfPureHurwitzGroupByBeta1j}
The action of $B_4$ on $\Sigma^i(C_1,C_2,C_3,C_4)$ is already determined by the action
of $\beta_{1j}$ for $j=2,3,4$ resp. $j=2,3$.
\end{corollary}
\noindent
For the induced permutation representation $\rho_4$ of $B_4$ on $\Sigma^i(C_1,C_2,C_3,C_4)$ we obtain the following corollaries.
\begin{corollary}
\label{ActionOfPureHurwitzGroupByBeta1jCase4}
$\rho_4(B_4) = \rho_4(\langle \beta_{12}, \beta_{13}, \beta_{14} \rangle) = \rho_4(\langle \beta_{12}, \beta_{13} \rangle)$.
\end{corollary}
\begin{corollary}
\label{PropertiesOfPermutationRepresentationsCase4}
If $\rho_4(\beta_{1j}), j = 2,3,4$ are involutions, then $\rho_4(B_4) \cong C_2 \times C_2$.
\end{corollary}
\begin{proof}
From $\rho_4(\beta_{12} \beta_{13} \beta_{14}) = \iota$ we get
$\rho_4(\beta_{12}) \rho_4(\beta_{13}) = \rho_4(\beta_{14})$ and by taking the inverse also
$\rho_4(\beta_{13}) \rho_4(\beta_{12}) = \rho_4(\beta_{14})$,
thus $\rho_4(\beta_{12}) \rho_4(\beta_{13}) = \rho_4(\beta_{13}) \rho_4(\beta_{12})$ and finally
$$\rho_4(B_4) = \langle \rho_4(\beta_{12}), \rho_4(\beta_{13}) \rangle =
\langle \rho(\beta_{12}) \rangle \times \langle \rho(\beta_{13}) \rangle \cong C_2 \times C_2.$$
\end{proof}
\begin{corollary}
\label{ActionOfPureHurwitzGroupByConjugation}
The length of the cycles of $\rho_4(\beta_{1j})$,
$j=2,3,4$ is bounded above by the orders $o(\sigma_3 \sigma_4)$, $o(\sigma_2 \sigma_4)=o(\sigma_4 \sigma_2)$
resp. $o(\sigma_2 \sigma_3)$ and $B_4$ acts by conjugation on $\Sigma^i(C_1,C_2,C_3,C_4)$.
\end{corollary}
\noindent
Define $$S(C_1,C_2,C_3,C_4) = \{\pi \in S_4 \mid (C_{\pi(1)},C_{\pi(2)},C_{\pi(3)},C_{\pi(4)}) = (C_1,C_2,C_3,C_4) \}$$
to be the group of symmetries of the class vector $(C_1,C_2,C_3,C_4)$.
\par
\noindent
If $S(C_1,C_2,C_3,C_4) > I$, so-called {\it topological automorphisms} 
(see \cite[5.]{RefMat1989} and \cite[1.]{RefHae1991})
$\eta_{23}$, $\eta_{34}$ and $\eta_{234}$ act on
$\Sigma^i(C_1,C_2,C_3,C_4)$ by
$$[\underline{\sigma}]^{\eta_{23}} = 
[\sigma_2^{-1} \sigma_1 \sigma_2,\sigma_3,\sigma_4 \sigma_2 \sigma_4^{-1},\sigma_4],$$
$$[\underline{\sigma}]^{\eta_{34}} = 
[\sigma_4 \sigma_1 \sigma_4^{-1},\sigma_2,\sigma_2^{-1} \sigma_4 \sigma_2,\sigma_3],$$
$$[\underline{\sigma}]^{\eta_{234}} = 
[\sigma_1,\sigma_2 \sigma_3 \sigma_2^{-1}, \sigma_2 \sigma_4 \sigma_2^{-1},\sigma_2]$$
for $S(C_1,C_2,C_3,C_4) \geq S$ with $S = \langle (2,3) \rangle$, $S = \langle (3,4) \rangle$
resp. $S = \langle (2,3,4) \rangle$.
We have
$$[\underline{\sigma}]^{\eta^2_{23}} = [\underline{\sigma}]^{\eta^2_{34}} = [\underline{\sigma}]^{\eta^3_{234}} = [\underline{\sigma}],
[\underline{\sigma}]^{\eta_{34}} = [\underline{\sigma}]^{\eta_{234} \eta_{23} \eta^{-1}_{234}},$$
$$[\underline{\sigma}]^{\beta_{13}} = [\underline{\sigma}]^{\eta_{23} \beta_{12} \eta^{-1}_{23}},
[\underline{\sigma}]^{\beta_{14}} = [\underline{\sigma}]^{\eta_{34} \beta_{13} \eta^{-1}_{34}},
[\underline{\sigma}]^{\beta_4} = [\underline{\sigma}]^{\eta_{234} \beta_3 \eta^{-1}_{234}}$$
and
$$[\underline{\sigma}]^{\beta_{13} \beta_3 \eta_{23}} =
[\underline{\sigma}]^{\beta_{14} \beta_4 \eta_{34}} =
[\underline{\sigma}]^{\beta_4 \eta_{34} \eta_{234}} = [\underline{\sigma}].$$
\par
\noindent
For an orbit $Z \subseteq \Sigma^i(C_1,C_2,C_3,C_4)$ under the action of $B_4$, let $z_{1j}$, $j=2,3,4$ be the number
of cycles of $\rho_4(\beta_{1j})$ on $Z$ and define
$$g_Z = 1 -  \vert Z \vert  + \frac{1}{2}(3\cdot  \vert Z \vert  - z_{12} - z_{13} - z_{14}).$$
For an orbit $Z^{sy} \subseteq \Sigma^i(C_1,C_2,C_3,C_4)^{sy}$ under the action of $B^{sy}_4$ with
$B^{sy}_4 = \langle B_4, \beta_3, \eta_{23} \rangle$, $B^{sy}_4 = \langle B_4, \beta_4, \eta_{34} \rangle$ or
$B^{sy}_4 = \langle B_4, \beta_3, \eta_{23}, \eta_{234} \rangle$,
let $z_j$, $j=2,3,4$ be the number of cycles of $\rho_4(\beta_{j})$ on $Z^{sy}$ and $z_{23}$, $z_{34}$,
$z_{234}$ be the number of cycles of $\rho_4(\eta_{23})$, $\rho_4(\eta_{34})$, $\rho_4(\eta_{234})$ on $Z^{sy}$
and define
$$g_{Z^{sy}} = 1 -  \vert Z^{sy} \vert  + \frac{1}{2}(3\cdot  \vert Z^{sy} \vert  - z_{12} - z_{3} - z_{23}),$$
$$g_{Z^{sy}} = 1 -  \vert Z^{sy} \vert  + \frac{1}{2}(3\cdot  \vert Z^{sy} \vert  - z_{13} - z_{4} - z_{34}),$$
$$g_{Z^{sy}} = 1 -  \vert Z^{sy} \vert  + \frac{1}{2}(3\cdot  \vert Z^{sy} \vert  - z_{3} - z_{23} - z_{234})$$
for $S(C_1,C_2,C_3,C_4) \geq S$ with $S = \langle (2,3) \rangle$, $S = \langle (3,4) \rangle$
respectively $S = \langle (2,3), (2,3,4) \rangle$.
Note that we have $z_{12}=z_{13}$, $z_{13}=z_{14}$, $z_3=z_4$ and $z_{23} = z_{34}$ for $Z^{sy}$ with the corresponding symmetry.
\par
\noindent
Call $g_Z$ the genus of an orbit $Z \subseteq \Sigma^i(C_1,C_2,C_3,C_4)$ under the action of $B_4$ and $g_{Z^{sy}}$
the genus of an orbit $Z^{sy} \subseteq \Sigma^i(C_1,C_2,C_3,C_4)^{sy}$ under the action of $B^{sy}_4$.
See also \cite[Definition 5.3]{RefMat1991}, \cite[Satz 7.2]{RefMat1991},
\cite[III, 5.2]{RefMM2018} and \cite[III, Theorem 7.8]{RefMM2018}. 
Note that in \cite{RefMat1989} and here, the first class is fixed, 
while the last class is fixed in \cite{RefMat1991} and \cite{RefMM2018}.
\par
\noindent
In inverse Galois theory, rigid braid orbits with $g_Z=0$ or $g_{Z^{sy}} = 0$ in rational class vectors are sufficient to
realize groups $G$ with trivial center $Z(G)$ as Galois groups over $\mathbb{Q}(t)$ and $\mathbb{Q}$. 
An orbit is for example rigid, if it is unique by its size. 
For details, see again \cite[III, 5]{RefMM2018} and \cite[III, 7]{RefMM2018}.
\par
\noindent
Examples show, that the conditions $g_Z=0$ or $g_{Z^{sy}} = 0$ in most cases are satisfied only for
small orbits, i.e. orbits with $ \vert Z \vert  \ll l^i(C_1,\dots, C_m)$ resp. $ \vert Z^{sy} \vert  \ll l^i(C_1,\dots, C_m)$.
Thus information about the splitting of $\Sigma^i(C_1,C_2,C_3,C_4)$ into (small) orbits under the action of $B_4$
is needed. This motivates to look for invariants under this action.

\section{Invariants under the action of braids}\label{SectionInvariantsUnderTheActionOfBraids}

Let $i_{B_m}: \Sigma^i(C_1,\dots, C_m) \rightarrow \mathbb{Z}$ be an invariant under the action of $B_m$, i.e.
$$i_{B_m}([\underline{\sigma}]^{\beta_{ij}}) = i_{B_m}([\underline{\sigma}])$$
for $1 \leq i < j \leq m$ and all $[\underline{\sigma}] \in \Sigma^i(C_1,\dots, C_m)$.
Then we have
$$\Sigma^i(C_1,\dots, C_m) = \bigcup_{k \in \mathbb{Z}} \Sigma^i((C_1,\dots, C_m),i_{B_m},k)$$
where each 
$$\Sigma^i((C_1,\dots, C_m),i_{B_m},k) = 
\{ [\underline{\sigma}] \in \Sigma^i(C_1,\dots, C_m) \mid i_{B_m}([\underline{\sigma}]) = k \}$$
is a union of $B_m$-orbits with the same value of the invariant $i_{B_m}(.)$. A trivial invariant is the length of a braid orbit.
An other invariant can be constructed as follows:
\begin{proposition}
\label{BmActionOnFixedPointsF}
Let $F \leq H_m$ be a group acting on $\Sigma^i(C_1,\dots, C_m)^{sy}$ with
\begin{equation} 
\label{eq:InvariantCondition}
[\underline{\sigma}]^{\beta_{ij} \varphi} = [\underline{\sigma}]^{\varphi \beta_{ij}}
\end{equation}
for $1 \leq i < j \leq m$, all $\varphi \in F$ and all $[\underline{\sigma}] \in \Sigma^i(C_1,\dots, C_m)^{sy}$.
Then $B_m$ acts on the fixed points of $F$ in $\Sigma^i(C_1,\dots, C_m)$.
\end{proposition}
\begin{proof}
For a fixed point $[\underline{\sigma}] \in \Sigma^i(C_1,\dots, C_m)$ of $F$, we have
$[\underline{\sigma}]^{\varphi} = [\underline{\sigma}]$ and
$$[\underline{\sigma}]^{\beta_{ij}} = [\underline{\sigma}]^{\varphi \beta_{ij}} = [\underline{\sigma}]^{\beta_{ij} \varphi},$$
thus $[\underline{\sigma}]^{\beta_{ij}}$ is a fixed point of $F$ in $\Sigma^i(C_1,\dots, C_m)$ because
$B_m$ acts on the set $\Sigma^i(C_1,\dots, C_m)$.
\end{proof}
\noindent
Defining $F_{B_m}(F,[\underline{\sigma}])=1$ for fixed points $[\underline{\sigma}]$ under $F$
and $F_{B_m}(F,[\underline{\sigma}])=-1$ otherwise gives an invariant and we get
$$
\Sigma^i((C_1,\dots, C_m),F_{B_m}(F, .), 1) = 
\{ [\underline{\sigma}] \in \Sigma^i(C_1,\dots, C_m) \mid \forall \varphi \in F: [\underline{\sigma}]^{\varphi} = [\underline{\sigma}] \}
$$
as union of $B_m$-orbits. Summarized, the set of fixed points of $F$ is a union of $B_m$ orbits.
\par
\noindent
In the next step, we define a group $F_m \leq H_m$ satisfying (\ref{eq:InvariantCondition}) from Proposition~\ref{BmActionOnFixedPointsF}.
Let $m=2n$, then by \cite[Theorem (ii)]{RefGon2007} and its proof (with adapting the notation 
from \cite{RefGon2007}), the subgroup $\langle x_m, y_m \rangle \leq H_m$ generated by
$$x_m = (\beta_2 \cdots \beta_{2n})(\beta_2 \cdots \beta_{2n-1}) \cdots (\beta_2 \beta_3) \beta_2,$$
and
$$y_m = (\beta_2 \cdots \beta_{n})(\beta_2 \cdots \beta_{n-1}) \cdots (\beta_2 \beta_3) \beta_2 \cdot$$
$$\beta^{-1}_{2n} (\beta^{-1}_{2n-1} \beta^{-1}_{2n}) \cdots (\beta^{-1}_{n+3} \cdots \beta^{-1}_{2n})
(\beta^{-1}_{n+2} \cdots \beta^{-1}_{2n})$$
is isomorphic to a quaternion group $Q_8$ of order $8$.
\par
\noindent
Let $\varphi_{m,1} = y_m$, $\varphi_{m,2} = x_m^{-1} y_m$, $\varphi_{m,3} = \varphi_{m,1} \varphi_{m,2}$,
$F_{m,k} = \langle \varphi_{m,k} \rangle$, $k=1,2,3$
and set $F_m = \langle \varphi_{m,1}, \varphi_{m,2} \rangle \leq H_m$, then we have
\begin{proposition}
\label{PropertiesOfFm}
Let $2 \leq n$ and $m=2n$, then
\begin{itemize}
\item[(a)]
$F_m \cong Q_8$
\item[(b)]
$\varphi_{m,2} = (\beta_2 \cdots \beta_m)^n$
\item[(c.1)]
The action of $\varphi_{m,1}$ on $[\underline{\sigma}]\in \Sigma^i(C_1,\dots, C_m)^{sy}$ is given by
$$[\underline{\sigma}]^{\varphi_{m,1}} =
[\sigma_n,\sigma_{n-1}^{\sigma_n},\sigma_{n-2}^{\sigma_{n-1} \sigma_n},\dots,
\sigma_1^{\sigma_2 \cdots \sigma_n},$$
$$\sigma_{2n}^{\tau},
\sigma_{2n-1}^{\sigma_{2n} \tau},
\sigma_{2n-2}^{\sigma_{2n-1}\sigma_{2n} \tau},\dots,
\sigma_{n+1}^{\sigma_{n+2} \cdots \sigma_{2n} \tau}]$$
with $\tau = \sigma_1 \cdots \sigma_n$ at least for $2 \leq n \leq 5$.
\item[(c.2)]
The action of $\varphi_{m,2}$ on $[\underline{\sigma}]\in \Sigma^i(C_1,\dots, C_m)^{sy}$ is given by
$$[\underline{\sigma}]^{\varphi_{m,2}} = [\sigma_{n+1},\dots,\sigma_{2n},\sigma_1,\dots,\sigma_n]$$
\item[(c.3)]
The action of $\varphi_{m,3}$ on $[\underline{\sigma}]\in \Sigma^i(C_1,\dots, C_m)^{sy}$ is given by
$$[\underline{\sigma}]^{\varphi_{m,3}} =
[\sigma_{2n},\sigma_{2n-1}^{\sigma_{2n}},\sigma_{2n-2}^{\sigma_{2n-1} \sigma_{2n}},\dots,
\sigma_{n+1}^{\sigma_{n+2}\cdots \sigma_{2n}},$$
$$\sigma_n^{\sigma_{n-1}^{-1} \cdots \sigma_1^{-1}},
\sigma_{n-1}^{\sigma_{n-2}^{-1} \cdots \sigma_1^{-1}},\dots,
\sigma_2^{\sigma_1^{-1}}, \sigma_1]$$
at least for $2 \leq n \leq 5$.
\item[(d)]
The natural projection $\pi_m(\varphi_{m,k})$, $k=1,2,3$ is given in table \ref{tab:Data_NaturalProjectionOfPhimk}
at least for $2 \leq n \leq 5$.
\item[(e)]
$[\underline{\sigma}]^{\varphi^2_{m,1}} = [\underline{\sigma}]^{\varphi^2_{m,2}} = [\underline{\sigma}]^{\varphi^2_{m,3}} = [\underline{\sigma}]$
\item[(f)]
$[\underline{\sigma}]^{\varphi_{m,1} \varphi_{m,2}} = [\underline{\sigma}]^{\varphi_{m,2} \varphi_{m,1}}$
\item[(g)]
$F_4 \vartriangleleft H_4$ and 
$F_4$ satisfies (\ref{eq:InvariantCondition}) from Proposition~\ref{BmActionOnFixedPointsF}.
\end{itemize}
\end{proposition}
\begin{table}[!htbp]
\centering
\footnotesize
\captionsetup{font=footnotesize}
\caption{The natural projection $\pi_m(\varphi_{m,k})$ for $k=1,2,3$ and $2 \leq n \leq 5$}
\label{tab:Data_NaturalProjectionOfPhimk}
\begin{tabular}{ll}
\hline\noalign{\smallskip}
$k$ & $\pi_m(\varphi_{m,k})$ \\
\noalign{\smallskip}\hline\noalign{\smallskip}
1 & $(1,n)(2,n-1)\cdots (\frac{n}{2},\frac{n+2}{2})(n+1,2n)(n+2,2n-1)\cdots (\frac{3n}{2},\frac{3n+2}{2})$ \\
  & $n \equiv 0\: mod\: 2$ \\
\noalign{\smallskip}\hline\noalign{\smallskip}
1 & $(1,n)(2,n-1)\cdots (\frac{n-1}{2},\frac{n+3}{2})(n+1,2n)(n+2,2n-1)\cdots (\frac{3n-1}{2},\frac{3n+3}{2})$ \\
  & $n \equiv 1\: mod\: 2$ \\
\noalign{\smallskip}\hline\noalign{\smallskip}
2 & $(1,n+1)(2,n+2) \cdots (n,2n)$ \\
\noalign{\smallskip}\hline\noalign{\smallskip}
3 & $(1,2n)(2,2n-1) \cdots (n,n+1)$ \\
\noalign{\smallskip}\hline
\end{tabular}
\end{table}
\begin{proof}
(a) We have $F_m = \langle y_m, x_m^{-1} y_m \rangle = \langle x_m, y_m \rangle \cong Q_8$ by \cite[Theorem (ii)]{RefGon2007}.
\par
\noindent
(b) With
$u_i = \beta_2 \beta_3 \cdots \beta_i, i=2,3,\dots, 2n$,
$v_i = \beta_i^{-1} \beta_{i+1}^{-1} \cdots \beta_{2n}^{-1}, i=2,3,\dots, 2n$
and
$w_i = \beta_i^{-1} \beta_{i-1}^{-1} \beta_{i-2}^{-1} \cdots \beta_{i-n+1}^{-1}, i=n+2,\dots, 2n$
we can write
$$\varphi_{2n,2} =x_{2n}^{-1} y_{2n} = u_2^{-1} u_3^{-1} \cdots u_{2n-1}^{-1} (u_{2n}^{-1} u_n) 
u_{n-1} \cdots u_2 \cdot
v_{2n} v_{2n-1} \cdots v_{n+2}.$$
Using $u_{2n}^{-1} u_n = w_{2n}$ and the braid relations (\ref{eq:BraidRelation1}), we get
$$\varphi_{2n,2} = u_2^{-1} u_3^{-1} \cdots u_{2n-2}^{-1} (u_{2n-1}^{-1} u_{n-1}) 
u_{n-2} \cdots u_2 \cdot
w_{2n} \cdot v_{2n} v_{2n-1} \cdots v_{n+2}.$$
This can be repeated until we reach
$$\varphi_{2n,2} = u_2^{-1} u_3^{-1} \cdots u_{n+1}^{-1} \cdot
w_{n+2} w_{n+3} \cdots w_{2n} \cdot
v_{2n} v_{2n-1} \cdots v_{n+2}.$$
Again using (\ref{eq:BraidRelation1}), we get
$$w_{2n} \cdot v_{2n} v_{2n-1} \cdots v_{n+2} = v_{2n} v_{2n-1} \cdots v_{n+1},$$
$$w_{2n-1} \cdot v_{2n} v_{2n-1} \cdots v_{n+1} = v_{2n-1} v_{2n-2} \cdots v_{n}, ...$$
and reach
$$\varphi_{2n,2} = u_2^{-1} u_3^{-1} \cdots u_{n+1}^{-1} \cdot v_{n+2} v_{n+1} \cdots v_3.$$
We have
$$u_{n+1}^{-1} \cdot v_{n+2} v_{n+1} \cdots v_3 = v_{n+1} v_n \cdots v_ 3 \cdot v_2,$$
$$u_n^{-1} \cdot v_{n+1} v_n \cdots v_3 = v_{n} v_{n-1} \cdots v_3 \cdot v_2, ...$$
and finally get
$$\varphi_{2n,2} = v_2^n = (\beta_2^{-1} \beta_3^{-1} \cdots \beta_{2n}^{-1})^n.$$
From the third braid relation (\ref{eq:BraidRelation3}) follows
$\beta_2 \beta_3 \cdots \beta_{2n} = \beta_2^{-1} \beta_3^{-1} \cdots \beta_{2n}^{-1}$
and (b) is proved.
\par
\noindent
Parts (c.1) and (c.3) have been verified with the help of a computer for $2 \leq n \leq 5$.
\par
\noindent
(c.2) From (b) we get $\varphi_{m,2} = \epsilon_m^n$. Now (c.2) follows from Proposition~\ref{EpsilonHmActsCyclic}.
\par
\noindent
(d) Part (d) follows from (c.1), (c.2), (c.3) and Proposition~\ref{DeducePiFromAction}.
\par
\noindent
(e) From the proof of \cite[Theorem (ii)]{RefGon2007}, we get the equations
\begin{equation} 
\label{eq:xm2EqualsDeltam}
x_m^2 = \Delta_m
\end{equation}
\begin{equation} 
\label{eq:ym2EqualsDeltam}
y_m^2 = \Delta_m
\end{equation}
\begin{equation} 
\label{eq:xmymxminvEqualsyminv}
x_m y_m x_m^{-1} = y_m^{-1}
\end{equation}
and from (\ref{eq:xmymxminvEqualsyminv}) follows
\begin{equation} 
\label{eq:ymEqualsxminvyminvxm}
 y_m = x_m^{-1} y_m^{-1} x_m
\end{equation}
By definition and (\ref{eq:ym2EqualsDeltam}), we have $\varphi^2_{m,1} = y_m^2 = \Delta_m$.
Using (\ref{eq:xmymxminvEqualsyminv}) and (\ref{eq:xm2EqualsDeltam}) we get
$$\varphi^2_{m,2} = x_m^{-1} y_m x_m^{-1} y_m =  x_m^{-2} (x_m y_m x_m^{-1}) y_m = x_m^{-2} = \Delta_m^{-1} = \Delta_m$$
because $\Delta_m$ is an involution. Again by definition and (\ref{eq:xmymxminvEqualsyminv}), we get
$$\varphi_{m,3} = \varphi_{m,1} \varphi_{m,2} =  y_m x_m^{-1} y_m = x_m^{-1} (x_m y_m x_m^{-1}) y_m =
x_m^{-1} y_m^{-1} y_m = x_m^{-1}$$
and finally $\varphi^2_{m,3} = \Delta_m$. Now (e) follows from Proposition~\ref{CenterHmActsTrivial}.
\par
\noindent
(f) By definition, (\ref{eq:ymEqualsxminvyminvxm}) and (\ref{eq:ym2EqualsDeltam}) we get
$$\varphi_{m,1} \varphi_{m,2} \varphi^{-1}_{m,1} \varphi^{-1}_{m,2} =
(y_m) (x_m^{-1} y_m) (y_m^{-1}) (y_m^{-1} x_m) = 
y_m (x_m^{-1} y_m^{-1} x_m) = y_m^2 = \Delta_m.$$
Now (f) follows from 
$$[\underline{\sigma}] = [\underline{\sigma}]^{\Delta_m} =
[\underline{\sigma}]^{\varphi_{m,1} \varphi_{m,2} \varphi^{-1}_{m,1} \varphi^{-1}_{m,2}}$$
by applying the action of $\varphi_{m,2} \varphi_{m,1}$.
\par
\noindent
(g) We have $F_4 \vartriangleleft H_4$ by \cite[III, Proposition 1.9]{RefMM2018}.
The second statement can be verified by direct computation.
\end{proof}
\noindent
For $m=4$, we get $\varphi_{4,1} = \beta_2 \beta_4^{-1}$ and $\varphi_{4,2} = (\beta_2 \beta_3 \beta_4)^2$.
The action of $\varphi_{4,k}$, $k=1,2,3$ on $\Sigma^i(C_1,C_2,C_3,C_4)^{sy}$ is given by
$$[\underline{\sigma}]^{\varphi_{4,1}} = [\sigma_2,\sigma_1,\sigma^{-1}_1 \sigma_4 \sigma_1,
\sigma^{-1}_1 \sigma^{-1}_4 \sigma_3 \sigma_4 \sigma_1],$$
$$[\underline{\sigma}]^{\varphi_{4,2}} = [\sigma_3,\sigma_4,\sigma_1,\sigma_2],$$
$$[\underline{\sigma}]^{\varphi_{4,3}} = [\sigma_4, \sigma^{-1}_4 \sigma_3 \sigma_4,
\sigma_1 \sigma_2 \sigma^{-1}_1,\sigma_1].$$
\par
\noindent
For $m=6$, we get $\varphi_{6,1} = \beta_2 \beta_3 \beta_2 \beta_6^{-1} \beta_5^{-1} \beta_6^{-1}$
and $\varphi_{6,2} = (\beta_2 \beta_3 \beta_4 \beta_5 \beta_6)^3$.
The action of $\varphi_{6,k}$, $k=1,2,3$ on $\Sigma^i(C_1,C_2,C_3,C_4,C_5,C_6)^{sy}$ is given by
$$[\underline{\sigma}]^{\varphi_{6,1}} = [\sigma_3, \sigma_2, \sigma^{-1}_2 \sigma_1 \sigma_2,
\sigma^{-1}_2 \sigma^{-1}_1 \sigma_6 \sigma_1 \sigma_2,$$
$$\sigma^{-1}_2 \sigma^{-1}_1 \sigma^{-1}_6 \sigma_5 \sigma_6 \sigma_1 \sigma_2, 
\sigma^{-1}_2 \sigma^{-1}_1 \sigma^{-1}_6 \sigma^{-1}_5 \sigma_4 \sigma_5 \sigma_6 \sigma_1 \sigma_2],$$
$$[\underline{\sigma}]^{\varphi_{6,2}} = [\sigma_4,\sigma_5,\sigma_6,\sigma_1,\sigma_2,\sigma_3],$$
$$[\underline{\sigma}]^{\varphi_{6,3}} = [\sigma_6, \sigma^{-1}_6 \sigma_5 \sigma_6,
\sigma^{-1}_6 \sigma^{-1}_5 \sigma_4 \sigma_5 \sigma_6,
\sigma_1 \sigma_2 \sigma_3 \sigma^{-1}_2 \sigma^{-1}_1, \sigma_1 \sigma_2 \sigma^{-1}_1, \sigma_1].$$
\begin{theorem}
\label{BraidInvariants}
$F_{B_4}(F_{4,k},.), k=1,2,3$ are invariants for the action of $B_4$ on $\Sigma^i(C_1,C_2,C_3,C_4)$.
\end{theorem}
\begin{proof}
This comes from Proposition~\ref{PropertiesOfFm}(g) and Proposition~\ref{BmActionOnFixedPointsF}.
\end{proof}
\begin{example}
\label{M11_11A11A11A11A}
Table \ref{tab:Data_M11_11A11A11A11A} contains the computation of the invariants from
Theorem~\ref{BraidInvariants} for the class vector $(11A,11A,11A,11A)$ of the Mathieu group $M_{11}$.
$\Sigma^i(11A,11A,11A,11A)$ splits under the action of $B_4$ into braid orbits of length
$2, 2, 2, 33, 198, 198, 198, 864$ and $2,996$.
\begin{table}[!htbp]
\centering
\footnotesize
\captionsetup{font=footnotesize}
\caption{Invariants for the class vector $(11A,11A,11A,11A)$ of $M_{11}$}
\label{tab:Data_M11_11A11A11A11A}
\begin{tabular}{lrrr}
\hline\noalign{\smallskip}
$size$ & $F_{B_4}(F_{4,1},.)$ & $F_{B_4}(F_{4,2},.)$ & $F_{B_4}(F_{4,3},.)$ \\
\noalign{\smallskip}\hline\noalign{\smallskip}
2 & 1 & -1 & -1 \\
2 & -1 & 1 & -1 \\
2 & -1 & -1 & 1 \\
33 & 1 & 1 & 1 \\
198 & 1 & -1 & -1 \\
198 & -1 & 1 & -1 \\
198 & -1 & -1 & 1 \\
864 & -1 & -1 & -1 \\
2,996 & -1 & -1 & -1 \\
\noalign{\smallskip}\hline
\end{tabular}
\end{table}
\end{example}

\section{Fixed points and orbits of length 2}\label{SectionFixedPointsAndOrbitsOfSize2}

From section \ref{SectionInvariantsUnderTheActionOfBraids} we know, that $B_4$ acts on fixed points of $F_{4,k}$, $k=1,2,3$.
Such fixed points can exist only if the class vector $(C_1,C_2,C_3,C_4)$ has suitable symmetries. 
For example, $\varphi_{4,1}$ acts on $\Sigma^i(C_1,C_2,C_3,C_4)^{sy}$ by permuting the classes $C_1,C_2$ and $C_3,C_4$.
Thus fixed points under the action of $\varphi_{4,1}$ can exist only if $C_1 = C_2$ and $C_3 = C_4$.
Now let us study fixed points similar to \cite[III, §2, Satz 5(b)]{RefMat1987}. This leads to the construction
of orbits of length $2$ under the action of $B_4$.
\par
\noindent
Here and in the following sections, let $G$ be a finite group with $Z(G)=I$.
If $(C,C,D,D)$ is a class vector of $G$ and $[\underline{\sigma}]=[\sigma_1,\sigma_2,\sigma_3,\sigma_4] \in \Sigma^i(C,C,D,D)$
is a fixed point under $F_{4,1}$, then we have
$$[\sigma_1,\sigma_2,\sigma_3,\sigma_4]^{\varphi_{4,1}} = 
[\sigma_2,\sigma_1,\sigma^{-1}_1 \sigma_4 \sigma_1, \sigma_2 \sigma_3 \sigma^{-1}_2].$$
Thus there exists an element $\tau \in G$ with
\begin{equation} 
\label{eq:ConjugationTau_1}
\tau^{-1} \sigma_1 \tau = \sigma_2
\end{equation}
\begin{equation}
\label{eq:ConjugationTau_2}
\tau^{-1} \sigma_2 \tau = \sigma_1
\end{equation}
\begin{equation}
\label{eq:ConjugationTau_3}
\tau^{-1} \sigma_3 \tau = \sigma^{-1}_1 \sigma_4 \sigma_1
\end{equation}
\begin{equation}
\label{eq:ConjugationTau_4}
\tau^{-1} \sigma_4 \tau = \sigma_2 \sigma_3 \sigma^{-1}_2
\end{equation}
From (\ref{eq:ConjugationTau_1}) and (\ref{eq:ConjugationTau_2}) follow $\tau^{-2} \sigma_1 \tau^2 = \sigma_1$
and $\tau^{-2} \sigma_2 \tau^2 = \sigma_2$.
Using (\ref{eq:ConjugationTau_3}) and (\ref{eq:ConjugationTau_4}), we get
$$\tau^{-2} \sigma_3 \tau^2 = \tau^{-1} \sigma^{-1}_1 \sigma_4 \sigma_1 \tau = 
(\tau^{-1} \sigma^{-1}_1 \tau)\tau^{-1} \sigma_4 \tau (\tau^{-1} \sigma_1 \tau) =$$
$$\sigma^{-1}_2 \sigma_2 \sigma_3 \sigma^{-1}_2 \sigma_2 = \sigma_3.$$
Due to $G=\langle \sigma_1, \sigma_2, \sigma_3 \rangle$ and $Z(G)=I$, we get $\tau^2 = \iota$,
thus $o(\tau)=1$ or $o(\tau)=2$. In this section, we look at $o(\tau)=1$, while $o(\tau)=2$ will 
be treated in section \ref{SectionMoreBraidOrbits}.
\par
\noindent
With $o(\tau)=1$, (\ref{eq:ConjugationTau_1}), (\ref{eq:ConjugationTau_2}), (\ref{eq:ConjugationTau_3}) and (\ref{eq:ConjugationTau_4})
we get $[\underline{\sigma}]=[\sigma_1,\sigma_1,\sigma_3,\sigma_1 \sigma_3 \sigma^{-1}_1]$
with $G=\langle \sigma_1, \sigma_3 \rangle$ and $o(\sigma_1 \sigma_3)=2$ because $o(\sigma_1 \sigma_3)=1$ leads to $G=I$.
\par
\noindent
In addition, $[\sigma_1,\sigma_1 \sigma_3,\sigma_1 \sigma_3 \sigma^{-1}_1]$ is a generating $3$-system of $G$ with
an involution $\sigma_1 \sigma_3$. Similar considerations can be applied to class vectors $(C,D,C,D)$ and $(C,D,D,C)$.
Having done these preparations, we can state the following
\begin{theorem}
\label{BraidOrbitsOfSize2}
Let $G>I$ be a finite group with $Z(G)=I$ and $\sigma_1, \sigma_2 \in G$ with
$G=\langle \sigma_1, \sigma_2 \rangle$ and $o(\sigma_1 \sigma_2)=2$, then (a), (b) and (c) hold.
\begin{itemize}
\item[(a)]
For $e_1=[\sigma_1,\sigma_1, \sigma_2,\sigma_1 \sigma_2 \sigma^{-1}_1]$ and $e_2=e_1^{\beta_{13}}$,
$$Z_1=\{ e_1, e_2 \}=\{ [\sigma_1,\sigma_1, \sigma_2,\sigma_1 \sigma_2 \sigma^{-1}_1], 
[\sigma_1,\sigma_2 \sigma_1 \sigma^{-1}_2,\sigma_2,\sigma_2] \}$$
is an orbit of length $2$ under the action of $B_4$ with
$\rho_4(\beta_{12})=(1)(2)$,
$\rho_4(\beta_{13})=(1,2)$,
$\rho_4(\beta_{14})=(1,2)$,
$\rho_4(B_4)\cong C_2$
and genus $g_{Z_1}=0$.
\item[(b)]
For $e_3=[\sigma_1,\sigma_2,\sigma_1,\sigma_2]$ and $e_4=e_3^{\beta_{12}}$,
$$Z_2=\{ e_3, e_4 \}=\{ [\sigma_1,\sigma_2,\sigma_1,\sigma_2], 
[\sigma_1,\sigma_2,\sigma^{-1}_2 \sigma_1 \sigma_2,\sigma_1 \sigma_2 \sigma^{-1}_1] \}$$
is an orbit of length $2$ under the action of $B_4$ with
$\rho_4(\beta_{12})=(3,4)$,
$\rho_4(\beta_{13})=(3)(4)$,
$\rho_4(\beta_{14})=(3,4)$,
$\rho_4(B_4)\cong C_2$
and genus $g_{Z_2}=0$.
\item[(c)]
For $e_5=[\sigma_1,\sigma_2,\sigma_2,\sigma^{-1}_2 \sigma_1 \sigma_2]$ and $e_6=e_5^{\beta_{12}}$,
$$Z_3=\{ e_5, e_6 \}=\{ [\sigma_1,\sigma_2,\sigma_2,\sigma^{-1}_2 \sigma_1 \sigma_2], 
[\sigma_1,\sigma_2,\sigma_1 \sigma_2 \sigma^{-1}_1,\sigma_1] \}$$
is an orbit of length $2$ under the action of $B_4$ with
$\rho_4(\beta_{12})=(5,6)$,
$\rho_4(\beta_{13})=(5,6)$,
$\rho_4(\beta_{14})=(5)(6)$,
$\rho_4(B_4)\cong C_2$
and genus $g_{Z_3}=0$.
\end{itemize}
\end{theorem}
\begin{proof}
(a) We have $e_1^{\beta_{13}}=[\sigma_1,\sigma_2 \sigma_1 \sigma^{-1}_2,\sigma_2,\sigma_2]$,
$e_1^{\beta_{12}}=e_1$, $e_1^{\beta_{13}\beta_{12}}=e_1^{\beta_{13}}$, $e_1^{\beta_{13}\beta_{13}}=e_1$
and $e_1^{\beta_{14}}=e_1^{\beta_{13}}$.
The equation $e_1=e_1^{\beta_{13}}$ would imply $G=I$. Thus $Z_1$ has length 2.
The genus $g_{Z_1}$ now computes as 
$$g_{Z_1}=1 - 2 + \frac{1}{2}(3\cdot 2 - 2 -1 -1) = 0.$$
The proof for (b) and (c) is analogous.
\end{proof}
\begin{example}
\label{M11_11A11A2A}
For the class vector $(11A,11A,2A)$ of the Mathieu group $M_{11}$ we have $l^i(11A,11A,2A)=1$.
Using Theorem~\ref{BraidOrbitsOfSize2} we can construct $3$ braid orbits of length $2$ for the
class vector $(11A,11A,11A,11A)$. These are exactly the orbits of length $2$
in Example~\ref{M11_11A11A11A11A}.
\end{example}

\section{Symmetries and orbits of length 6}\label{SectionSymmetriesAndOrbitsOfSize6}

The orbits in Theorem~\ref{BraidOrbitsOfSize2} belong to class vectors $(C,C,D,D)$, $(C,D,C,D)$
or $(C,D,D,C)$. For $C=D$, they belong to the single class vector $(C,C,C,C)$. This allows to
construct braid orbits $Z^{sy}$ of length $6$.
\begin{theorem}
\label{BraidOrbitsOfSize6}
Let $G>I$ be a finite group with $Z(G)=I$ and $\sigma_1, \sigma_2 \in G$ with
$G=\langle \sigma_1, \sigma_2 \rangle$, $o(\sigma_1 \sigma_2)=2$ and $[\sigma_1]=[\sigma_2]$, then
$$Z^{sy} = Z_1 \cup Z_2 \cup Z_3$$
is an orbit of length $6$ under the action of 
$B^{sy}_4 = \langle \beta_{12}, \beta_{13}, \beta_{14}, \beta_3, \eta_{23}, \eta_{234} \rangle$ with
\begin{itemize}
\item[]
$\rho_4(\beta_{12})=(1)(2)(3,4)(5,6)$,
\item[]
$\rho_4(\beta_{13})=(1,2)(3)(4)(5,6)$,
\item[]
$\rho_4(\beta_{14})=(1,2)(3,4)(5)(6)$,
\item[]
$\rho_4(\beta_3)=(1,3,2,4)(5)(6)$,
\item[]
$\rho_4(\eta_{23})=(1,4)(2,3)(5,6)$,
\item[]
$\rho_4(\eta_{234})=(1,6,4)(2,5,3)$,
\item[]
$\rho_4(B_4) \cong C_2 \times C_2 < A_6$,
\item[]
$\rho_4(B^{sy}_4) \cong \rho_4(T) \cong S_4 < S_6$
\end{itemize}
for $T=\langle \beta_3, \eta_{23}, \eta_{234} \rangle$ and with $g_{Z^{sy}} = 0$.
\end{theorem}
\begin{proof}
With the assumptions and the notation from Theorem~\ref{BraidOrbitsOfSize2} we get $e_i \neq e_j$
for $i \neq j$ and all $i,j \in \{ 1,\dots,6 \}$, thus $Z^{sy}$ has exactly $6$ elements.
The action of $\beta_{12}, \beta_{13}$ and $\beta_{14}$ can be taken from Theorem~\ref{BraidOrbitsOfSize2}.
Table \ref{tab:ActionBeta3Eta23Eta234} contains the action of $\beta_3, \eta_{23}$ and $\eta_{234}$.
\begin{table}[!htbp]
\centering
\footnotesize
\captionsetup{font=footnotesize}
\caption{The action of $\beta_3, \eta_{23}$ and $\eta_{234}$}
\label{tab:ActionBeta3Eta23Eta234}
\begin{tabular}{lllllll}
\hline\noalign{\smallskip}
. & $e_1$ & $e_2$ & $e_3$ & $e_4$ & $e_5$ & $e_6$ \\
\noalign{\smallskip}\hline\noalign{\smallskip}
$\beta_3$    & $e_3$ & $e_4$ & $e_2$ & $e_1$ & $e_5$ & $e_6$ \\
$\eta_{23}$  & $e_4$ & $e_3$ & $e_2$ & $e_1$ & $e_6$ & $e_5$ \\
$\eta_{234}$ & $e_6$ & $e_5$ & $e_2$ & $e_1$ & $e_3$ & $e_4$ \\
\noalign{\smallskip}\hline
\end{tabular}
\end{table}
We get $\rho_4(\beta_3)=(1,3,2,4)(5)(6)$, $\rho_4(\eta_{234})=(1,6,4)(2,5,3)$, $\rho_4(\eta_{234})=(1,6,4)(2,5,3)$
and $\rho_4(T)$ acts transitively on $\{ 1,\dots,6 \}$.
With $T \leq B^{sy}_4$,
$\rho_4(\beta_3)^2 = \rho_4(\beta_{14})$,
$\rho_4(\eta_{234}) \rho_4(\beta_3)^2 \rho_4(\eta_{234})^{-1} = \rho_4(\beta_{12})$ and
$\rho_4(\beta_{12}) \rho_4(\beta_{13}) \rho_4(\beta_{14}) = \iota$ 
we get $\rho_4(B^{sy}_4) \cong \rho_4(T)$.
\par
\noindent
From Corollary~\ref{PropertiesOfPermutationRepresentationsCase4} follows $\rho_4(B_4) \cong C_2 \times C_2$.
Using the computer algebra system GAP \cite{RefGAP}, we get $\rho_4(T) \cong S_4$.
Finally, the genus is $$g_{Z^{sy}} = 1 - 6 + \frac{1}{2}(3\cdot 6 - 3 -3 -2) = 0.$$
\end{proof}
\noindent
With the possible exception of rigidity, all assumptions from \cite[III, Theorem 7.10]{RefMM2018}
for the orbit $Z^{sy}$ from Theorem~\ref{BraidOrbitsOfSize6} are satisfied. If $Z^{sy}$ is rigid,
and $[\sigma_1]$ is a rational conjugacy class, $G$ occurs as Galois group of a geometric Galois
extension of $\mathbb{Q}(v,t)$.
\begin{example}
\label{M11_11A11A11A11ASymmetric}
For the class vector $(11A,11A,2A)$ of the Mathieu group $M_{11}$ we already know $l^i(11A,11A,2A)=1$.
Using Theorem~\ref{BraidOrbitsOfSize6} we can construct an orbit of length $6$ for the
class vector $(11A,11A,11A,11A)$. This is a rigid orbit due to its length and
yields a geometric Galois realization of $M_{11}$ over $\mathbb{Q}(\sqrt{-11})(v,t)$.
See also Example~\ref{M11_11A11A11A11A} and Example~\ref{M11_11A11A2A}.
\end{example}
\begin{example}
\label{SO53_10A10A10A1ASymmetric}
Let $(10A,10A,10A,10A)$ be the class vector of the special orthogonal group $SO_5(3)$ of dimension $5$ 
over the field $\mathbb{F}_3$ with the unique class of elements of order $10$.
Thus $10A$ is a rational conjugacy class of $SO_5(3)$.
Table  \ref{tab:Data_BraidOrbitsOfSO53} contains the splitting into $B_4$- resp. $B^{sy}_4$-orbits 
of the set $\Sigma^i(10A,10A,10A,10A)$.
\begin{table}[!htbp]
\centering
\footnotesize
\captionsetup{font=footnotesize}
\caption{All Braid orbits of $SO_5(3)$ in $\Sigma^i(10A,10A,10A,10A)$}
\label{tab:Data_BraidOrbitsOfSO53}
\begin{tabular}{ll}
\hline\noalign{\smallskip}
$\vert Z \vert$       & $2_{18}$, $20_6$, $30_3$, $120_3$, $410_3$, 510, $772_3$, $1,250_3$, $5,180_3$, 257,600, 271,160 \\
\noalign{\smallskip}\hline\noalign{\smallskip}
$\vert Z^{sy} \vert$  & $6_6$, 90, 120, 360, 510, 1,230, 2,316, 3,750, 15,540, 257,600, 271,160 \\
\noalign{\smallskip}\hline
\end{tabular}
\end{table}
\par
\noindent
For the $4$ classes of involutions $2A$, $2B$, $2C$ and $2D$ of $SO_5(3)$, we have 
$$l^i(10A,10A,2A)=1,l^i(10A,10A,2B)=5,$$
$$l^i(10A,10A,2C)=l^i(10A,10A,2D)=0.$$
Using Theorem~\ref{BraidOrbitsOfSize6} we can construct the $6$ $B^{sy}_4$-orbits of length $6$.
The orbit of length $6$ originating from $(10A,10A,2A)$ is a rigid orbit and
yields a geometric Galois realization of $SO_5(3)$ over $\mathbb{Q}(v,t)$.
\end{example}
\noindent
Interchanging $\sigma_1$ and $\sigma_2$ keeps the prerequisites in Theorem~\ref{BraidOrbitsOfSize2} and
Theorem~\ref{BraidOrbitsOfSize6} due to $o(\sigma_1 \sigma_2)=o(\sigma_2 \sigma_1)$.
With $e_i = e_i(\sigma_1, \sigma_2)$ for $i=1,\dots,6$, let 
$$e_{i+6} = e_i(\sigma_2, \sigma_1), i=1,\dots,6,$$
$$Z_4 = \{ e_7,e_8 \}, Z_5 = \{ e_9,e_{10} \},  Z_6 = \{ e_{11},e_{12} \},$$
$$Z^{sy}_1 = Z_1 \cup Z_2 \cup Z_3, Z^{sy}_2 = Z_4 \cup Z_5 \cup Z_6.$$
Then Theorem~\ref{BraidOrbitsOfSize2} and Theorem~\ref{BraidOrbitsOfSize6}
also hold for $Z_4$, $Z_5$, $Z_6$ and $Z^{sy}_2$.
\begin{theorem}
\label{BraidOrbitsOfSize12}
Let $G>I$ be a finite group with $Z(G)=I$ and $\sigma_1, \sigma_2 \in G$ with
$G=\langle \sigma_1, \sigma_2 \rangle$, $o(\sigma_1 \sigma_2)=2$ and $[\sigma_1]=[\sigma_2]$, then
$$Z^{H_4} = Z^{sy}_1 \cup Z^{sy}_2$$
is an orbit of length $6$ or $12$ under the action of $H_4 = \langle \beta_2, \beta_3, \beta_4 \rangle$ with
\begin{itemize}
\item[]
$\rho_4(\beta_2)=(1)(2)(3,5,4,6)$ resp. $(1)(2)(3,11,4,12)(5,10,6,9)(7)(8)$,
\item[]
$\rho_4(\beta_3)=(1,3,2,4)(5)(6)$ resp. $(1,3,2,4)(5)(6)(7,9,8,10)(11)(12)$,
\item[]
$\rho_4(\beta_4)=(1)(2)(3,6,4,5)$ resp. $(1)(2)(3,6,4,5)(7)(8)(9,12,10,11)$,
\item[]
$\rho_4(H_4) \cong S_4 \cong (C_2 \times C_2):S_3 < S_6$ resp. $(C_2 \times C_2):S_4 < A_{12}$,
\item[]
$\rho_4(B_4) \cong C_2 \times C_2 < A_6$ resp. $C_2 \times C_2 < A_{12}$ and
\item[]
$\rho_4(H_4)/\rho_4(B_4) \cong S_3$ resp. $S_4$.
\end{itemize}
\end{theorem}
\begin{proof}
The action of $\beta_2, \beta_3$ and $\beta_4$ is given in table \ref{tab:ActionBeta2Beta3Beta4}.
\begin{table}[!htbp]
\centering
\footnotesize
\captionsetup{font=footnotesize}
\caption{The action of $\beta_2, \beta_3$ and $\beta_4$}
\label{tab:ActionBeta2Beta3Beta4}
\begin{tabular}{lllllllllllll}
\hline\noalign{\smallskip}
.              & $e_1$ & $e_2$ & $e_3$ & $e_4$ & $e_5$ & $e_6$ & $e_7$ & $e_8$ & $e_9$ & $e_{10}$ & $e_{11}$ & $e_{12}$ \\
\noalign{\smallskip}\hline\noalign{\smallskip}
$\beta_2$ & $e_1$ & $e_2$ & $e_{11}$ & $e_{12}$ & $e_{10}$ & $e_9$ & $e_7$ & $e_8$  & $e_5$  & $e_6$  & $e_4$  & $e_3$ \\
$\beta_3$ & $e_3$ & $e_4$ & $e_2$  & $e_1$  & $e_5$ &  $e_6$ & $e_9$ & $e_{10}$ & $e_8$  & $e_7$  & $e_{11}$ & $e_{12}$ \\
$\beta_4$ & $e_1$ & $e_2$ & $e_6$  & $e_5$  & $e_3$ &  $e_4$ & $e_7$ & $e_8$  & $e_{12}$ & $e_{11}$ & $e_9$  & $e_{10}$ \\
\noalign{\smallskip}\hline
\end{tabular}
\end{table}
By Theorem~\ref{BraidOrbitsOfSize6}, $Z^{sy}_1$ and $Z^{sy}_2$ have length 6. If $Z^{sy}_1 \cap Z^{sy}_2 = \emptyset$,
$Z^{H_4}$ has length 12 and $\rho_4(H_4)$ acts transitive on $\{ 1,\dots,12 \}$.
Now $sign(\rho_4(\beta_i)) = 1, i = 2,3,4$ leads to $\rho_4(H_4) \leq A_{12}$.
Using GAP \cite{RefGAP}, we get $\rho_4(H_4) \cong (C_2 \times C_2):S_4$.
\par
\noindent
Now let $Z^{sy}_1 \cap Z^{sy}_2 \neq \emptyset$ and $e \in Z^{sy}_1 \cap Z^{sy}_2$.
Then we have $Z^{sy}_1 = {\{ e \}}^{B^{sy}_4} = Z^{sy}_2$, thus $Z^{H_4} = Z^{sy}_1 = Z^{sy}_2$
and $Z^{H_4}$ has length 6.
From table \ref{tab:ActionBeta2Beta3Beta4} we get $\rho_4(\beta_3)=(1,3,2,4)(5)(6)$,
$\rho_4(\beta_4)=(1)(2)(3,6,4,5)$ and considering the fixed points of $\rho_4(\beta_3)$ and 
$\rho_4(\beta_4)$ further $\{ e_1,e_2 \}=\{ e_7,e_8 \}$,
$\{ e_3,e_4 \}=\{ e_9,e_{10} \}$ and $\{ e_5,e_6 \}=\{ e_{11},e_{12} \}$.
For $e_5=e_{11}$ we get $\rho_4(\beta_2)=(1)(2)(3,5,4,6)$, for $e_5=e_{12}$ we get
$\rho_4(\beta_2)=(1)(2)(3,6,4,5)= ((1)(2)(3,5,4,6))^{-1}$.
Without restriction, we can assume $e_5=e_{11}$.
Using GAP \cite{RefGAP}, we get $\rho_4(H_4) \cong (C_2 \times C_2):S_3$.
From Corollary~\ref{PropertiesOfPermutationRepresentationsCase4} follows $\rho_4(B_4) \cong C_2 \times C_2$.
\par
\noindent
Finally, from Proposition~\ref{PropertiesOfPermutationRepresentations} and
$\vert \rho_4(H_4)/\rho_4(B_4) \vert = 6$ resp. $24$ follows $\rho_4(H_4)/\rho_4(B_4) \cong S_3$ resp. $S_4$.
\end{proof}
\noindent
Note that we have $S_3 \cong S_2(2)$, $A_4 \cong S_2(3) $ and $S_6 \cong S_4(2)$.
With the assumptions in Theorem~\ref{BraidOrbitsOfSize12}, suppose there exits an element $\tau \in G$
with $\tau^{-1} \sigma_1 \tau = \sigma_2$ and $\tau^{-1} \sigma_2 \tau = \sigma_1$, then we
have $o(\tau)=2$ and $e^{\tau}_i = e_{i+6}$ for $i=1,\dots 6$,
thus $Z^{H_4} = Z^{sy}_1 = Z^{sy}_2$.
\begin{example}
\label{M11_11A11A11A11ASymmetricContinued}
For the class vector $(11A,11A,11A,11A)$ of the Mathieu group $M_{11}$ there exists exactly
one orbit $Z^{H_4}$ with length 6 and $\rho_4(H_4) \cong S_4 < S_6$.
See also Example~\ref{M11_11A11A11A11A}, Example~\ref{M11_11A11A2A} and Example~\ref{M11_11A11A11A11ASymmetric}.
\end{example}
\noindent
Another interesting example is the class vector $(8A,8A,8A,8A)$ of the Mathieu group $M_{23}$.
See section \ref{SectionClassVectorsOfDimension4InM23} later.

\section{More braid orbits in dimension 4}\label{SectionMoreBraidOrbits}

Now we look at $o(\tau)=2$ from section \ref{SectionFixedPointsAndOrbitsOfSize2}. The element $\varphi_{4,1}$ acts on
the set $\Sigma^i(C,C,D,D)$ in the same way as $\eta_2$ from \cite[III, §2, Satz 3]{RefMat1987}.
Thus we can reuse parts of the proof to \cite[III, §2, Satz 5(b)]{RefMat1987}.
From equation (\ref{eq:ConjugationTau_2}) follows $G=\langle \tau, \sigma_2, \sigma_3 \rangle$
and we have $o(\tau \sigma_2 \sigma_3) \leq 2$.
For $o(\tau \sigma_2 \sigma_3)=1$, $[\underline{\sigma}]$ equals $e_2$ from Theorem~\ref{BraidOrbitsOfSize2},
thus without restriction, we can assume $o(\tau \sigma_2 \sigma_3)=2$.
Then $[(\tau \sigma_2 \sigma_3)^{-1}, \tau, \sigma_2, \sigma_3] \in \Sigma^i_4(G)$
is a generating $4$-system of $G$ with two classes of involutions $[\tau \sigma_2 \sigma_3]$ and $[\tau]$.
\par
\noindent
So starting from systems in $\Sigma^i(C,C,D,D)$, we translated them to systems in $\Sigma^i(2A,2B,C,D)$
with some involution classes $2A$ and $2B$.
\par
\noindent
Conversely, using $\varphi_{4,1}$, we can translate systems in $\Sigma^i(2A,2B,C,D)$ into systems
in $\Sigma^i(C,C,D,D)$. If $C$ is a class of involutions, we can repeat the translation step and
get translated systems in $\Sigma^i(D,D,D,D)$.
Similar considerations can be applied to $\varphi_{4,2}$ and $\varphi_{4,3}$.
\begin{theorem}
\label{TranslationOfBraidOrbits}
Let $G$ be a finite group with $Z(G)=I$, $(C_1,C_2,C_3,C_4)$ a class vector of length $4$ of $G$
and $[\underline{\sigma}] \in \Sigma^i_4(C_1,C_2,C_3,C_4)$.
\begin{itemize}
\item[(a)]
If $o(\sigma_1) = o(\sigma_2) = 2$ and 
$G=\langle \sigma^{-1}_2 \sigma_3 \sigma_2, \sigma_3, \sigma_4 \rangle$, 
then we have
\par
$f_{4,1}([\underline{\sigma}])=[\sigma^{-1}_2 \sigma_3 \sigma_2, \sigma_3, \sigma_4, \sigma^{-1}_1 \sigma_4 \sigma_1]
\in \Sigma^i_4(C_3,C_3,C_4,C_4)$
\par
with $f_{4,1}([\underline{\sigma}])^{\varphi_{4,1}} = f_{4,1}([\underline{\sigma}])$.
\item[(b)]
If $o(\sigma_3) = o(\sigma_4) = 2$ and 
$G=\langle \sigma_1, \sigma_4 \sigma_1 \sigma^{-1}_4, \sigma_4 \sigma_2 \sigma^{-1}_4 \rangle$, 
then we have
\par
$g_{4,1}([\underline{\sigma}])=[\sigma_1, \sigma_4 \sigma_1 \sigma^{-1}_4, \sigma_4 \sigma_2 \sigma^{-1}_4, 
\sigma_4 \sigma_3 \sigma_2 \sigma^{-1}_3 \sigma^{-1}_4] \in \Sigma^i_4(C_1,C_1,C_2,C_2)$
\par
with $g_{4,1}([\underline{\sigma}])^{\varphi_{4,1}} = g_{4,1}([\underline{\sigma}])$.
\item[(c)]
If $o(\sigma_2) = o(\sigma_4) = 2$ and 
$G=\langle \sigma_1, \sigma^{-1}_2 \sigma_3 \sigma_2, \sigma^{-1}_2 \sigma_1 \sigma_2 \rangle$, 
then we have
\par
$f_{4,2}([\underline{\sigma}])=[\sigma_1, \sigma^{-1}_2 \sigma_3 \sigma_2, \sigma^{-1}_2 \sigma_1 \sigma_2, \sigma_3]
\in \Sigma^i_4(C_1,C_3,C_1,C_3)$
\par
with $f_{4,2}([\underline{\sigma}])^{\varphi_{4,2}} = f_{4,2}([\underline{\sigma}])$.
\item[(d)]
If $o(\sigma_2) = o(\sigma_3) = 2$ and 
$G=\langle  \sigma_4, \sigma_1, \sigma^{-1}_2 \sigma_1 \sigma_2 \rangle$, 
then we have
\par
$f_{4,3}([\underline{\sigma}])=[\sigma_4, \sigma_1, \sigma^{-1}_2 \sigma_1 \sigma_2, \sigma^{-1}_3 \sigma_4 \sigma_3]
\in \Sigma^i_4(C_4,C_1,C_1,C_4)$
\par
with $f_{4,3}([\underline{\sigma}])^{\varphi_{4,3}} = f_{4,3}([\underline{\sigma}])$.
\end{itemize}
\end{theorem}
\begin{proof}
(a) We have
$$\sigma^{-1}_2 \sigma_3 \sigma_2 \sigma_3 \sigma_4 \sigma^{-1}_1 \sigma_4 \sigma_1 =
\sigma^{-1}_2 \sigma_3 (\sigma_1 \sigma_1) \sigma_2 \sigma_3 \sigma_4 \sigma^{-1}_1 \sigma_4 \sigma_1 =
\sigma_2 \sigma_3 \sigma_4 \sigma_1 = \iota,$$
thus $f_{4,1}([\underline{\sigma}]) \in \Sigma^i_4(C_3,C_3,C_4,C_4)$.
By direct computation we get 
$$f_{4,1}([\underline{\sigma}])^{\varphi_{4,1}} = f_{4,1}([\underline{\sigma}]).$$
The proof for (b) to (d) is analogous.
\end{proof}
\noindent
We omit the cases $o(\sigma_1) = o(\sigma_3) = 2$ and $o(\sigma_1) = o(\sigma_4) = 2$ as they are not needed here.
Again let $G$ be a finite group with $Z(G)=I$, $(2A,2B,C,D)$ a class vector of $G$ with involution classes $2A$ and $2B$
and $Z \subseteq \Sigma^i_4(2A,2B,C,D)$ be a $B_4$-orbit. Define
$$f_{4,1}(Z) = \{ f_{4,1}([\underline{\sigma}]) \mid [\underline{\sigma}] \in Z, 
G=\langle \sigma^{-1}_2 \sigma_3 \sigma_2, \sigma_3, \sigma_4 \rangle \} \subseteq \Sigma^i_4(C,C,D,D).$$
By Theorem~\ref{TranslationOfBraidOrbits}(a), $f_{4,1}(Z)$ is a set of fixed points of $\langle \varphi_{4,1} \rangle$.
The set 
$$Z^{\beta_2} = \{ [\underline{\sigma}]^{\beta_2} \mid [\underline{\sigma}] \in Z \} \subseteq \Sigma^i_4(2B,2A,C,D)$$
is a $B_4$-orbit and we get that
$f_{4,1}(Z^{\beta_2}) \subseteq \Sigma^i_4(C,C,D,D)$ again is a set of fixed points of $\langle \varphi_{4,1} \rangle$.
Examples, especially the $M_{24}$-example in section \ref{SectionClassVectorsOfDimension4InM24} show, 
that $f_{4,1}(Z) \cup  f_{4,1}(Z^{\beta_2})$ is often already a $B_4$-orbit in $\Sigma^i_4(C,C,D,D)$.
The sets $f_{4,2}(Z)$ and $f_{4,3}(Z)$ can be defined and treated in the same way.
\par
\noindent
In \cite[III, 5.4]{RefMM2018} the same results are achieved for the translation of braid orbits.
The most interesting application of Theorem~\ref{TranslationOfBraidOrbits}
is the class vector $(12B,2A,2A,2A)$ of the Mathieu group $M_{24}$.
This example will be introduced in section \ref{SectionClassVectorsOfDimension4InM24}.

\section{Fixed points and orbits of length 4}\label{SectionFixedPointsAndOrbitsOfSize4}

In \cite[Theorem 3]{RefGon2008} a complete list of the maximal finite subgroups of $H_m$ is given.
For $H_4$, the maximal finite subgroups up to conjugation are the binary tetrahedral group $T_1$
and the dicyclic group $Dic_{16}$ of order $16$.
Their algebraic structure is given by 
$T_1 = \langle x_4, y_4 \rangle \rtimes \langle \alpha_{4,1}^2 \rangle \cong Q_8 \rtimes C_3$ 
with $x_4$, $y_4$ from section \ref{SectionInvariantsUnderTheActionOfBraids}, $\alpha_{4,1} = \beta_2 \beta_3 \beta_4^2$,
and $Dic_{16} = \langle x_4, \alpha_{4,0} \rangle$ with $\alpha_{4,0} = \beta_2 \beta_3 \beta_4$,
see \cite[Remark 13]{RefGon2008}. 
The element $\alpha_{4,0}$ has order $8$, while $\alpha_{4,1}$ has order $6$.
\par
\noindent
Looking at the structure of the fixed points of the action of $\langle \alpha_{4,1}^2 \rangle$ on the set $\Sigma^i(C_1,C_2,C_3,C_4)^{sy}$
in the same way as we did it for $Q_8 = \langle x_4, y_4 \rangle = F_4$ in section \ref{SectionFixedPointsAndOrbitsOfSize2}
yields the next theorem.
\begin{theorem}
\label{BraidOrbitsOfSize4}
Let $G>I$ be a finite group with $Z(G)=I$ and $\sigma_1, \sigma_2 \in G$ with
$G=\langle \sigma_1, \sigma_2 \rangle$,
$\sigma_1 \sigma_2 \sigma_1 = \sigma_2 \sigma_1 \sigma_2$,
$o(\sigma_1 \sigma_2 \sigma_1 )=2$,
$C=[\sigma_1]$ and $D=[\sigma_1 \sigma_2 \sigma_1]$.
Then $Z_4 = \{ h_1\}^{B_4} = \{ h_1,h_2,h_3,h_4 \} \subseteq \Sigma^i(C,C,C,D)$ with
\begin{itemize}
\item[]
$h_1 = [\sigma_1, \sigma_2, \sigma_1, \sigma_1 \sigma_2 \sigma_1],$
\item[]
$h_2 = [\sigma_2 \sigma_1 \sigma_2^{-1}, \sigma_2, \sigma_1, \sigma_2 \sigma_1^2],$
\item[]
$h_3 = [\sigma_2, \sigma_1, \sigma_1, \sigma_2 \sigma_1^2],$
\item[]
$h_4 = [\sigma_2, \sigma_2, \sigma_1, \sigma_2^2 \sigma_1]$
\end{itemize}
is an orbit of length $4$ under the action of $B_4$ with
$\rho_4(\beta_{12})=(1,2,3)(4)$,
$\rho_4(\beta_{13})=(1,3,4)(2)$,
$\rho_4(\beta_{14})=(1,4,2)(3)$,
$\rho_4(B_4)\cong A_4$
and with genus $g_{Z_4}=0$.
\end{theorem}
\begin{proof}
From $\sigma_1 = \sigma_1 \sigma_2 \sigma_1 \sigma_1^{-1} \sigma_2^{-1} = 
\sigma_2 \sigma_1 \sigma_2 \sigma_1^{-1} \sigma_2^{-1}$
we get $C = [\sigma_1] = [\sigma_2]$.
\par
\noindent
With $\sigma_1(\sigma_2 \sigma_1^2)\sigma_1^{-1} = \sigma_1 \sigma_2 \sigma_1$
and $\sigma_2^{-1}(\sigma_2^2 \sigma_1)\sigma_2 = \sigma_2 \sigma_1 \sigma_2 = \sigma_1 \sigma_2 \sigma_1$
we get
$D = [\sigma_1 \sigma_2 \sigma_1] = [\sigma_2 \sigma_1^2] = [\sigma_2^2 \sigma_1]$
and $h_i \in \Sigma^i(C,C,C,D)$ for $i=1,2,3,4$ follows.
By direct computation, $h_i \neq h_j$ for $i \neq j, i,j = 1,2,3,4$ can be verified, thus $\vert Z_4 \vert = 4$.
Table \ref{tab:ActionB4OnZ4} contains the action of $B_4$ on $Z_4$.
\begin{table}[!htbp]
\centering
\footnotesize
\captionsetup{font=footnotesize}
\caption{The action of $B_4$ on $Z_4$}
\label{tab:ActionB4OnZ4}
\begin{tabular}{lllll}
\hline\noalign{\smallskip}
. & $h_1$ & $h_2$ & $h_3$ & $h_4$ \\
\noalign{\smallskip}\hline\noalign{\smallskip}
$\beta_{12}$ & $h_2$ & $h_3$ & $h_1$ & $h_4$ \\
$\beta_{13}$ & $h_3$ & $h_2$ & $h_4$ & $h_1$ \\
$\beta_{14}$ & $h_4$ & $h_1$ & $h_3$ & $h_2$ \\
\noalign{\smallskip}\hline
\end{tabular}
\end{table}
We get $\rho_4(\beta_{12})=(1,2,3)(4)$, $\rho_4(\beta_{13})=(1,3,4)(2)$, $\rho_4(\beta_{14})=(1,4,2)(3)$,
$\rho_4(B_4)\cong A_4$ and
$$g_{Z_4}= 1 - 4 + \frac{1}{2}(3\cdot 4 -2 -2 -2) = 0.$$
\end{proof}
\noindent
Note that we have $h_1^{\alpha_{4,1}^2} = h_1$ and $h_2^{\alpha_{4,1}^2} \neq h_2$, thus the $B_4$-orbit $Z_4$ is not
a set of fixed points under $\langle \alpha_{4,1}^2 \rangle$.
With the possible exception of rigidity, all assumptions from \cite[III, Corollary 5.8]{RefMM2018}
for the $B_4$-orbit $Z_4$ are satisfied. If $Z_4$ is rigid, and $C$ is a rational conjugacy class, 
$G$ occurs as Galois group of a geometric Galois extension of $\mathbb{Q}(u_1,u_2,u_3,u_4,t)$.
\par
\noindent
If $G$ is a group satisfying the assumptions of Theorem~\ref{BraidOrbitsOfSize4}, then
$G$ is a factor group of the {\it Artin braid group} $Ar_3$ with $2$ strands.
Moreover,
$$[\tau_1, \tau_2, \tau_3] = [\sigma_1 \sigma_2 \sigma_1, \sigma_1 \sigma_2, \sigma_1]$$
is a generating $3$-system of $G$ with $o(\tau_1)=2$ and $o(\tau_2)=3$.
Thus $G$ has a $(2,3)$-generating system and hence is a factor group of the modular
group $PSL_2(\mathbb{Z}) \cong C_2 \ast C_3$.
Note that we have $Ar_3/Z(Ar_3) \cong PSL_2(\mathbb{Z})$, see \cite[Theorem A.2]{RefKT2008}.
\begin{corollary}
\label{Systems23ToBraidOrbitsOfSize4}
Let $G>I$ be a finite group with $Z(G)=I$ and $(2A,3A,C)$ a class vector of $G$ with $l^i(2A,3A,C)>0$,
$2A$ a class of involutions and $3A$ a class of elements with order $3$.
Then $\Sigma^i(C,C,C,2A)$ contains $B_4$-orbits $Z_4$ with length $4$.
In addition, we have $l^i(C,C,3A^{-1})>0$.
\end{corollary}
\begin{proof}
For $[\tau_1, \tau_2, \tau_3] \in \Sigma^i(2A,3A,C)$ we have $o(\tau_1)=2$, $o(\tau_2)=3$
and $\tau_3 = (\tau_1 \tau_2)^{-1}$.
Define $\sigma_1 = \tau_3 = \tau_2^2 \tau_1$, $\sigma_2 = \tau_1 \tau_2^2$ and $\sigma_3 = (\sigma_1 \sigma_2)^{-1}$.
Then $\sigma_1$ and $\sigma_2$ satisfy the assumptions of Theorem~\ref{BraidOrbitsOfSize4}, so for each
$(2,3)$-generating system of $G$ in $\Sigma^i(2A,3A,C)$ we can construct a $B_4$-orbit $Z_4$
using Theorem~\ref{BraidOrbitsOfSize4}.
Due to $\sigma_1 = \tau_1^{-1} \sigma_2 \tau_1$ and $\sigma_1 \sigma_2 \sigma_1 = \tau_1$ we get
$Z_4 \subseteq \Sigma^i(C,C,C,2A)$.
\par
\noindent
In addition, $[\sigma_1, \sigma_2, (\sigma_1, \sigma_2)^{-1}]$ is a generating $3$-system of $G$
in $\Sigma^i(C,C,3A^{-1})$ due to  $[\sigma_2] = [\sigma_1] = [\tau_3]$ 
and $\sigma_1 \sigma_2 = \tau_2$.
\end{proof}
\begin{example}
\label{M23_CCC2A_NoZ4}
The Mathieu group $M_{23}$ contains a single class $2A$ of involutions. By direct computation we get, that
no class vector $(C,C,C,2A)$ of $M_{23}$ contains a $B_4$-orbit $Z_4$ with length $\vert Z_4 \vert = 4$.
See also tables \ref{tab:Data_BraidOrbitsOfM23} and \ref{tab:CurrentStateOfComputations} for the rational classes of $M_{23}$.
This corresponds to the fact, that $M_{23}$ has no $(2,3)$-generating system, see \cite[Anhang 1.2]{RefHae1987}.
The same holds for the Suzuki groups $Sz(2^{2n+1})$ as they do not contain elements of order $3$.
\end{example}
\begin{example}
\label{M24_CCC2_Z4}
The Mathieu group $M_{24}$ has two classes of involutions $2A, 2B$ and two classes of elements of order $3$, 
namely $3A$ and $3B$ with $3A^{-1}=3A$ and $3B^{-1}=3B$.
Let $23A$ be one of its two classes of elements of order $23$. Due to
$l^i(2A,3A,23A) = l^i(2B,3B,23A) = 0$ and $l^i(2A,3B,23A) = l^i(2B,3A,23A) = 1$,
$M_{24}$ has two $(2,3,23)$-generating systems, see \cite[Anhang 2.2]{RefHae1987}.
By Corollary~\ref{Systems23ToBraidOrbitsOfSize4}, we get a $B_4$-orbit $Z_4$ in $\Sigma^i(23A,23A,23A,2A)$
and an other $B_4$-orbit $Z_4$ in $\Sigma^i(23A,23A,23A,2B)$.
Each $B_4$-orbit $Z_4$ is unique by its size, thus rigid and we get
two geometric Galois extensions of $\mathbb{Q}(\sqrt{-23})(u_1,u_2,u_3,u_4,t)$ with Galois group $M_{24}$.
\end{example}
\begin{example}
\label{L2_8_7A7A7A2A_Part1}
The set $\Sigma^i(7A,7A,7A,2A)$ of $L_2(8) \leq A_9$ splits under the action of $B_4$ into
two braid orbits $Z_4$ and $Z_{84}$ of length $4$ and $84$.
We have $[\underline{\sigma}] \in Z_4$ with
$[\underline{\sigma}]=$
$$[(2,9,4,3,5,7,6), (1,4,2,8,7,9,5), (2,9,4,3,5,7,6), (1,3)(2,7)(4,5)(6,8)].$$
$[\underline{\sigma}]$ satisfies $\sigma_1 = \sigma_3$, $\sigma_4 = \sigma_1 \sigma_2 \sigma_1$,
$o(\sigma_1 \sigma_2 \sigma_1 )=2$ and $\sigma_1 \sigma_2 \sigma_1 = \sigma_2 \sigma_1 \sigma_2$,
thus $[\underline{\sigma}] = h_1$ and $Z_4$ is the $B_4$-orbit from Theorem~\ref{BraidOrbitsOfSize4}.
The $B_4$-orbit $Z_4$ is a rigid orbit due to its size and
yields a geometric Galois realization of $L_2(8)$ over $\mathbb{Q}(\cos(\frac{2\pi}{7}))(u_1,u_2,u_3,u_4,t)$.
\end{example}
\noindent
In table \ref{tab:Data_B4OrbitsSiz4} we give examples of class vectors containing $B_4$-orbits $Z_4$ 
of length $\vert Z_4 \vert = 4$.
Always $Z_4$ is the $B_4$-orbit from Theorem~\ref{BraidOrbitsOfSize4} and we have $g_{Z_4}=0$, $\rho^t_4(\beta_{ij}) = (1)(3)$
and $\rho_4(B_4) \cong A_4$ on $Z_4$.
\begin{table}[H]
\centering
\footnotesize
\captionsetup{font=footnotesize}
\caption{Class vectors containing $B_4$-orbits $Z_4$ of length $\vert Z_4 \vert = 4$}
\label{tab:Data_B4OrbitsSiz4}
\begin{tabular}{lll}
\hline\noalign{\smallskip}
$G$       & $C$                         & $ \vert Z \vert $ \\
\noalign{\smallskip}\hline\noalign{\smallskip}
$A_5$     & $(5A,5A,5A,2A)$ & $4$ \\
$S_7$     & $(10A,10A,10A,2C)$  & $4, 260, 784$ \\
$L_2(7)$  & $(7A,7A,7A,2A)$ & $4$ \\
$L_2(8)$  & $(7A,7A,7A,2A)$ & $4, 84$ \\
$L_2(8)$  & $(9A,9A,9A,2A)$ & $4, 36$ \\
$L_2(11)$ & $(11A,11A,11A,2A)$  & $4, 32$ \\
$L_2(13)$ & $(13A,13A,13A,2A)$  & $4, 32$ \\
$L_2(16)$ & $(15A,15A,15A,2A)$  & $4, 300$ \\
$L_2(16)$ & $(17A,17A,17A,2A)$  & $4, 204$ \\
$L_3(3)$  & $(13A,13A,13A,2A)$  & $4, 288$ \\
$S_6(2)$ & $(15A,15A,15A,2D)$   & $4_2, 1.629.944$ \\
$M_{12}$  & $(10A,10A,10A,2B)$  & $4_2, 22.404_2$ \\
$M_{12}$  & $(11A,11A,11A,2A)$  & $4, 31.788$ \\
$M_{12}$  & $(11A,11A,11A,2B)$  & $4, 16.428, 16.844$ \\
\noalign{\smallskip}\hline
\end{tabular}
\end{table}
\par
\noindent
In contrast to $F_4$, the group $\langle \alpha_{4,1}^2 \rangle$ does not satisfy (\ref{eq:InvariantCondition}) from Proposition~\ref{BmActionOnFixedPointsF}.
So it cannot be used to define an invariant $F_{B_4}(.,.)$. But looking at the fixed points of the action of 
$\langle \alpha_{4,1}^2 \rangle$ on $\Sigma^i(C_1,C_2,C_3,C_4)^{sy}$ nevertheless generated a new generic $B_4$-orbit $Z_4$.
This is a hint to check all maximal finite subgroups of $H_m$ from \cite[Theorem 3]{RefGon2008}.
Theorem~\ref{BraidOrbitsOfSize4} will be used in section \ref{SectionSmallOrbitsInDimension6} for the construction of small $B_6$-orbits.
\par
\noindent
Finally, let us have a look at the action of $\langle \alpha_{4,0} \rangle$ on $\Sigma^i(C_1,C_2,C_3,C_4)^{sy}$ which is given by
$$[\underline{\sigma}]^{\alpha_{4,0}} = [\sigma_2, \sigma_3, \sigma_4, \sigma_1].$$
This is the action of $d(\eta_4)$ in \cite[III, § 1, Satz 3(4)]{RefMat1987} and from
\cite[III, § 2, Satz 5(d)]{RefMat1987} we get
\begin{proposition}
\label{FixedPointsOfAlpha40}
Let $G>I$ be a finite group with $Z(G)=I$ and $(C_1,C_2,C_3)$ a class vector of $G$ with $l^i(C_1,C_2,C_3)>0$,
$[\underline{\sigma}] = [\sigma_1,\sigma_2,\sigma_3] \in \Sigma^i(C_1,C_2,C_3)$ with $\sigma_1^4 = \sigma_2^4 = \iota$,
$G = \langle \sigma_1^{-2} \sigma_3 \sigma_1^2, \sigma_1^{-1} \sigma_3 \sigma_1, \sigma_3 \rangle$ and $C = C_3$. Then we have
$$v = [\sigma_1^{-2} \sigma_3 \sigma_1^2, \sigma_1^{-1} \sigma_3 \sigma_1, \sigma_3, \sigma_1^{-3} \sigma_3 \sigma_1^3] \in \Sigma^i(C,C,C,C)$$
and
$v^{\alpha_{4,0}} = v$.
\end{proposition}
\begin{proof}
It remains to show $v_1 v_2 v_3 v_4 = \iota$ for $v = [v_1, v_2, v_3, v_4]$. We have
$$v_1 v_2 v_3 v_4 = \sigma_1^{-2} \sigma_3 \sigma_1^2 \sigma_1^{-1} \sigma_3 \sigma_1 
\sigma_3 \sigma_1^{-3} \sigma_3 \sigma_1^3 =$$
$$\sigma_1^{-2} \sigma_3 \sigma_1 \sigma_3 \sigma_1 \sigma_3 \sigma_1 \sigma_3 \sigma_1 \sigma_1^2 =
\sigma_1^{-2} (\sigma_3 \sigma_1)^4 \sigma_1^2 = \iota$$
due to $\sigma_2^{-1} = \sigma_3 \sigma_1$ and $\sigma_2^4 = \iota$.
\end{proof}
\noindent
Note that in \cite[III, § 2, Satz 5(d)]{RefMat1987} $G = \langle \sigma_1^{-1} \sigma_3 \sigma_1, \sigma_3 \rangle$
is assumed, which is stronger than $G = \langle \sigma_1^{-2} \sigma_3 \sigma_1^2, \sigma_1^{-1} \sigma_3 \sigma_1,
\sigma_3 \rangle$.
\par
\noindent
Is $\{ v \}^{H_4}$ a generic $H_4$-orbit that splits into generic $B_4$-orbits? The answer is no, see 
Example~\ref{A7TransformationFrom4A4A7A} and Example~\ref{S6TransformationFrom4A4B6A}, which show different
values for $\vert \{ v \}^{H_4} \vert$.
\begin{example}
\label{A7TransformationFrom4A4A7A}
Let $G = A_7$.
Then for $[\underline{\sigma}]=$
$$[(1)(2,3)(4,5,6,7),(1,2)(3,4,6,7)(5),(1,3,6,7,5,4,2)] \in \Sigma^i(4A,4A,7A)$$
all assumptions from Proposition~\ref{FixedPointsOfAlpha40} are satisfied and we get $\vert \{ v \}^{H_4} \vert = 126$.
The $H_4$-orbit $\{ v \}^{H_4}$ splits into three $B_4$-orbits $\{ v \}^{B_4}$ of length $42$.
The set $\Sigma^i(7A,7A,7A,7A)$ splits under the action of $H_4$ into orbits of length
$$6_2, 21_2, 120, 126_2, 252, 280, 300_2, 308, 432_2$$
and under the action of $B_4$ into orbits of length
$$2_6, 21_2, 42_6, 84_3, 120, 280, 300_2, 308, 432_2.$$
\end{example}
\begin{example}
\label{S6TransformationFrom4A4B6A}
Let $G = S_6$.
Then for $[\underline{\sigma}]=$
$$[(1,2)(3,4,6,5),(1,5,6,3)(2)(4),(1,5,2)(3,4)(6)] \in \Sigma^i(4A,4B,6A)$$
all assumptions from Proposition~\ref{FixedPointsOfAlpha40} are satisfied and we get $\vert \{ v \}^{H_4} \vert = 216$.
The $H_4$-orbit $\{ v \}^{H_4}$ splits into three $B_4$-orbits $\{ v \}^{B_4}$ of length $72$.
The set $\Sigma^i(6A,6A,6A,6A)$ splits under the action of $H_4$ into orbits of length
$$36, 60, 120, 216, 360$$
and under the action of $B_4$ into orbits of length
$$12_3, 20_3, 72_3, 120, 360.$$
\end{example}
\noindent
Proposition~\ref{FixedPointsOfAlpha40} can be used to construct $B_4$-orbits $Z$ in $\Sigma^i(C,C,C,C)$.
Due to $\varphi_{4,2} = (\beta_2 \beta_3 \beta_4)^2 = \alpha_{4,0}^2$ we get $v^{\varphi_{4,2}} = v^{\alpha_{4,0}^2} = v$.
Thus $\{ v \}^{B_4}$ is a $B_4$-orbit with invariant $F_{B_4}(F_{4,2}, v) = 1$.
Using Theorem~\ref{BraidOrbitsOfSize2}(b) and Theorem~\ref{TranslationOfBraidOrbits}(c) 
leads to the same $B_4$-orbits $Z$ in $\Sigma^i(C,C,C,C)$.

\section{Class vectors of dimension 4 in $M_{23}$}\label{SectionClassVectorsOfDimension4InM23}

In \cite{RefHae1987} and \cite{RefHae1991}, for some class vectors of length $4$ of $M_{23}$
the braid orbits and braid genera have been computed. No orbits $Z$ resp. $Z^{sy}$ with $g_Z=0$ or $g_{Z^{sy}}=0$
have been found.
Computation of braid orbits for several simple groups showed, that rational symmetric class vectors $(C,C,C,C)$
seem to be suitable candidates, because $\Sigma^i(C,C,C,C)$ often splits into multiple $B_4$-orbits,
some of them rigid, small and with small genus.
For these $B_4$-orbits $Z$ almost always $ \vert Z \vert  \leq l^i(C,C,2,2)$ is satisfied.
See also section \ref{SectionHeuristicsForSearching} for heuristics how to find suitable class vectors.
\par
\noindent
Interesting class vectors for $M_{23}$ are 
$$(3A,3A,3A,3A),(4A,4A,4A,4A),(5A,5A,5A,5A),$$
$$(6A,6A,6A,6A),(8A,8A,8A,8A).$$
Table \ref{tab:Data_4DimClassvectorsM23} and table \ref{tab:Data_BraidOrbitsOfM23} contain details
for these class vectors.
\par
\noindent
For a class vector $(C_1,\dots, C_m)$, the {\it normalized structure constant} defined in
\cite[I, 5.3 (5.2)]{RefMM2018} is denoted by $n(C_1,\dots, C_m)$.
With the {\it floor}-function $\lfloor . \rfloor$, we have
$$l^i(C_1,\dots, C_m) \leq \lfloor n(C_1,\dots, C_m) \rfloor$$
due to the definition of $n(C_1,\dots, C_m)$, and \cite[I, Corollary 5.6]{RefMM2018}.
By \cite[I, Theorem 5.8]{RefMM2018} $n(C_1,\dots, C_m)$ can be computed from the character table of $G$ via
$$
n(C_1,\dots, C_m)=\frac{{\vert Z(G) \vert}\cdot{\vert G \vert}^{m-2}}
{\prod_{i=1}^m {\vert C_G(\sigma_i) \vert}} \sum\limits_{\chi \in Irr(G)}
\frac{\prod_{i=1}^m \chi(\sigma_i)}{\chi(\iota)^{m-2}}
$$
with $\sigma_i \in C_i$ for $i=1,\dots, m$. Here, $C_G(\sigma)$ is the centralizer of $\sigma$ in $G$ and
$Irr(G)$ the set of complex irreducible characters of $G$.
We take character tables from \cite{RefAtlas1985}.
\begin{table}[!htbp]
\centering
\footnotesize
\captionsetup{font=footnotesize}
\caption{Some class vectors of length $3$ and $4$ of $M_{23}$}
\label{tab:Data_4DimClassvectorsM23}
\begin{tabular}{lrrl}
\hline\noalign{\smallskip}
$C$ & $l^i(C)$ & $\lfloor n(C) \rfloor$ & Ref.\\
\noalign{\smallskip}\hline\noalign{\smallskip}
$(2A,2A,2A)$ & 0 & 0 & \cite{RefHae1987}\\
$(3A,3A,2A)$ & 0 & 1 & \cite{RefHae1987}\\
$(4A,4A,2A)$ & 0 & 8 & \cite{RefHae1987}\\
$(5A,5A,2A)$ & 0 & 40 & \cite{RefHae1987}\\
$(6A,6A,2A)$ & 0 & 34 & \cite{RefHae1987}\\
$(8A,8A,2A)$ & 28 & 61 & \cite{RefHae1987}\\
\noalign{\smallskip}\hline\noalign{\smallskip}
$(2A,2A,2A,2A)$ & 0 & 19 & \cite{RefHae1987}\\
$(3A,2A,2A,2A)$ & 0 & 112 & \cite{RefHae1987}\\
$(4A,2A,2A,2A)$ & 0 & 378 & \cite{RefHae1987}\\
$(5A,2A,2A,2A)$ & 0 & 658 & \cite{RefHae1987}\\
$(6A,2A,2A,2A)$ & 0 & 547 & \cite{RefHae1987}\\
$(8A,2A,2A,2A)$ & 0 & 664 & \cite{RefHae1987}\\
\noalign{\smallskip}\hline\noalign{\smallskip}
$(3A,3A,2A,2A)$ & 0 & 1,105 & \cite{RefHae1987}\\
$(4A,4A,2A,2A)$ & 2,456 & 19,048 & \cite{RefHae1987}, \cite{RefHae1991}\\
$(5A,5A,2A,2A)$ & 30,400 & 83,906 & \cite{RefHae1987}, \cite{RefHae1991}\\
$(6A,6A,2A,2A)$ & 72,528 & 101,811 & \cite{RefHae1991}\\
$(8A,8A,2A,2A)$ & 198,488 & 228,399 & \cite{RefHae1991}\\
\noalign{\smallskip}\hline\noalign{\smallskip}
$(3A,3A,3A,2A)$ & 996 & 12,352 & \cite{RefHae1987}, \cite{RefHae1991}\\
$(4A,4A,4A,2A)$ & 768,528 & 1,298,855 & -\\
$(5A,5A,5A,2A)$ & 8,569,480 & 12,593,829 & -\\
$(6A,6A,6A,2A)$ & 21,846,222 & 22,398,281 & -\\
$(8A,8A,8A,2A)$ & 74,725,280 & 75,791,456 & -\\
\noalign{\smallskip}\hline\noalign{\smallskip}
$(3A,3A,3A,3A)$ & 37,296 & 155,711 & \cite{RefHae1991}\\
$(4A,4A,4A,4A)$ & 85,607,040 & 102,663,428 & -\\
$(5A,5A,5A,5A)$ & 1,811,840,784 & 2,124,083,306 & -\\
$(6A,6A,6A,6A)$ & 5,004,867,456 & 5,020,189,817 & -\\
$(8A,8A,8A,8A)$ & 25,382,467,456 & 25,432,112,544 & -\\
\noalign{\smallskip}\hline
\end{tabular}
\end{table}
\begin{table}[!htbp]
\centering
\footnotesize
\captionsetup{font=footnotesize}
\caption{Some braid orbits of $M_{23}$}
\label{tab:Data_BraidOrbitsOfM23}
\begin{tabular}{llll}
\hline\noalign{\smallskip}
$C$ & $ \vert Z \vert $      & $g_Z$ & \\
    & $ \vert Z^{sy} \vert $ & $g_{Z^{sy}}$ & \\
\noalign{\smallskip}\hline\noalign{\smallskip}
$(4A,4A,2A,2A)$ & 2,456 & 447 & \\
                & 2,456 & 169 & \\
$(5A,5A,2A,2A)$ & 30,400 & 5,223 & \\
                & 30,400 & 2,071 & \\
$(6A,6A,2A,2A)$ & 72,528 & 15,125 & \\
                & 72,528 & 6,065 & \\
$(8A,8A,2A,2A)$ & 198,488 & 41,277 & \\
                & 198,488 & 16,617 & \\
\noalign{\smallskip}\hline\noalign{\smallskip}
$(3A,3A,3A,2A)$ & 996 & 193 & \\
                & 996 & 33 & \\
$(4A,4A,4A,2A)$ & 768,528 & 212,497 & \\
                & 768,528 & 35,217 & \\
$(5A,5A,5A,2A)$ & 8,569,480 & 2,522,100 & \\
                & 8,569,480 & 419,486 & \\
$(6A,6A,6A,2A)$ & ? & ? & \\
                & ? & ? & \\
$(8A,8A,8A,2A)$ & ? & ? & \\
                & ? & ? & \\
\noalign{\smallskip}\hline\noalign{\smallskip}
$(3A,3A,3A,3A)$ & 8,316, 28,980 & 2,029, 7,525 & \\
                & 8,316, 28,980 & 325, 1,206 & \\
$(4A,4A,4A,4A)$ & $2,456_3$ & $459_3$ & *\\
                & 7,368     & 169 & *\\
$(5A,5A,5A,5A)$ & $30,400_3$ & $5,626_3$ & *\\
                & 91,200 & 2,071 & *\\
$(6A,6A,6A,6A)$ & $72,528_3$ & $14,380_3$ & *\\
                & 217,584 & 6,065 & *\\
$(8A,8A,8A,8A)$ & $2_{84}$, $198,488_3$ & $0_{84}$, $39,505_3$ & *\\
                & $6_{28}$, 595,464 & $0_{28}$, 16,617 & *\\
\noalign{\smallskip}\hline
\end{tabular}
\end{table}
\par
\noindent
The $84$ orbits of length $2$ in $\Sigma^i(8A,8A,8A,8A)$ come from Theorem~\ref{BraidOrbitsOfSize2} while the $28$
orbits of length $6$ come form Theorem~\ref{BraidOrbitsOfSize6}, both due to $l^i(8A,8A,2A)=28$.
Theorem~\ref{BraidOrbitsOfSize12} leads to $14$ $H_4$-orbits of length $12$ with
$\rho_4(H_4) \cong (C_2 \times C_2):S_4 < A_{12}$.
\par
\noindent
The $B_4$-action for $(3A,3A,3A,3A)$ could be computed completely. For the remaining class vectors,
we tried to find small orbits by iterating through all $4$-systems 
$[\underline{\sigma}] \in \Sigma^i(C_1,C_2,C_3,C_4)$ and checking whether the length of the $B_4$-orbit 
of $[\underline{\sigma}]$ lies under a predefined limit $l$, i.e. we computed the set 
$$B_4((C_1,C_2,C_3,C_4),l) = 
\{ [\underline{\sigma}] \in \Sigma^i(C_1,C_2,C_3,C_4) \mid \vert [\underline{\sigma}]^{B_4} \vert \leq l \}$$
which is a union of $B_4$-orbits by definition. As seen above, $l=l^i(C,C,2,2)$ would be an interesting limit
for $(C,C,C,C)$, but the set $B_4((C_1,C_2,C_3,C_4),l)$ is hard to compute for big values of $l$.
Table \ref{tab:CurrentStateOfComputations} gives the current state of the computations for $M_{23}$.
\begin{table}[!htbp]
\centering
\footnotesize
\captionsetup{font=footnotesize}
\caption{Current state of the computations for $M_{23}$}
\label{tab:CurrentStateOfComputations}
\begin{tabular}{lrrl}
\hline\noalign{\smallskip}
Class vector & $l$ & $ \vert B_4((C_1,C_2,C_3,C_4),l) \vert $ & $B_4$-orbits \\
\noalign{\smallskip}\hline\noalign{\smallskip}
$(4A,4A,4A,4A)$ & 2,500 & 7,368 & $2,456_3$ \\
$(5A,5A,5A,5A)$ & 250 & 0 & - \\
$(6A,6A,6A,2A)$ & 4 & 0 & - \\
$(6A,6A,6A,6A)$ & 2 & 0 & - \\
$(8A,8A,8A,2A)$ & 4 & 0 & - \\
$(8A,8A,8A,8A)$ & 2 & 168 & $2_{84}$ \\
\noalign{\smallskip}\hline
\end{tabular}
\end{table}
\par
\noindent
Thus $\Sigma^i(4A,4A,4A,4A)$ has three $B_4$-orbits of length $2,456$ and no smaller $B_4$-orbits.
These orbits come from $(4A,4A,2A,2A)$ by Theorem~\ref{TranslationOfBraidOrbits}(b).
Due to $g_Z=459$ and $g_{Z^{sy}}=169$ for the corresponding orbit $Z^{sy}$ with size 7,368,
we can rule out the class vector $(4A,4A,4A,4A)$ because orbits $Z$ with $ \vert Z \vert  \geq 2,456$
will have $g_Z \geq 459$ and $g_{Z^{sy}} \geq 169$ with very high probability.
\par
\noindent
It remains to compute the cardinalities of the sets
$B_4((5A,5A,5A,5A),30,400)$,
$B_4((6A,6A,6A,6A),72,528)$ and
$B_4((8A,8A,8A,8A),198,488)$.
Either these class vectors can be ruled out similar to $(4A,4A,4A,4A)$, or small $B_4$-orbits occur
that may allow geometric Galois realizations of $M_{23}$ over $\mathbb{Q}(v,t)$.
In the age of cloud computing, it seems possible to perform these computations as they can be
easily done in parallel.
Looking at non-rational class vectors of $M_{23}$, we can close this section with
\begin{theorem}
\label{M23AsGaloisGroupOverQSqrt}
For the Mathieu group $M_{23}$ exist two geometric Galois extensions, namely $N_1/\mathbb{Q}(\sqrt{-7})(v,t)$
and $N_2/\mathbb{Q}(\sqrt{-15})(v,t)$ with 
$$Gal(N_1/\mathbb{Q}(\sqrt{-7})(v,t)) \cong Gal(N_2/\mathbb{Q}(\sqrt{-15})(v,t)) \cong M_{23}$$
and corresponding class vectors $(14A,2A,2A,2A)$ resp. $(15A,2A,2A,2A)$.
In particular, $M_{23}$ occurs as Galois group over $\mathbb{Q}(\sqrt{-7})$ and $\mathbb{Q}(\sqrt{-15})$.
\end{theorem}
\begin{proof}
For the class vector $(14A,2A,2A,2A)$ of $M_{23}$ we have 
\par
\noindent
$l^i(14A,2A,2A,2A)=84$ and $\Sigma^i(14A,2A,2A,2A)$ is a single $B_4$-orbit. 
Using the group of symmetries $S = \langle (2,3), (2,3,4) \rangle$ and
$$\rho^t_4(\beta_3)=(3)(4)^3(5)^3(6)^9, \rho^t_4(\eta_{23})=(2)^{42}, \rho^t_4(\eta_{234})=(3)^{28},$$
we get $g_{Z_1^{sy}} = 0$ for $Z_1^{sy} = \Sigma^i(14A,2A,2A,2A)$.
\par
\noindent
For the class vector $(15A,2A,2A,2A)$ of $M_{23}$ we have $l^i(15A,2A,2A,2A)=90$ and $\Sigma^i(15A,2A,2A,2A)$
is a single $B_4$-orbit. 
Again using the group of symmetries $S = \langle (2,3), (2,3,4) \rangle$ and
$$\rho^t_4(\beta_3)=(3)(4)^2(5)^5(6)^9, \rho^t_4(\eta_{23})=(2)^{45}, \rho^t_4(\eta_{234})=(3)^{30},$$
we get $g_{Z_2^{sy}} = 0$ for $Z_2^{sy} = \Sigma^i(15,2A,2A,2A)$.
Now all assumptions from \cite[III, Theorem 7.10]{RefMM2018} are satisfied for $Z_1^{sy}$, $Z_2^{sy}$
and Theorem~\ref{M23AsGaloisGroupOverQSqrt} follows.
\end{proof}
\noindent
The class vectors in Theorem~\ref{M23AsGaloisGroupOverQSqrt} are the first ones yielding a direct geometric 
realization of $M_{23}$ as Galois group over rational function fields $k(v,t)$ with $(k:\mathbb{Q}) \leq 2$. 
The geometric $M_{23}$-Galois extensions over $\mathbb{Q}(\sqrt{-23})(t)$ and $\mathbb{Q}(\sqrt{-7})(t)$
in \cite{RefHS1985} and \cite{RefHae1987} have been deduced from $M_{24}$-Galois extensions.

\section{Class vectors of dimension 4 in $M_{24}$}\label{SectionClassVectorsOfDimension4InM24}

In \cite{RefHae1987} and \cite{RefHae1991}, the braid orbits and braid genera for all rational class 
vectors of length $4$ in $M_{24}$ with $g_{M_{23}} \leq 1$ have been computed.
\par
\noindent
Only one class vector with $g_Z \leq 1$ for a $B_4$-orbit $Z$, namely $(12B,2A,2A,2A)$
with $l^i(12B,2A,2A,2A)=144$, was found.
Here, $Z_1 = \Sigma^i(12B,2A,2A,2A)$ is a rigid $B_4$-orbit with
$$\rho^t_4(\beta_{12})=\rho^t_4(\beta_{13})=\rho^t_4(\beta_{14})=(2)^6(3)^{39}(5)^3,$$
$$\rho^t_4(\beta_3)=\rho^t_4(\beta_4)=(3)(4)^3(5)^3(6)^{19},$$
$$\rho^t_4(\eta_{23})=\rho^t_4(\eta_{34})=(2)^{72}, \rho^t_4(\eta_{234})=(3)^{48},$$
$g_{Z_1}=1$ and $g_{Z^{sy}_1}=0$ for all non-trivial groups of symmetries $S(12B,2A,2A,2A)$.
This leads to geometric Galois realizations $N_1/\mathbb{Q}(v,t)$ with $Gal(N_1/\mathbb{Q}(v,t)) \cong M_{24}$
for all groups of symmetries, see
\cite[Nachtrag]{RefHae1987}, \cite[Beispiel 1.1]{RefHae1991}, \cite[Satz 9.4]{RefMat1991}
and \cite[III, Theorem 7.12]{RefMM2018}.
We have $g_{M_{23}}=0$, but $N^{M_{23}}_1$ is not a rational function field, see \cite{RefGra1996}.
\par
\noindent
Applying Theorem~\ref{TranslationOfBraidOrbits}(b) to $Z_1 = \Sigma^i(12B,2A,2A,2A)$ generates a rigid $B_4$-orbit $Z_2$
of length $ \vert Z_2 \vert =144$ in $\Sigma^i(12B,12B,2A,2A)$ and with
$$\rho^t_4(\beta_{12})=(1)^{12}(3)^{39}(5)^3, \rho^t_4(\beta_{13})=\rho^t_4(\beta_{14})=(3)(4)^3(5)^3(6)^{19},$$
$$\rho^t_4(\beta_4)=(2)^6(3)^{39}(5)^3, \rho^t_4(\eta_{34})=(2)^{72},$$
$g_{Z_2}=20$ and $g_{Z^{sy}_2}=0$ for $S(12B,12B,2A,2A) = \langle (3,4) \rangle$.
This leads to a geometric Galois realization
$N_2/\mathbb{Q}(v,t)$ with $Gal(N_2/\mathbb{Q}(v,t)) \cong M_{24}$ and $g_{M_{23}}=7$.
\par
\noindent
Applying Theorem~\ref{TranslationOfBraidOrbits}(b) once again to $Z_2$ generates a rigid 
$B_4$-orbit $Z_3$ of length $ \vert Z_3 \vert =144$ in $\Sigma^i(12B,12B,12B,12B)$ and with
$$\rho^t_4(\beta_{12})=\rho^t_4(\beta_{13})=\rho^t_4(\beta_{14})=(2)^6(3)^{39}(5)^3,$$
$$\rho^t_4(\beta_3)=\rho^t_4(\beta_4)=(3)(4)^3(5)^3(6)^{19},$$
$$\rho^t_4(\eta_{23})=\rho^t_4(\eta_{34})=(2)^{72}, \rho^t_4(\eta_{234})=(3)^{48},$$
$g_{Z_3}=1$ and $g_{Z^{sy}_3}=0$ for all non-trivial groups of symmetries
\par
\noindent
$S(12B,12B,12B,12B)$ fixing the first class.
This leads to geometric Galois realizations 
$N_3/\mathbb{Q}(v,t)$ with $Gal(N_3/\mathbb{Q}(v,t)) \cong M_{24}$ and $g_{M_{23}}=21$
for all groups of symmetries fixing the first class.
With 
\begin{itemize}
\item[]
$l^i(12B,2A,2A,2B)=l^i(12B,2A,2B,2A)=1,224,$
\item[]
$l^i(12B,2A,2B,2B)=3,072$
\end{itemize}
and Theorem~\ref{TranslationOfBraidOrbits}(b) we can recognize
the two braid orbits of length $1,224+1,224=2,448$ and $3,072$ in $\Sigma^i(12B,12B,2A,2A)$ as translates
originating from the class vectors 
$$(12B,2A,2A,2B),(12B,2A,2B,2A),(12B,2A,2B,2B).$$
\par
\noindent
Applying Theorem~\ref{TranslationOfBraidOrbits}(b) to the four braid orbits in $\Sigma^i(12B,12B,2A,2A)$, we get four
braid orbits of length $144, 2,448, 3,072$ and $982,102$ in the set $\Sigma^i(12B,12B,12B,12B)$.
\par
\noindent
This procedure can be repeated starting from class vectors $(12B,2B,.,.)$ resp. $(12B,12B,.,.)$ with involution classes
$2A$ or $2B$ on the empty positions.
\par
\noindent
Using Theorems \ref{BraidOrbitsOfSize2} and \ref{BraidOrbitsOfSize6}, $l^i(12B,12B,2A)=75$ and $l^i(12B,12B,2B)=233$,
we get $75 \cdot 3 + 233 \cdot 3 = 225+699$ $B_4$-orbits $Z$ of length $2$ 
and $75+233$ $B^{sy}_4$-orbits $Z^{sy}$ of length $6$ in $\Sigma^i(12B,12B,12B,12B)$.
If one of these orbits $Z^{sy}$ would be rigid, a further Galois extension 
of $M_{24}$ over $\mathbb{Q}(v,t)$ could be deduced.
\par
\noindent
Tables \ref{tab:Data_4DimClassvectorsM24} and \ref{tab:Data_BraidOrbitsOfM24} contain more details for the
class vectors in this section.
For abbreviation we use $(12B,2A_3)=(12B,2A,2A,2A)$, $(12B_2,2A_2)=(12B,12B,2A,2A)$ and $(12B_4)=(12B,12B,12B,12B)$.
\begin{table}[!htbp]
\centering
\footnotesize
\captionsetup{font=footnotesize}
\caption{Some class vectors of length $4$ of $M_{24}$}
\label{tab:Data_4DimClassvectorsM24}
\begin{tabular}{lrrrl}
\hline\noalign{\smallskip}
$C$ & $l^i(C)$ & $\lfloor n(C) \rfloor$ & $g_{M_{23}}$ & Ref.\\
\noalign{\smallskip}\hline\noalign{\smallskip}
$(12B,2A_3)$ & 144 & 180 & 0 & \cite{RefHae1987}, \cite{RefHae1991}, \cite{RefMat1991}\\
$(12B_2,2A_2)$ & 987,766 & 995,938 & 7 & \cite{RefHae1991}\\
$(12B_4)$ & ? & 2,898,930,176,000 & 21 & -\\
\noalign{\smallskip}\hline
\end{tabular}
\end{table}
\begin{table}[!htbp]
\centering
\footnotesize
\captionsetup{font=footnotesize}
\caption{Some braid orbits of $M_{24}$}
\label{tab:Data_BraidOrbitsOfM24}
\begin{tabular}{llll}
\hline\noalign{\smallskip}
$C$ & $ \vert Z \vert $      & $g_Z$ & \\
    & $ \vert Z^{sy} \vert $ & $g_{Z^{sy}}$ & \\
\noalign{\smallskip}\hline\noalign{\smallskip}
$(12B,2A_3)$   & 144 & 1 & \\
               & 144 & 0 & \\
\noalign{\smallskip}\hline\noalign{\smallskip}
$(12B_2,2A_2)$ & 144, 2,448, 3,072, 982,102 & 20, 364, 484, 224,490 & \\
               & 144, 2,448, 3,072, 982,102 & 0, 85, 159, 95,861 & \\
\noalign{\smallskip}\hline\noalign{\smallskip}
$(12B_4)$      & $2_{225}$, $2_{699}$, 144, 2,448, 3,072, 982,102 & $0_{225}$, $0_{699}$, 1, ?, ?, ? & *\\
               & $6_{75}$, $6_{233}$, 144, 2,448, 3,072, 982,102 & $0_{75}$, $0_{233}$, 0, ?, ?, ? & *\\
\noalign{\smallskip}\hline
\end{tabular}
\end{table}
\par
\noindent
Summarized, we get
\begin{theorem}
\label{M24AsGaloisGroup}
(\cite{RefHae1987}, \cite{RefHae1991}, \cite{RefMat1991},\cite{RefMM2018})
There exist three geometric Galois extensions $N_i/\mathbb{Q}(v,t), i=1,2,3$ with 
$$Gal(N_i/\mathbb{Q}(v,t)) \cong M_{24}$$
and corresponding class vectors 
$$(12B,2A,2A,2A), (12B,12B,2A,2A), (12B,12B,12B,12B).$$
In particular, $M_{24}$ occurs as Galois group over $\mathbb{Q}$.
\end{theorem}
\noindent
See \cite{RefHae1987}, \cite{RefMat1991} and \cite{RefMM2018} for $N_1$ and \cite{RefHae1991} for $N_2$.

\section{Small orbits in dimension 6}\label{SectionSmallOrbitsInDimension6}

Let $G$ be a finite group with $Z(G)=I$ and $n \geq 2$.
For $n=2$, using Theorem~\ref{BraidOrbitsOfSize2}(b), we can construct
$B_4$-orbits $Z_2 = [\sigma_1, \sigma_2, \sigma_1, \sigma_2]^{B_4}$ of size $2$ in $\Sigma^i(C_1,C_2,C_1,C_2)$ 
from $B_3$-orbits $Z_1 = [\sigma_1, \sigma_2, \sigma_3]^{B_3}$ in $\Sigma^i(C_1,C_2,2A)$ with length $1$.
Note that $B_3$ acts trivial on $\Sigma^i(C_1,C_2,C_3)$, thus $B_3$-orbits $Z_1$ always have length $1$
and we can use each $[\underline{\sigma}] \in \Sigma^i(C_1,C_2,2A)$.
For this construction, fixed points under $F_{4,2}$ have been used.
\par
\noindent
Now let $m=2n$ and $[\underline{\sigma}] = [\sigma_1, \dots, \sigma_m] \in \Sigma^i(C_1,\dots,C_m)$ be a fixed point under 
the action of $F_{m,2} = \langle \varphi_{m,2} \rangle$. Then we have
$$[\sigma_1, \dots, \sigma_{2n}] = [\sigma_1, \dots, \sigma_{2n}]^{\varphi_{m,2}} =
[\sigma_{n+1}, \dots,\sigma_{2n},\sigma_1, \dots, \sigma_n]$$
by Proposition~\ref{PropertiesOfFm}(c.2). Thus there exists an element $\tau \in G$ with
\begin{equation} 
\label{eq:ConjugationTau_Part1}
\tau^{-1} \sigma_k \tau = \sigma_{n+k}, k=1,\dots,n
\end{equation}
\begin{equation}
\label{eq:ConjugationTau_Part2}
\tau^{-1} \sigma_{n+k} \tau = \sigma_k, k=1,\dots,n
\end{equation}
At once we get $C_k = C_{n+k}$ for $k=1,\dots,n$.
Inserting (\ref{eq:ConjugationTau_Part1}) in (\ref{eq:ConjugationTau_Part2}) gives
$$\sigma_k = \tau^{-2} \sigma_k \tau^2, k=1,\dots,n$$
and inserting (\ref{eq:ConjugationTau_Part2}) in (\ref{eq:ConjugationTau_Part1}) results in
$$\sigma_{n+k} = \tau^{-2} \sigma_{n+k} \tau^2, k=1,\dots,n$$
Thus $\tau^2 \in Z(G)$, $o(\tau) = 1$ or $o(\tau) = 2$ and
\begin{equation} 
\label{eq:FixedpointFm2Sigma}
[\underline{\sigma}] = [\sigma_1, \dots, \sigma_n, \tau^{-1} \sigma_1 \tau, \dots, \tau^{-1} \sigma_n \tau]
\end{equation}
In this section we look at $o(\tau) = 1$, while $o(\tau) = 2$ will not be treated here.
With $\tau = \iota$, (\ref{eq:FixedpointFm2Sigma}) becomes
$$[\underline{\sigma}] = [\sigma_1, \dots, \sigma_n, \sigma_1, \dots, \sigma_n] \in \Sigma^i(C_1,\dots,C_n,C_1,\dots,C_n)$$
and we get $G = \langle \sigma_1, \dots, \sigma_{2n} \rangle = \langle \sigma_1, \dots, \sigma_n \rangle$ 
and $(\sigma_1 \cdots \sigma_n)^2 = \sigma_1 \cdots \sigma_{2n} = \iota$.
Now we have a method how to construct fixed points $[\underline{\sigma}] \in \Sigma^i(C_1,\dots,C_m)$ under the
action of $F_{m,2}$ from systems in $\Sigma^i(C_1,\dots,C_n)$ or $\Sigma^i(C_1,\dots,C_n,D)$ with
a class of involutions $D$.
Next, we are interested in the $B_m$-orbit of $[\sigma_1, \dots, \sigma_n, \sigma_1, \dots, \sigma_n]$.
\par
\noindent
Let $(C_1,\dots,C_n,D)$ be a class vector of length $n+1$ of $G$ with a class of
involutions $D=2A$ or with $D=1A$ at position $n+1$ (here we allow a trivial class in a class vector).
For 
$$[\underline{\sigma}] = [\sigma_1, \dots, \sigma_n, \sigma_{n+1}] \in \Sigma^i(C_1,\dots,C_n,D),$$ 
we have
$$[\widehat{\underline{\sigma}}] = [\sigma_1, \dots, \sigma_n, \sigma_1, \dots, \sigma_n] \in \Sigma^i(C_1,\dots,C_n,C_1,\dots,C_n)$$
because $\sigma_{n+1}$ is an involution or $\sigma_{n+1} = \iota$.
By construction, we have $[\widehat{\underline{\sigma}}]^{\varphi_{m,2}} = [\widehat{\underline{\sigma}}]$.
\begin{proposition}
\label{TransformedObjectsDifferent}
Let $n \geq 2$.
If $[\underline{\sigma}], [\underline{\tau}] \in \Sigma^i(C_1,\dots,C_n,D)$
with $[\underline{\sigma}] \neq [\underline{\tau}]$ then 
$[\widehat{\underline{\sigma}}] \neq [\widehat{\underline{\tau}}]$.
\end{proposition}
\begin{proof}
This comes from $G = \langle \sigma_1, \dots, \sigma_n \rangle$.
\end{proof}
\begin{proposition}
\label{TransformationOfOrbits}
Let $2 \leq n \leq 3$.
If $[\underline{\tau}] \in [\underline{\sigma}]^{B_{n+1}}$, then 
$[\widehat{\underline{\tau}}] \in [\widehat{\underline{\sigma}}]^{B_{2n}}$.
\end{proposition}
\begin{proof}
For $n=2$, we have $[\underline{\tau}] \in [\underline{\sigma}]^{B_3} = \{[\underline{\sigma}]\}$, thus
$[\underline{\tau}] = [\underline{\sigma}]$ and $[\widehat{\underline{\tau}}] \in [\widehat{\underline{\sigma}}]^{B_4}$.
\par
\noindent
Now let $n=3$ and $[\underline{\tau}] \in [\underline{\sigma}]^{B_4}$. Then there exists a $\beta \in B_4$
with $[\underline{\tau}] = [\underline{\sigma}]^{\beta}$.
From Corollary~\ref{ActionOfPureHurwitzGroupByBeta1j} we know that the action of $B_4$ is already determined by
the action of $\beta_{12}$ and $\beta_{13}$. Thus without restriction, we can assume
$$\beta = \prod\limits_{i_{12}, i_{13} \in \{-1,0,1\} }^{finite} \beta_{12}^{i_{12}} \beta_{13}^{i_{13}} \in B_4.$$
Again without restriction, we can normalize the action of $\beta_{12}$ and $\beta_{13}$ to look like
$$[\underline{\sigma}]^{\beta_{12}} = [., ., ., \sigma_4] = [\underline{\sigma}^{\beta_{12}}]$$ and
$$[\underline{\sigma}]^{\beta_{13}} = [., ., ., \sigma_4] = [\underline{\sigma}^{\beta_{13}}] $$
because $\beta_{12}$ and $\beta_{13}$ act by conjugation, see for example Proposition~\ref{ActionOfPureHurwitzGroup}.
Finally we get
$$[\underline{\tau}] = [\underline{\sigma}]^{\beta} = [., ., ., \sigma_4] = [\underline{\sigma}^{\beta}].$$
With $\tilde{\underline{\sigma}} = (\sigma_1, \sigma_2, \sigma_3)$ and
$$[\widehat{\underline{\sigma}}] = [\sigma_1, \sigma_2, \sigma_3, \sigma_1, \sigma_2, \sigma_3] =
[\tilde{\underline{\sigma}}, \tilde{\underline{\sigma}}]$$
we have
$$[\widehat{\underline{\tau}}] = [\tilde{\underline{\sigma}}^{\beta}, \tilde{\underline{\sigma}}^{\beta}].$$
Setting $\widehat{\beta}_{12} = \beta_5^2 \in B_6$, $\widehat{\beta}_{13} = \beta_5^{-1} \beta_6^2 \beta_5 \in B_6$ and 
$$\widehat{\beta} = \prod\limits_{i_{12}, i_{13} \in \{-1,0,1\} }^{finite} \widehat{\beta}_{12}^{i_{12}} \widehat{\beta}_{13}^{i_{13}} \in B_6$$
by replacing $\beta_{12}$ with $\widehat{\beta}_{12}$ and $\beta_{13}$ with $\widehat{\beta}_{13}$ in the product
representation of $\beta$, identifying $\beta_{12} \in B_4$ with $\beta_{12} \in B_6$ and
$\beta_{13} \in B_4$ with $\beta_{13} \in B_6$, we get
$$[\widehat{\underline{\sigma}}]^{\beta \widehat{\beta}} = [\tilde{\underline{\sigma}}^{\beta}, \tilde{\underline{\sigma}}]^{\widehat{\beta}} =
[\tilde{\underline{\sigma}}^{\beta}, \tilde{\underline{\sigma}}^{\beta}] = [\widehat{\underline{\tau}}]$$
with $\beta \widehat{\beta} \in B_6$.
Finally, $[\widehat{\underline{\tau}}] \in [\widehat{\underline{\sigma}}]^{B_6}$ follows.
\end{proof}
\noindent
In case of $n=3$, examples listed in tables \ref{tab:Data_B4OrbitsSiz4} and \ref{tab:Data_B6OrbitsSiz40} suggest 
to construct $B_6$-orbits $Z_{40} = [\sigma_1, \sigma_2, \sigma_3, \sigma_1, \sigma_2, \sigma_3]^{B_6}$ of 
length $40$ in $\Sigma^i(C_1,C_2,C_3,C_1,C_2,C_3)$ from $B_4$-orbits 
$Z_4 = [\sigma_1, \sigma_2, \sigma_3, \sigma_4]^{B_4}$ of length $4$ in $\Sigma^i(C_1,C_2,C_3,2A)$.
\par
\noindent
Usually, the $B_4$-orbit of $[\underline{\sigma}] \in \Sigma^i(C_1,C_2,C_3,2A)$ has not length $4$,
thus we may expect constraints on $[\underline{\sigma}]$ for getting $B_4$-orbits $Z_4$ of length $4$.
These constraints are the assumptions in Theorem~\ref{BraidOrbitsOfSize4}.
Table \ref{tab:PropertiesTransformationOfOrbits} contains the properties of the transformation of $B_{n+1}$-orbits to $B_{2n}$-orbits.
For $n=2$, see Theorem~\ref{BraidOrbitsOfSize2}, for $n=3$ see Theorem~\ref{BraidOrbitsOfSize4} and Theorem~\ref{BraidOrbitsOfSize40} below.
\begin{table}[!htbp]
\centering
\footnotesize
\captionsetup{font=footnotesize}
\caption{Properties of the transformation of $B_{n+1}$-orbits to $B_{2n}$-orbits}
\label{tab:PropertiesTransformationOfOrbits}
\begin{tabular}{lllllll}
\hline\noalign{\smallskip}
$n$ & $\vert [\underline{\sigma}]^{B_{n+1}} \vert$ & $\rho^t_{n+1}(\beta_{ij})$ & $\rho_{n+1}(B_{n+1})$ 
& $\vert [\widehat{\underline{\sigma}}]^{B_{2n}} \vert$ & $\rho^t_{2n}(\beta_{ij})$ & $\rho_{2n}(B_{2n})$\\
\noalign{\smallskip}\hline\noalign{\smallskip}
2   & 1 & $(1)$    & $I$   & 2  & $(1)^2$ or $(2)$ & $C_2$ \\
3   & 4 & $(1)(3)$ & $A_4 \cong S_2(3)$ & 40 & $(1)^{13}(3)^9$  & $S_4(3)$ \\
\noalign{\smallskip}\hline
\end{tabular}
\end{table}
\begin{example}
\label{L2_8_7A7A7A2A_Part2}
For the class vector $(7A,7A,7A,2A)$ of the linear group $L_2(8)$, the set
$\Sigma^i(7A,7A,7A,2A)$ contains a single $B_4$-orbit
$Z_4 = \{ h_1 \}^{B_4} = \{ h_1,h_2,h_3,h_4 \}$ of length $4$, see Example~\ref{L2_8_7A7A7A2A_Part1} and Theorem~\ref{BraidOrbitsOfSize4}.
The set 
$$Y_{40} = \{ [\widehat{\underline{\sigma}}] \in \Sigma^i(7A,7A,7A,7A,7A,7A) \mid 
[\underline{\sigma}] \in Z_4 \}^{B_6} = \{ \widehat{h_1} \}^{B_6}$$
with $h_1$ from Example~\ref{L2_8_7A7A7A2A_Part1} is a $B_6$-orbit of length $40$ in $\Sigma^i(7A,7A,7A,7A,7A,7A)$.
By direct computation, we get a splitting of $\Sigma^i(7A,7A,7A,7A,7A,7A)$ into two $B_6$-orbits
$Z_{40}$ and $Z_{540,015}$ of length $40$ and $540,015$. Thus we have $Y_{40} = Z_{40}$.
In addition, we have 
$$\rho^t_4(\beta_{ij}) = (1)(3), 1 \leq i < j \leq 4, \rho_4(B_4) \cong A_4 \cong S_2(3)$$
for $Z_4$ and
$$\rho^t_6(\beta_{ij}) = (1)^{13}(3)^9, 1 \leq i < j \leq 6, \rho_6(B_6) \cong S_4(3),$$
$$\rho^t_6(\beta_j) = (1)^{13}(3)^9, 2 \leq j \leq 6, \rho_6(H_6) \cong S_4(3)$$
for $Z_{40}$.
\end{example}
\noindent
Note that $S_4(3) \cong O_5(3) \cong U_4(2) \cong O_6^-(2)$.
Table \ref{tab:Data_B6OrbitsSiz40} contains examples of class vectors containing $B_6$-orbits $Z_{40}$
of length $\vert Z_{40} \vert = 40$.
\begin{table}[H]
\centering
\footnotesize
\captionsetup{font=footnotesize}
\caption{Class vectors containing $B_6$-orbits $Z_{40}$ with $g_{Z_{40}} = 6$}
\label{tab:Data_B6OrbitsSiz40}
\begin{tabular}{llllll}
\hline\noalign{\smallskip}
$G$       & $C$                         & $ \vert Z \vert $ & $\rho^t_6(\beta_{ij})$ & $\rho_6(B_6)$ & \\
          &                             &                   & on $Z_{40}$            & on $Z_{40}$ & \\
\noalign{\smallskip}\hline\noalign{\smallskip}
$A_5$     & $(5A_6)$       &  40, 975 & $(1)^{13} (3)^9$ & $S_4(3)$ & \\
$S_7$     & $(10A_6)$ &  40, ? & $(1)^{13} (3)^9$ & $S_4(3)$ & *\\
$L_2(7)$  & $(7A_6)$       &  40, 3,600, 3,885 & $(1)^{13} (3)^9$ & $S_4(3)$ & \\
$L_2(8)$  & $(7A_6)$       &  40, 540,015 & $(1)^{13} (3)^9$ & $S_4(3)$ & \\
$L_2(8)$  & $(9A_6)$       &  40, 124,335 & $(1)^{13} (3)^9$ & $S_4(3)$ & \\
$L_2(11)$ & $(11A_6)$ &  40, 53,952, 54,835 & $(1)^{13} (3)^9$ & $S_4(3)$ & \\
$L_2(13)$ & $(13A_6)$ &  40, 139,200, 174,265 & $(1)^{13} (3)^9$ & $S_4(3)$ & \\
$L_2(16)$ & $(15A_6)$ &  40, 24,274,575 & $(1)^{13} (3)^9$ & $S_4(3)$ & \\
$L_2(16)$ & $(17A_6)$ &  40, 11,511,567 & $(1)^{13} (3)^9$ & $S_4(3)$ & \\
$L_3(3)$  & $(13A_6)$ &  40, ? & $(1)^{13} (3)^9$ & $S_4(3)$ & *\\
$S_6(2)$  & $(15A_6)$ &  $40_2$, ? & $(1)^{13} (3)^9$ & $S_4(3)$ & *\\
$M_{12}$  & $(10A_6)$ &  $40_2$, ? & $(1)^{13} (3)^9$ & $S_4(3)$ & *\\
$M_{12}$  & $(11A_6)$ &  $40_2$, ? & $(1)^{13} (3)^9$ & $S_4(3)$ & *\\
\noalign{\smallskip}\hline
\end{tabular}
\end{table}
\noindent
The examples in tables \ref{tab:Data_B4OrbitsSiz4} and \ref{tab:Data_B6OrbitsSiz40} follow the pattern presented in Example~\ref{L2_8_7A7A7A2A_Part2}.
In table \ref{tab:Data_B6OrbitsSiz40} we always have $\rho^t_6(\beta_j) = (1)^{13}(3)^9, 2 \leq j \leq 6$ 
and $\rho_6(H_6) \cong S_4(3)$.
Table \ref{tab:Data_B3OrbitsSiz1} contains the corresponding class vectors $(C,C,C,1A) = (C,C,C)$.
\begin{table}[H]
\centering
\footnotesize
\captionsetup{font=footnotesize}
\caption{Class vectors $(C,C,C,1A) = (C,C,C)$}
\label{tab:Data_B3OrbitsSiz1}
\begin{tabular}{llr}
\hline\noalign{\smallskip}
$G$       & $C$             & $l^i(C)$ \\
\noalign{\smallskip}\hline\noalign{\smallskip}
$A_5$     & $(5A,5A,5A)$    & 1 \\
$S_7$     & $(10A,10A,10A)$ & 0 \\
$L_2(7)$  & $(7A,7A,7A)$    & 1 \\
$L_2(8)$  & $(7A,7A,7A)$    & 1 \\
$L_2(8)$  & $(9A,9A,9A)$    & 1 \\
$L_2(11)$ & $(11A,11A,11A)$ & 1 \\
$L_2(13)$ & $(13A,13A,13A)$ & 1 \\
$L_2(16)$ & $(15A,15A,15A)$ & 1 \\
$L_2(16)$ & $(17A,17A,17A)$ & 1 \\
$L_3(3)$  & $(13A,13A,13A)$ & 5 \\
$S_6(2)$  & $(15A,15A,15A)$ & 372 \\
$M_{12}$  & $(10A,10A,10A)$ & 68 \\
$M_{12}$  & $(11A,11A,11A)$ & 93 \\
\noalign{\smallskip}\hline
\end{tabular}
\end{table}
\begin{theorem}
\label{BraidOrbitsOfSize40}
Let $G>I$ be a finite group with $Z(G)=I$ and $\sigma_1, \sigma_2 \in G$ with
$G=\langle \sigma_1, \sigma_2 \rangle$,
$\sigma_1 \sigma_2 \sigma_1 = \sigma_2 \sigma_1 \sigma_2$,
$o(\sigma_1 \sigma_2 \sigma_1 )=2$, $C = [\sigma_1]$ and
$[\underline{\sigma}]$ be a generating $6$-system with
$[\underline{\sigma}] = [\sigma_1, \sigma_2, \sigma_1, \sigma_1, \sigma_2, \sigma_1]$.
Then
$$Z_{40} = [\underline{\sigma}]^{B_6} = [\underline{\sigma}]^{H_6} \subseteq \Sigma^i(C,C,C,C,C,C)$$
is an orbit of length $40$ under the action of $B_6$ and also under the action of $H_6$ with
\begin{itemize}
\item[]
$\rho^t_6(\beta_{ij}) = (1)^{13}(3)^9, 1 \leq i < j \leq 6$,
\item[]
$\rho_6(B_6) \cong S_4(3)$,
\item[]
$\rho^t_6(\beta_j) = (1)^{13}(3)^9, 2 \leq j \leq 6$,
\item[]
$\rho_6(H_6) \cong S_4(3)$
\end{itemize}
and genus $g_{Z_{40}}=6$.
\end{theorem}
\begin{proof} We have
$[\underline{\sigma}] = [\sigma_1, \sigma_2, \sigma_1, \sigma_1, \sigma_2, \sigma_1] = \widehat{h_1}$ with $h_1$ from Theorem~\ref{BraidOrbitsOfSize4},
thus $C = [\sigma_1] = [\sigma_2]$.
Theorem~\ref{BraidOrbitsOfSize40} is true for $G=L_2(8)$ with $\sigma_1$ and $\sigma_2$
from Example~\ref{L2_8_7A7A7A2A_Part1}, see also Example~\ref{L2_8_7A7A7A2A_Part2}.
Computing the action of $H_6$ on $\{ \widehat{h_1} \}^{H_6}$
for Example~\ref{L2_8_7A7A7A2A_Part1} results in tables \ref{tab:ElementsZ40List} and \ref{tab:PermutationRepresentationOnZ40Betaj}.
At the moment, assume that tables \ref{tab:ElementsZ40List} and \ref{tab:PermutationRepresentationOnZ40Betaj}
also describe the action of $H_6$ on $\{ \widehat{h_1} \}^{H_6}$ for the generic case. This will be shown in Proposition~\ref{BraidOrbitAction40}.
Then we have $\vert \{ \widehat{h_1} \}^{H_6} \vert = 40$.
The permutations $\rho_6(\beta_{ij})$ can be computed from $\rho_6(\beta_j)$ due to
$$\beta_{ij} = \beta_{i+1}^{-1} \cdots \beta_{j-1}^{-1} \beta_j^2 \beta_{j-1} \cdots \beta_{i+1}, 1 \leq i < j \leq 6$$
and are listed in table \ref{tab:PermutationRepresentationOnZ40Betaij}.
Using GAP \cite{RefGAP} it is easy to show that
$\rho_6(B_6) = \rho_6(\langle \beta_{ij} \rangle)$ acts transitively on 
$\{ \widehat{h_1} \}^{H_6} \supseteq \{ \widehat{h_1} \}^{B_6}$, 
resulting in $\{ \widehat{h_1} \}^{H_6} = \{ \widehat{h_1} \}^{B_6} = Z_{40}$.
The structure of the permutation groups $\rho_6(B_6)$ and $\rho_6(H_6)$ can be computed using GAP \cite{RefGAP}.
From the permutation types $\rho^t_6(\beta_{ij})$ we get
$$g_{Z_{40}}= 1 - 40 + \frac{1}{2}(5\cdot 40 -22 -22 -22 -22 -22) = 6.$$
\end{proof}
\begin{proposition}
\label{BraidOrbitAction40}
With the prerequisites of Theorem~\ref{BraidOrbitsOfSize40}, tables \ref{tab:ElementsZ40List} and \ref{tab:PermutationRepresentationOnZ40Betaj}
describe the action of $H_6$ on $\{ \widehat{h_1} \}^{H_6}$.
\end{proposition}
\begin{proof}
By construction, we have $\widehat{h_1} = k_1$ and $\{ \widehat{h_1} \}^{H_6} \supseteq \{ k_1,\dots,k_{40} \}$.
If we can verify
$$k_l^{\beta_j} = k_{\rho_6(\beta_j)(l)}$$
for $j=2,\dots,6$ and $l=1,\dots,40$ with $k_l$ from table \ref{tab:ElementsZ40List} 
and $\rho_6(\beta_j)$ from table \ref{tab:PermutationRepresentationOnZ40Betaj} the proposition is demonstrated.
This has been done with the help of a computer.
To this the prerequisites of Theorem~\ref{BraidOrbitsOfSize40} play a crucial role.
\end{proof}
\begin{table}[H]
\centering
\footnotesize
\captionsetup{font=footnotesize}
\caption{The elements of $Z_{40}$}
\label{tab:ElementsZ40List}
\begin{tabular}{lll}
\hline\noalign{\smallskip}
$l$        & $k_l$ & $\beta \in H_6$ with $k_1^{\beta} = k_l$ \\
\noalign{\smallskip}\hline\noalign{\smallskip}
1&$[\sigma_1,\sigma_2,\sigma_1,\sigma_1,\sigma_2,\sigma_1]$&$\iota$\\
2&$[\sigma_2,\sigma_1,\sigma_1,\sigma_1,\sigma_1^{-1}\sigma_2\sigma_1,\sigma_1]$&$\beta_2$\\
3&$[\sigma_1,\sigma_2\sigma_1\sigma_2^{-1},\sigma_2,\sigma_1,\sigma_2,\sigma_1]$&$\beta_3$\\
4&$[\sigma_1,\sigma_2\sigma_1\sigma_2^{-1},\sigma_2\sigma_1\sigma_2^{-1},\sigma_2,\sigma_2,\sigma_1]$&$\beta_3\beta_4$\\
5&$[\sigma_1,\sigma_2,\sigma_1,\sigma_1\sigma_2\sigma_1^{-1},\sigma_1,\sigma_1]$&$\beta_5$\\
6&$[\sigma_2,\sigma_1,\sigma_1,\sigma_2,\sigma_1,\sigma_1]$&$\beta_2\beta_5$\\
7&$[\sigma_1,\sigma_2\sigma_1\sigma_2^{-1},\sigma_2,\sigma_1\sigma_2\sigma_1^{-1},\sigma_1,\sigma_1]$&$\beta_3\beta_5$\\
8&$[\sigma_1,\sigma_2,\sigma_1,\sigma_1,\sigma_2\sigma_1\sigma_2^{-1},\sigma_2]$&$\beta_6$\\
9&$[\sigma_2,\sigma_1,\sigma_1,\sigma_1,\sigma_1^{-1}\sigma_2\sigma_1\sigma_2^{-1}\sigma_1,\sigma_1^{-1}\sigma_2\sigma_1]$
&$\beta_2\beta_6$\\
10&$[\sigma_1,\sigma_2\sigma_1\sigma_2^{-1},\sigma_2,\sigma_1,\sigma_2\sigma_1\sigma_2^{-1},\sigma_2]$&$\beta_3\beta_6$\\
\noalign{\smallskip}\hline\noalign{\smallskip}
11&$[\sigma_1,\sigma_2\sigma_1\sigma_2^{-1},\sigma_2\sigma_1\sigma_2^{-1},\sigma_2,\sigma_2\sigma_1\sigma_2^{-1},\sigma_2]$
&$\beta_3\beta_4\beta_6$\\
12&$[\sigma_1,\sigma_2,\sigma_2^{-1}\sigma_1\sigma_2,\sigma_2^{-1}\sigma_1\sigma_2,\sigma_1,\sigma_2^{-1}\sigma_1\sigma_2]$
&$\beta_2^{2}$\\
13&$[\sigma_1,\sigma_2,\sigma_1,\sigma_2,\sigma_1,\sigma_2]$&$\beta_3\beta_2$\\
14&$[\sigma_1,\sigma_2^{-1}\sigma_1\sigma_2,\sigma_2^{-1}\sigma_1^{-1}\sigma_2\sigma_1\sigma_2^{-1}\sigma_1\sigma_2,
\sigma_1,\sigma_1,\sigma_2^{-1}\sigma_1\sigma_2]$&$\beta_3\beta_4\beta_2$\\
15&$[\sigma_1,\sigma_2,\sigma_2^{-1}\sigma_1\sigma_2,\sigma_2,\sigma_2^{-1}\sigma_1\sigma_2,
\sigma_2^{-1}\sigma_1\sigma_2]$&$\beta_2\beta_5\beta_2$\\
16&$[\sigma_1,\sigma_2^{-1}\sigma_1\sigma_2,\sigma_1,\sigma_2,\sigma_2^{-1}\sigma_1\sigma_2,
\sigma_2^{-1}\sigma_1\sigma_2]$&$\beta_3\beta_5\beta_2$\\
17&$[\sigma_1,\sigma_2,\sigma_2^{-1}\sigma_1\sigma_2,\sigma_2^{-1}\sigma_1\sigma_2,
\sigma_2^{-1}\sigma_1^{-1}\sigma_2\sigma_1\sigma_2^{-1}\sigma_1\sigma_2,\sigma_1]$&$\beta_2\beta_6\beta_2$\\
18&$[\sigma_1,\sigma_2^{-1}\sigma_1\sigma_2,\sigma_1,\sigma_2^{-1}\sigma_1\sigma_2,
\sigma_2^{-1}\sigma_1^{-1}\sigma_2\sigma_1\sigma_2^{-1}\sigma_1\sigma_2,\sigma_1]$&$\beta_3\beta_6\beta_2$\\
19&$[\sigma_1,\sigma_1,\sigma_2\sigma_1\sigma_2^{-1},\sigma_1,\sigma_2,\sigma_1]$&$\beta_3^{2}$\\
20&$[\sigma_1,\sigma_1,\sigma_2\sigma_1\sigma_2^{-1},\sigma_1\sigma_2\sigma_1^{-1},\sigma_1,\sigma_1]$&$\beta_3\beta_5\beta_3$\\
\noalign{\smallskip}\hline\noalign{\smallskip}
21&$[\sigma_1,\sigma_1,\sigma_2\sigma_1\sigma_2^{-1},\sigma_1,\sigma_2\sigma_1\sigma_2^{-1},\sigma_2]$&$\beta_3\beta_6\beta_3$\\
22&$[\sigma_1,\sigma_1,\sigma_2,\sigma_2^{-1}\sigma_1\sigma_2,\sigma_1,\sigma_2^{-1}\sigma_1\sigma_2]$&$\beta_2^{2}\beta_3$\\
23&$[\sigma_1,\sigma_1,\sigma_2^{-1}\sigma_1\sigma_2,\sigma_1,\sigma_1,
\sigma_2^{-1}\sigma_1\sigma_2]$&$\beta_3\beta_4\beta_2\beta_3$\\
24&$[\sigma_1,\sigma_1,\sigma_2,\sigma_2,
\sigma_2^{-1}\sigma_1\sigma_2,\sigma_2^{-1}\sigma_1\sigma_2]$&$\beta_2\beta_5\beta_2\beta_3$\\
25&$[\sigma_1,\sigma_1,\sigma_2,\sigma_2^{-1}\sigma_1\sigma_2,
\sigma_2^{-1}\sigma_1^{-1}\sigma_2\sigma_1\sigma_2^{-1}\sigma_1\sigma_2,\sigma_1]$&$\beta_2\beta_6\beta_2\beta_3$\\
26&$[\sigma_1,\sigma_2,\sigma_1^{2}\sigma_2\sigma_1^{-2},\sigma_1,\sigma_1,\sigma_1]$&$\beta_5\beta_4$\\
27&$[\sigma_2,\sigma_1,\sigma_1\sigma_2\sigma_1^{-1},\sigma_1,\sigma_1,\sigma_1]$&$\beta_2\beta_5\beta_4$\\
28&$[\sigma_1,\sigma_2\sigma_1\sigma_2^{-1},\sigma_1,\sigma_2\sigma_1\sigma_2^{-1},\sigma_2\sigma_1\sigma_2^{-1},
\sigma_2]$&$\beta_3\beta_4\beta_6\beta_4$\\
29&$[\sigma_1,\sigma_2^{-1}\sigma_1\sigma_2,\sigma_2^{-1}\sigma_1\sigma_2,
\sigma_2^{-1}\sigma_1^{-1}\sigma_2\sigma_1\sigma_2^{-1}\sigma_1\sigma_2,
\sigma_1,\sigma_2^{-1}\sigma_1\sigma_2]$&$\beta_3\beta_4\beta_2\beta_4$\\
30&$[\sigma_1,\sigma_2,\sigma_2^{-1}\sigma_1\sigma_2\sigma_1^{-1}\sigma_2,\sigma_2^{-1}\sigma_1\sigma_2,
\sigma_2^{-1}\sigma_1\sigma_2,\sigma_2^{-1}\sigma_1\sigma_2]$&$\beta_2\beta_5\beta_2\beta_4$\\
\noalign{\smallskip}\hline\noalign{\smallskip}
31&$[\sigma_1,\sigma_1,\sigma_2\sigma_1\sigma_2^{-1}\sigma_1\sigma_2\sigma_1^{-1}\sigma_2^{-1},
\sigma_2\sigma_1\sigma_2^{-1},\sigma_2\sigma_1\sigma_2^{-1},\sigma_2]$&$\beta_3\beta_6\beta_3\beta_4$\\
32&$[\sigma_1,\sigma_1,\sigma_1,\sigma_2,\sigma_1,\sigma_2^{-1}\sigma_1\sigma_2]$&$\beta_2^{2}\beta_3\beta_4$\\
33&$[\sigma_1,\sigma_1,\sigma_1,\sigma_2,\sigma_2^{-1}\sigma_1^{-1}\sigma_2\sigma_1\sigma_2^{-1}\sigma_1\sigma_2,
\sigma_1]$&$\beta_2\beta_6\beta_2\beta_3\beta_4$\\
34&$[\sigma_1,\sigma_2\sigma_1\sigma_2^{-1},\sigma_2\sigma_1\sigma_2^{-1},\sigma_2^{2}\sigma_1\sigma_2^{-2},
\sigma_2,\sigma_2]$&$\beta_3\beta_4\beta_6\beta_5$\\
35&$[\sigma_1,\sigma_1,\sigma_1,\sigma_1^{-1}\sigma_2\sigma_1,\sigma_2,
\sigma_2^{-1}\sigma_1\sigma_2]$&$\beta_2^{2}\beta_3\beta_4\beta_5$\\
36&$[\sigma_1,\sigma_2,\sigma_1,\sigma_1,\sigma_1,\sigma_2\sigma_1\sigma_2^{-1}]$&$\beta_6^{2}$\\
37&$[\sigma_2,\sigma_1,\sigma_1,\sigma_1,\sigma_1,\sigma_1^{-1}\sigma_2\sigma_1\sigma_2^{-1}\sigma_1]$&$\beta_2\beta_6^{2}$\\
38&$[\sigma_1,\sigma_2,\sigma_2^{-1}\sigma_1\sigma_2,\sigma_2^{-1}\sigma_1\sigma_2,\sigma_2^{-1}\sigma_1\sigma_2,
\sigma_2^{-1}\sigma_1^{-1}\sigma_2\sigma_1\sigma_2^{-1}\sigma_1\sigma_2]$&$\beta_2\beta_6\beta_2\beta_6$\\
39&$[\sigma_2,\sigma_1^{2}\sigma_2\sigma_1^{-2},\sigma_1,\sigma_1,\sigma_1,\sigma_1]$&$\beta_2\beta_5\beta_4\beta_3$\\
40&$[\sigma_1,\sigma_1,\sigma_1,\sigma_1,\sigma_1^{-1}\sigma_2\sigma_1,\sigma_2^{-1}\sigma_1\sigma_2]$
&$\beta_2^{2}\beta_3\beta_4\beta_5^{2}$\\
\noalign{\smallskip}\hline
\end{tabular}
\end{table}
\begin{table}[!htbp]
\centering
\footnotesize
\captionsetup{font=footnotesize}
\caption{The action of $\beta_j$ on $Z_{40}$}
\label{tab:PermutationRepresentationOnZ40Betaj}
\begin{tabular}{ll}
\hline\noalign{\smallskip}
$j$        & $\rho_6(\beta_j)$ \\
\noalign{\smallskip}\hline\noalign{\smallskip}
2  & $(1,2,12)(3,13,11)(4,14,28)(5,6,15)(7,16,34)(8,9,17)(10,18,29)(26,27,30)(36,37,38)$ \\
3  & $(1,3,19)(5,7,20)(8,10,21)(12,22,13)(14,23,36)(15,24,16)(17,25,18)(27,39,30)(28,38,31)$ \\
4  & $(3,4,18)(5,26,7)(6,27,16)(10,11,28)(13,14,29)(15,30,34)(19,25,33)(21,31,35)(22,32,23)$ \\
5  & $(1,5,18)(2,6,29)(3,7,17)(8,13,16)(9,11,34)(10,12,15)(19,20,25)(21,22,24)(32,35,40)$ \\
6  & $(1,8,36)(2,9,37)(3,10,14)(4,11,29)(12,17,38)(13,18,28)(19,21,23)(22,25,31)(32,33,35)$ \\
\noalign{\smallskip}\hline
\end{tabular}
\end{table}
\begin{table}[!htbp]
\centering
\footnotesize
\captionsetup{font=footnotesize}
\caption{The action of $\beta_{ij}$ on $Z_{40}$}
\label{tab:PermutationRepresentationOnZ40Betaij}
\begin{tabular}{ll}
\hline\noalign{\smallskip}
$i,j$ & $\rho_6(\beta_{ij})$ \\
\noalign{\smallskip}\hline\noalign{\smallskip}
1,2   & $(1,12,2)(3,11,13)(4,28,14)(5,15,6)(7,34,16)(8,17,9)(10,29,18)(26,30,27)(36,38,37)$ \\
1,3   & $(1,11,22)(2,19,13)(4,31,36)(5,34,24)(6,20,16)(8,29,25)(9,21,18)(23,28,37)(26,39,30)$ \\
1,4   & $(1,10,23)(2,33,29)(3,21,36)(4,9,35)(5,39,15)(7,30,24)(8,14,19)(11,37,32)(16,20,27)$ \\
1,5   & $(1,39,12)(3,30,22)(4,34,32)(6,29,33)(9,37,40)(13,19,27)(14,25,16)(15,23,18)(17,24,36)$ \\
1,6   & $(2,9,40)(6,11,32)(8,12,24)(10,22,16)(13,15,21)(14,30,31)(23,27,28)(29,34,35)(36,39,38)$ \\
\noalign{\smallskip}\hline\noalign{\smallskip}
2,3   & $(1,19,3)(5,20,7)(8,21,10)(12,13,22)(14,36,23)(15,16,24)(17,18,25)(27,30,39)(28,31,38)$ \\
2,4   & $(1,33,18)(4,19,17)(6,15,39)(7,20,26)(8,35,28)(11,21,38)(12,29,23)(13,36,32)(24,34,27)$ \\
2,5   & $(2,12,39)(3,19,26)(4,25,7)(5,18,33)(6,23,10)(8,36,40)(9,24,38)(11,27,22)(16,32,28)$ \\
2,6   & $(1,8,40)(2,24,17)(3,6,21)(4,27,31)(5,13,32)(7,29,22)(14,23,26)(16,35,18)(37,38,39)$ \\
\noalign{\smallskip}\hline\noalign{\smallskip}
3,4   & $(3,18,4)(5,7,26)(6,16,27)(10,28,11)(13,29,14)(15,34,30)(19,33,25)(21,35,31)(22,23,32)$ \\
3,5   & $(1,3,26)(2,13,27)(4,17,5)(6,14,8)(9,15,28)(11,30,12)(20,25,33)(21,23,40)(24,32,31)$ \\
3,6   & $(1,6,10)(2,15,18)(4,30,38)(5,29,12)(14,26,36)(19,21,40)(20,22,32)(24,35,25)(27,37,28)$ \\
\noalign{\smallskip}\hline\noalign{\smallskip}
4,5   & $(1,18,5)(2,29,6)(3,17,7)(8,16,13)(9,34,11)(10,15,12)(19,25,20)(21,24,22)(32,40,35)$ \\
4,6   & $(1,16,28)(2,34,4)(3,15,38)(5,36,13)(6,37,11)(7,14,12)(19,24,31)(20,23,22)(33,35,40)$ \\
\noalign{\smallskip}\hline\noalign{\smallskip}
5,6   & $(1,36,8)(2,37,9)(3,14,10)(4,29,11)(12,38,17)(13,28,18)(19,23,21)(22,31,25)(32,35,33)$ \\
\noalign{\smallskip}\hline
\end{tabular}
\end{table}
\noindent
Note that we have $k_1^{\varphi_{6,2}} = k_1$ and $k_2^{\varphi_{6,2}} \neq k_2$, thus the $B_6$-orbit $Z_{40}$
is not a set of fixed points under $F_{6,2}$.
\begin{corollary}
\label{ExplicitZ40}
With the assumptions and notation of Theorem~\ref{BraidOrbitsOfSize40}, $C = [\sigma_1]$ and
$$[\underline{\sigma}] = [\sigma_1, \sigma_2, \sigma_1, \sigma_1, \sigma_2, \sigma_1],$$
we have
$$Z_{40} = [\underline{\sigma}]^{B_6} = [\underline{\sigma}]^{H_6} = \{ k_1, \dots, k_{40} \} \subseteq \Sigma^i(C,C,C,C,C,C).$$
If $l^i(C,C,C) > 0$ and $[\underline{\tau}] \in \Sigma^i(C,C,C)$, we have
$$Z_{40} \cap T_{B_6} =  Z_{40} \cap T_{H_6} = \emptyset$$
for $T_{B_6} = [\widehat{\underline{\tau}}]^{B_6}$ resp. $T_{H_6} = [\widehat{\underline{\tau}}]^{H_6}$.
\end{corollary}
\begin{proof}
The first part comes from Theorem~\ref{BraidOrbitsOfSize40} together with Proposition~\ref{BraidOrbitAction40}.
\par
\noindent
Now let us assume $Z_{40} \cap T_{H_6} \neq \emptyset$. Then we have $Z_{40} = T_{H_6}$ as both sets are $H_6$-orbits
and for $[\underline{\tau}] \in \Sigma^i(C,C,C)$ we get
$[\widehat{\underline{\tau}}] = [\tau_1, \tau_2, \tau_3, \tau_1, \tau_2, \tau_3] \in Z_{40} = \{ k_1, \dots, k_{40} \}$.
Thus there exists an index $l \in \{ 1, \dots, 40 \}$ with
$[\widehat{\underline{\tau}}] = k_l = [\kappa_1, \kappa_2, \kappa_3, \kappa_4, \kappa_5, \kappa_6]$
and we get an element $\gamma \in G$ with
$$\gamma^{-1} \tau_i \gamma = \kappa_i, \gamma^{-1} \tau_i \gamma = \kappa_{i+3}, i=1,2,3.$$ This leads to
\begin{equation} 
\label{eq:KappaCondition1}
\kappa_1 = \kappa_4
\end{equation}
\begin{equation} 
\label{eq:KappaCondition2}
\kappa_2 = \kappa_5
\end{equation}
\begin{equation} 
\label{eq:KappaCondition3}
\kappa_3 = \kappa_6
\end{equation}
\begin{equation}
\label{eq:KappaCondition4}
\kappa_1 \kappa_2 \kappa_3 = \iota
\end{equation}
\begin{equation}
\label{eq:KappaCondition5}
\kappa_4 \kappa_5 \kappa_6 = \iota
\end{equation}
These conditions rule out $k_l$ from table \ref{tab:ElementsZ40List} for $l = 1, \dots, 40$.
For example in case $l=1$ we get
$\kappa_1 \kappa_2 \kappa_3 = \sigma_1 \sigma_2 \sigma_1$ and therefore $\kappa_1 \kappa_2 \kappa_3 \neq \iota$
due to $o(\sigma_1 \sigma_2 \sigma_1)=2$.
Table \ref{tab:ConditionCheckForkl} contains the check for $k_l$, $l = 1, \dots, 40$.
\begin{table}[h!]
\centering
\footnotesize
\captionsetup{font=footnotesize}
\caption{Check for $k_l$, $l = 1, \dots, 40$}
\label{tab:ConditionCheckForkl}
\begin{tabular}{ll}
\hline\noalign{\smallskip}
Condition        & Violated for $l$ resp. $k_l$ \\
\noalign{\smallskip}\hline\noalign{\smallskip}
(\ref{eq:KappaCondition1}) & 2, 4, 9, 15, 16, 18, 20, 27, 30, 32, 33, 35, 37, 39 \\
(\ref{eq:KappaCondition2}) & 5, 13, 14, 19, 26, 29, 31, 36, 40 \\
(\ref{eq:KappaCondition3}) & 3, 8, 28 \\
(\ref{eq:KappaCondition4}) & 1, 6, 7, 10, 12, 17, 22, 23, 24, 25, 38 \\
(\ref{eq:KappaCondition5}) & 11, 21, 34, \\
\noalign{\smallskip}\hline
\end{tabular}
\end{table}
Finally, our hypothesis turns out to be wrong and we get $Z_{40} \cap T_{H_6} = \emptyset$.
Due to $T_{B_6} \subseteq T_{H_6}$ we get $Z_{40} \cap T_{B_6} = \emptyset$.
\end{proof}
\begin{corollary}
\label{Systems23ToBraidOrbitsOfSize40}
Let $G>I$ be a finite group with $Z(G)=I$ and $(2A,3A,C)$ a class vector of $G$ with $l^i(2A,3A,C)>0$,
$2A$ a class of involutions and $3A$ a class of elements with order $3$.
Then $\Sigma^i(C,C,C,C,C,C)$ contains $B_6$-orbits $Z_{40}$ with size $40$.
If $l^i(C,C,C) > 0$, then $\Sigma^i(C,C,C,C,C,C)$ contains at least two $B_6$-orbits.
\end{corollary}
\begin{proof}
We use Corollary~\ref{Systems23ToBraidOrbitsOfSize4} to construct elements $h_1$ resp. $\widehat{h_1}$
from $(2,3)$-generating systems of $G$ and get $B_6$-orbits $Z_{40}$ from Theorem~\ref{BraidOrbitsOfSize40}.
From Corollary~\ref{ExplicitZ40} follows the existence of $B_6$-orbits different from $B_6$-orbits $Z_{40}$.
\end{proof}
\noindent
Finally we will look at a possible generalization.
Let $G>I$ be a finite group with $Z(G)=I$ and $\sigma_1, \sigma_2, \sigma_3 \in G$ with
$G=\langle \sigma_1, \sigma_2, \sigma_3 \rangle$,
\begin{equation} 
\label{eq:Artin4BraidRelation1}
\sigma_1 \sigma_2 \sigma_1 = \sigma_2 \sigma_1 \sigma_2, \sigma_2 \sigma_3 \sigma_2 = \sigma_3 \sigma_2 \sigma_3
\end{equation}
\begin{equation}
\label{eq:Artin4BraidRelation2}
\sigma_1 \sigma_3 = \sigma_3 \sigma_1
\end{equation}
\begin{equation}
\label{eq:Artin4TrivialCenter}
(\sigma_1 \sigma_2 \sigma_3)^4 = \iota
\end{equation}
Thus $G$ is a factor group of the Artin braid group $Ar_4$ with $3$ strands.
Using (\ref{eq:Artin4BraidRelation1}) and (\ref{eq:Artin4BraidRelation2}) we get 
$C = [\sigma_1] = [\sigma_2] = [\sigma_3]$ and
$$\sigma_1 \sigma_2 \sigma_1 \sigma_2 \sigma_3 \sigma_2 = \sigma_1 \sigma_2 \sigma_1 \sigma_3 \sigma_2 \sigma_3 =
\sigma_1 \sigma_2 \sigma_3 \sigma_1 \sigma_2 \sigma_3 = (\sigma_1 \sigma_2 \sigma_3)^2.$$
From (\ref{eq:Artin4TrivialCenter}) we get
$$h_1= [\sigma_1, \sigma_2, \sigma_1, \sigma_2, \sigma_3, \sigma_2,
\sigma_1 \sigma_2 \sigma_1 \sigma_2 \sigma_3 \sigma_2] \in \Sigma^i(C,C,C,C,C,C,D)$$
with $D = [(\sigma_1 \sigma_2 \sigma_3)^2]$ and
$$\widehat{h_1} = 
[\sigma_1, \sigma_2, \sigma_1, \sigma_2, \sigma_3, \sigma_2, \sigma_1, \sigma_2, \sigma_1, \sigma_2, \sigma_3, \sigma_2] 
\in \Sigma^i(C,\dots,C).$$
By definition, $\widehat{h_1}$ is a fixed point under the action of $F_{12,2}$.
With respect to Theorem~\ref{BraidOrbitsOfSize4} and Theorem~\ref{BraidOrbitsOfSize40}, we may
ask whether $\{ h_1\}^{B_7}$, $\{ \widehat{h_1} \}^{B_{12}}$ and $\{ \widehat{h_1} \}^{H_{12}}$
are generic braid orbits.
\begin{example}
\label{S4_2B2B2B2B2B2B2A_Part1}
For the symmetric group $S_4$, the elements $\sigma_1 = (1)(2)(3,4)$, $\sigma_2 = (1)(2,3)(4)$ and $\sigma_3 = (1,2)(3)(4)$
generate $S_4$ and satisfy (\ref{eq:Artin4BraidRelation1}), (\ref{eq:Artin4BraidRelation2}) and (\ref{eq:Artin4TrivialCenter}).
Using a computer and GAP \cite{RefGAP}, we get that $Z_{480} = \{ h_1\}^{B_7}$ is an orbit of length $480$
under the action of $B_7$ with
\begin{itemize}
\item[]
$\rho^t_7(\beta_{ij}) = (1)^{156} (3)^{108}, 1 \leq i < j \leq 6$,
\item[]
$\rho^t_7(\beta_{i7}) = (1)^{160} (2)^{160}, 1 \leq i \leq 6$,
\item[]
$\vert \rho_7(B_7) \vert = 2^{182} \cdot 3^{48} \cdot 5$
\end{itemize}
and genus $g_{Z_{480}}=141$.
The group $\rho_7(B_7)$ is not solvable.
\end{example}
\noindent
Just like we did it in the proof for Theorem~\ref{BraidOrbitsOfSize40}, we could check whether the action of $B_7$
in Example~\ref{S4_2B2B2B2B2B2B2A_Part1} is generic. Unfortunately supporting software for $m \geq 7$ is not available.

\section{Class vectors of dimension 5 and 6} \label{SectionClassvectorsOfDimension5And6}

Table \ref{tab:Data_56DimClassvectors} contains some class vectors of length $5$ and $6$ of $M_{23}$, $M_{24}$
and the Suzuki group $Sz(8)$. We always have $l^i(C) = \vert Z \vert$.
Column $ \vert Z \vert $ now contains the length of $B_5$- resp. $B_6$-orbits and $g_Z$ is the corresponding braid orbit genus 
(the setup in section \ref{SectionTheActionOfBraids} can easily be extended to class vectors of length $m$).
None of the class vectors has a small orbit with $g_Z=0$, thus \cite{RefMM2018}, III, Theorem 5.7 cannot be applied
to realize $M_{23}$, $M_{24}$ or $Sz(8)$ with one of these class vectors as Galois group over $\mathbb{Q}(v,t)$.
For $m>4$, the class vector $(4B,2A,2A,2A,2A)$ of $M_{24}$ is the only one in $M_{24}$ with $g_{M_{23}}=0$.
In \cite{RefKoe2014}, 5.2, this class vector is also considered. But until now, the methods used in \cite{RefKoe2014} 
did not generate a Galois realization of $M_{23}$ over $\mathbb{Q}$.
\begin{table}[H]
\centering
\footnotesize
\captionsetup{font=footnotesize}
\caption{Some class vectors of length $5$ and $6$ with $l^i(C) = \vert Z \vert$}
\label{tab:Data_56DimClassvectors}
\begin{tabular}{llrrrrl}
\hline\noalign{\smallskip}
$G$ & $C$ & $g_{M_{23}}$ & $ \vert Z \vert $ & $g_Z$ & $\frac{g_Z}{\vert Z \vert}$ & Ref.\\
\noalign{\smallskip}\hline\noalign{\smallskip}
$M_{23}$ & $(2A_5)$ & - & 0 & - & - & -\\
$M_{23}$ & $(3A,2A_4)$ & - & 21,456 & 12,961 & 0.6041 & \cite{RefKoe2014}, \cite{RefMag2003}\\
$M_{23}$ & $(4A,2A_4)$ & - & 244,224 & 170,785 & 0.6993 & -\\
$M_{23}$ & $(5A,2A_4)$ & - & 732,000 & 519,677 & 0.7099 & -\\
$M_{23}$ & $(6A,2A_4)$ & - & 1,050,336 & 776,541 & 0.7393 & -\\
$M_{23}$ & $(8A,2A_4)$ & - & 1,859,328 & 1,381,729 & 0.7331 & -\\
\noalign{\smallskip}\hline\noalign{\smallskip}
$M_{24}$ & $(2A_5)$ & - & 0 & - & - & -\\
$M_{24}$ & $(2B,2A_4)$ & - & 0 & - & - & -\\
$M_{24}$ & $(3A,2A_4)$ & - & 0 & - & - & -\\
$M_{24}$ & $(3B,2A_4)$ & 1 & 75,816 & 57,345 & 0.7564 & -\\
$M_{24}$ & $(4A,2A_4)$ & 1 & 74,496 & 53,133 & 0.7132 & -\\
$M_{24}$ & $(4B,2A_4)$ & 0 & 72,000 & 50,789 & 0.7054 & \cite{RefKoe2014}, \cite{RefMag2003}\\
$M_{24}$ & $(4C,2A_4)$ & 2 & 468,928 & 367,193 & 0.7830 & -\\
$M_{24}$ & $(2B_5)$ & 7 & 477,680 & 257,865 & 0.5398 & -\\
\noalign{\smallskip}\hline\noalign{\smallskip}
$Sz(8)$ & $(2A_5)$ & - & 22,752 & 16,249 & 0.7142 & -\\
$Sz(8)$ & $(5A,2A_4)$ & - & 291,150 & 202,819 & 0.6966 & -\\
\noalign{\smallskip}\hline\noalign{\smallskip}
$M_{23}$ & $(2A_6)$ & - & 16,463,280 & 9,722,161 & 0.5905 & -\\
\noalign{\smallskip}\hline\noalign{\smallskip}
$M_{24}$ & $(2A_6)$ & 1 & ? & ? & ? & -\\
$M_{24}$ & $(2B_6)$ & 13 & ? & ? & ? & -\\
\noalign{\smallskip}\hline\noalign{\smallskip}
$Sz(8)$ & $(2A_6)$ & - & 10,507,680 & 12,175,651 & 1.1587 & -\\
\noalign{\smallskip}\hline
\end{tabular}
\end{table}

\section{Specialized class vectors and orbits}\label{SectionDerivedClassvectors}

Again, let $G$ be a finite group with $Z(G)=I$ and $[\underline{\sigma}] = [\sigma_1, \dots, \sigma_m]$
a generating $m$-system in $\Sigma^i(C_1,\dots, C_m)$.
For $1 \leq i < m$, we can define
$$[\underline{\sigma}]_{(i)} = [\sigma_1, \dots, \sigma_{i-1}, \sigma_i \sigma_{i+1}, \sigma_{i+2}, \dots, \sigma_m] \in 
{\overline\Sigma}^i(C_1,\dots, C_{i-1}, D_i,  C_{i+2}, \dots, C_m)$$
with $D_i = [\sigma_i \sigma_{i+1}]$ and ${\overline\Sigma}^i(\dots) \supseteq {\Sigma}^i(\dots)$
allowing also $m-1$-systems that may generate a proper subgroup of $G$.
In this section, we will only consider the case $[\underline{\sigma}]_{(i)} \in \Sigma^i(\dots)$.
We use the notation
$$[\underline{\sigma}]_{(i_1,i_2)} = ([\underline{\sigma}]_{(i_1)})_{(i_2)}$$
for $1 \leq i < m$ and $1 \leq j < m-1$.
If we know the $B_m$-orbit $[\underline{\sigma}]^{B_m}$, what can be said about the $B_{m-1}$-orbit
$[\underline{\sigma}]_{(i)}^{B_{m-1}}$? We give some examples that may point to an answer.
\begin{example}
\label{FromZ40ToZ4}
We have 
$$(k_1)_{(4,4)} = [\sigma_1,\sigma_2,\sigma_1,\sigma_1,\sigma_2,\sigma_1]_{(4,4)} = [\sigma_1,\sigma_2,\sigma_1,\sigma_1\sigma_2\sigma_1] = h_1,$$
see Theorem~\ref{BraidOrbitsOfSize40} and Theorem~\ref{BraidOrbitsOfSize4}. Thus we can get $Z_4$ from $Z_{40}$
by specialization.
\end{example}
\begin{example}
\label{FromZ40ToZ9}
Now consider
$$(k_1)_{(5)} = [\sigma_1,\sigma_2,\sigma_1,\sigma_1,\sigma_2,\sigma_1]_{(5)} = [\sigma_1,\sigma_2,\sigma_1,\sigma_1,\sigma_2\sigma_1].$$
Looking at $(k_1)_{(5)}$ in the examples listed in table \ref{tab:Data_B6OrbitsSiz40} leads to a $B_5$-orbit $Z_9$
deduced from $Z_{40}$ by specialization, see Theorem~\ref{BraidOrbitsOfSize9} below.
\end{example}
\begin{theorem}
\label{BraidOrbitsOfSize9}
Let $G>I$ be a finite group with $Z(G)=I$ and $\sigma_1, \sigma_2 \in G$ with
$G=\langle \sigma_1, \sigma_2 \rangle$, $o(\sigma_1 \sigma_2 \sigma_1 )=2$,
$\sigma_1 \sigma_2 \sigma_1 = \sigma_2 \sigma_1 \sigma_2$, $C = [\sigma_1]$, $D = [\sigma_2\sigma_1]$ and further
$[\underline{\sigma}] = [\sigma_1, \sigma_2, \sigma_1, \sigma_1, \sigma_2\sigma_1]$.
Then
$$Z_9 = [\underline{\sigma}]^{B_5} \subseteq \Sigma^i(C,C,C,C,D)$$
is an orbit of length $9$ under the action of $B_5$ with
\begin{itemize}
\item[]
$\rho^t_5(\beta_{ij}) = (1)^3(3)^2, 1 \leq i < j \leq 4$,
\item[]
$\rho^t_5(\beta_{i5}) = (1)(2)^4, 1 \leq i < 5$,
\item[]
$\rho_5(B_5) \cong ((C_3 \times C_3) \rtimes Q_8) \rtimes C_3$
\end{itemize}
and genus $g_{Z_9}=0$.
\end{theorem}
\begin{proof} The proof is analogous to the proof of Theorem~\ref{BraidOrbitsOfSize40}.
Let 
$$Z_9 = \{p_1, \dots, p_9\}$$
with $p_i, i=1, \dots, 9$, from table \ref{tab:ElementsZ9List}.
Then tables \ref{tab:ElementsZ9List} and \ref{tab:PermutationRepresentationOnZ9Betaij}
describe the action of $B_5$ on $Z_9$. Again this can be verified with the help of a computer.
The structure of the permutation group $\rho_5(B_5)$ can be computed using GAP \cite{RefGAP}.
From the permutation types $\rho^t_5(\beta_{ij})$ we get
$$g_{Z_9}= 1 - 9 + \frac{1}{2}(4\cdot 9 -5 -5 -5 -5) = 0.$$
\end{proof}
\begin{table}[!htbp]
\centering
\footnotesize
\captionsetup{font=footnotesize}
\caption{The elements of $Z_9$}
\label{tab:ElementsZ9List}
\begin{tabular}{lll}
\hline\noalign{\smallskip}
$l$        & $p_l$ & $\beta \in B_5$ with $p_1^{\beta} = p_l$ \\
\noalign{\smallskip}\hline\noalign{\smallskip}
1&$[\sigma_1,\sigma_2,\sigma_1,\sigma_1,\sigma_2\sigma_1]$&$\iota$\\
2&$[\sigma_2,\sigma_1\sigma_2\sigma_1^{-1},\sigma_1,\sigma_1,\sigma_2\sigma_1]$&$\beta_{12}$\\
3&$[\sigma_2^{-1}\sigma_1\sigma_2,\sigma_1,\sigma_1,\sigma_2,\sigma_1\sigma_2]$&$\beta_{13}$\\
4&$[\sigma_2,\sigma_1,\sigma_1^{-1}\sigma_2\sigma_1,\sigma_2,\sigma_2\sigma_1]$&$\beta_{14}$\\
5&$[\sigma_2^{-1}\sigma_1\sigma_2,\sigma_1,\sigma_1,\sigma_1,\sigma_2\sigma_1]$&$\beta_{13}\beta_{14}$\\
6&$[\sigma_1,\sigma_1,\sigma_1,\sigma_1^{-1}\sigma_2\sigma_1,\sigma_1\sigma_2]$&$\beta_{13}\beta_{15}$\\
7&$[\sigma_1,\sigma_1,\sigma_1^{-1}\sigma_2\sigma_1,\sigma_2,\sigma_1\sigma_2]$&$\beta_{14}\beta_{15}$\\
8&$[\sigma_1,\sigma_1,\sigma_2,\sigma_1,\sigma_1\sigma_2]$&$\beta_{23}$\\
9&$[\sigma_1,\sigma_2,\sigma_1,\sigma_2,\sigma_1\sigma_2]$&$\beta_{12}\beta_{23}$\\
\noalign{\smallskip}\hline
\end{tabular}
\end{table}
\begin{table}[!htbp]
\centering
\footnotesize
\captionsetup{font=footnotesize}
\caption{The action of $\beta_{ij}$ on $Z_9$}
\label{tab:PermutationRepresentationOnZ9Betaij}
\begin{tabular}{ll}
\hline\noalign{\smallskip}
$i,j$ & $\rho_5(\beta_{ij})$ \\
\noalign{\smallskip}\hline\noalign{\smallskip}
1,2   & $(1,2,5)(3,9,4)(6)(7)(8)$ \\
1,3   & $(1,3,7)(2)(4)(5,8,9)(6)$ \\
1,4   & $(1,4,8)(2)(3,5,6)(7)(9)$ \\
1,5   & $(1,2)(3,6)(4,7)(5)(8,9)$ \\
\noalign{\smallskip}\hline\noalign{\smallskip}
2,3   & $(1,8,4)(2,9,7)(3)(5)(6)$ \\
2,4   & $(1,6,9)(2,3,8)(4)(5)(7)$ \\
2,5   & $(2,5)(3,7)(4,8)(6,9)(1)$ \\
\noalign{\smallskip}\hline\noalign{\smallskip}
3,4   & $(3,4,9)(5)(6,7,8)(1)(2)$ \\
3,5   & $(1,4)(2,3)(5,9)(6,7)(8)$ \\
\noalign{\smallskip}\hline\noalign{\smallskip}
4,5   & $(1,9)(2,4)(3,5)(6)(7,8)$ \\
\noalign{\smallskip}\hline
\end{tabular}
\end{table}
\noindent
Note that elements in $D$ have order 3 due to 
$$(\sigma_2\sigma_1)^3 =(\sigma_2\sigma_1\sigma_2)(\sigma_1 \sigma_2\sigma_1)=(\sigma_1 \sigma_2\sigma_1)^2=\iota.$$
\begin{example}
\label{L2_8_7A7A7A7A3A}
For the class vector $(7A,7A,7A,7A,3A)$ of $L_2(8)$, the set
\par
\noindent
$\Sigma^i(7A,7A,7A,7A,3A)$ splits under the action of $B_5$ into
two braid orbits $Z_9$ and $Z_{5,904}$ of length $9$ and $5,904$.
We have $[\underline{\sigma}] \in Z_9$ with
$$[\underline{\sigma}]=[(2,9,4,3,5,7,6), (1,4,2,8,7,9,5), (2,9,4,3,5,7,6),$$
$$(2,9,4,3,5,7,6), (2,8,6)(3,5,1)(4,9,7)].$$
$[\underline{\sigma}]$ satisfies $\sigma_1 = \sigma_3 = \sigma_4$, $\sigma_5 = \sigma_2 \sigma_1$,
$o(\sigma_1 \sigma_2 \sigma_1 )=2$ and $\sigma_1 \sigma_2 \sigma_1 = \sigma_2 \sigma_1 \sigma_2$,
thus $[\underline{\sigma}] = p_1$ and $Z_9$ is the $B_5$-orbit from Theorem~\ref{BraidOrbitsOfSize9}.
\end{example}
\noindent
Looking at $(k_1)_{(i_1,i_2,i_3)}$, $(k_1)_{(i_1,i_2)}$ and $(k_1)_{(i_1)}$ for all possible 
values of $i_1, i_2, i_3$ leads to some more braid orbits deduced from $Z_{40}$ by specialization. 
Table \ref{tab:SomeBraidOrbitsDerivedFromZ40} contains the interesting cases. 
We omit $(k_1)_{(i_1,i_2,i_3)}$ because $B_3$ acts trivial on $\Sigma^i(C_1,C_2,C_3)$.
Note that we have $Z_{s,m} = \{(k_1)_s\}^{B_m} $ in table \ref{tab:SomeBraidOrbitsDerivedFromZ40}.
\begin{table}[!htbp]
\centering
\footnotesize
\captionsetup{font=footnotesize}
\caption{Some braid orbits deduced from $Z_{40}$ by specialization}
\label{tab:SomeBraidOrbitsDerivedFromZ40}
\begin{tabular}{cllrll}
\hline\noalign{\smallskip}
$s$ & $m$ & $(k_1)_s$ & $\vert Z_{s,m} \vert$ & $\rho_m(B_m)$ on $Z_{s,m}$ & $g_{Z_{s,m}}$ \\
\noalign{\smallskip}\hline\noalign{\smallskip}
()     & 6 & $[\sigma_1,\sigma_2,\sigma_1,\sigma_1,\sigma_2,\sigma_1]$ & 40 & $S_4(3)$ & 6 \\
(5)    & 5 & $[\sigma_1,\sigma_2,\sigma_1,\sigma_1,\sigma_2\sigma_1]$ &  9 & $((C_3 \times C_3) \rtimes Q_8) \rtimes C_3$ & 0 \\
(3)    & 5 & $[\sigma_1,\sigma_2,\sigma_1\sigma_1,\sigma_2,\sigma_1]$ & 12 & $((C_3 \times C_3) \rtimes Q_8) \rtimes C_3$ & 0 \\
(4, 4) & 4 & $[\sigma_1,\sigma_2,\sigma_1,\sigma_1\sigma_2\sigma_1]$ &  4 & $A_4 \cong S_2(3)$ & 0 \\
(3, 4) & 4 & $[\sigma_1,\sigma_2,\sigma_1\sigma_1,\sigma_2\sigma_1]$ &  3 & $S_3 \cong S_2(2)$ & 0 \\
(2, 3) & 4 & $[\sigma_1,\sigma_2\sigma_1,\sigma_1\sigma_2,\sigma_1]$ &  2 & $C_2$ & 0 \\
\noalign{\smallskip}\hline
\end{tabular}
\end{table}

\section{Heuristics for searching class vectors}\label{SectionHeuristicsForSearching}

In this section, let $m \geq 4$. For realizing finite groups $G$ as Galois groups using braid action, 
we need $B_m$-orbits $Z \subseteq \Sigma^i(C_1,\dots, C_m)$ of $G$ with $g_Z=0$
or small values of $g_Z$ and $g_{Z^{sy}}=0$ for $m=4$.
The formula for the computation of $g_Z$ is
$$g_Z = 1 -  \vert Z \vert  + \frac{1}{2}((m-1)\cdot  \vert Z \vert  - \sum\limits_{j=2}^{m} z_{1j})$$
with $z_{1j}$, $j=2,\dots,m$ the number of cycles of $\rho_m(\beta_{1j})$ on $Z$.
We have already seen in section \ref{SectionTheActionOfBraidsInDimension4} that the length of 
cycles of $\rho_m(\beta_{1j})$, $j=2,\dots,m$ is bounded above by orders of elements in $G$.
Thus, if $ \vert Z \vert $ increases, $z_{1j}$, $j=2,\dots,m$ must increase. 
For a finite group $G$, let
$${\bar o}_G = \frac{1}{ \vert G \vert } \sum\limits_{\sigma \in G} o(\sigma) =
\frac{1}{ \vert G \vert } \sum\limits_{C \in cl(G)} o(C) \vert C \vert  \in \mathbb{Q}$$
be the average element order of $G$ where the second sum runs over all conjugacy classes $C$ in $G$.
Setting $ \vert Z \vert ={\bar o}_G \cdot x$, assuming  $\lim\limits_{ \vert Z \vert \to\infty} z_{1j}=x$, $j=2,\dots,m$
and using the formula for $g_Z$ leads to
$${\tilde g}_Z = \frac{(m-3){\bar o}_G -m+1}{2{\bar o}_G} \vert Z \vert +1$$
and
$$\lim\limits_{\vert Z \vert \to\infty} \frac{{\tilde g}_Z}{\vert Z \vert} =
\frac{(m-3){\bar o}_G -m+1}{2{\bar o}_G} = l_{m, G}.$$
Table \ref{tab:Data_LimitRatioClassvectorsM23M24Sz8} contains the values of $l_{m, G}$ for
class vectors of $M_{23}$, $M_{24}$ and $Sz(8)$. For simplicity, we use the rounded values
${\bar o}_{M_{23}}=10.9757$, ${\bar o}_{M_{24}}=13.0090$ and ${\bar o}_{Sz(8)}=7.5312$.
\begin{table}[!htbp]
\centering
\footnotesize
\captionsetup{font=footnotesize}
\caption{Limit $l_{m, G}$ for class vectors of $M_{23}$, $M_{24}$ and $Sz(8)$}
\label{tab:Data_LimitRatioClassvectorsM23M24Sz8}
\begin{tabular}{llr}
\hline\noalign{\smallskip}
$m$ & $G$ & $l_{m, G}$ \\
\noalign{\smallskip}\hline\noalign{\smallskip}
4 & $M_{23}$ & 0.3633 \\
4 & $M_{24}$ & 0.3847 \\
4 & $Sz(8)$ & 0.3008 \\
\noalign{\smallskip}\hline\noalign{\smallskip}
5 & $M_{23}$ & 0.8178 \\
5 & $M_{24}$ & 0.8463 \\
5 & $Sz(8)$ & 0.7344 \\
\noalign{\smallskip}\hline\noalign{\smallskip}
6 & $M_{23}$ & 1.2722 \\
6 & $M_{24}$ & 1.3078 \\
6 & $Sz(8)$ & 1.1680 \\
\noalign{\smallskip}\hline
\end{tabular}
\end{table}
In tables \ref{tab:Data_56DimClassvectors} and \ref{tab:Data_Ratio4DimClassvectors} we list
rounded values for $\frac{g_Z}{\vert Z \vert}$.
\begin{table}[!htbp]
\centering
\footnotesize
\captionsetup{font=footnotesize}
\caption{Ratio $\frac{g_Z}{\vert Z \vert}$ for class vectors of length $4$ of $M_{23}$}
\label{tab:Data_Ratio4DimClassvectors}
\begin{tabular}{llrrr}
\hline\noalign{\smallskip}
$G$ & $C$ & $\vert Z \vert$ & $g_Z$ & $\frac{g_Z}{\vert Z \vert}$ \\
\noalign{\smallskip}\hline\noalign{\smallskip}
$M_{23}$ & $(2A,2A,3A,5A)$ & 980 & 119 & 0.1214\\
$M_{23}$ & $(2A,2A,8A,8A)$ & 198,488 & 41,277 & 0.2080\\
$M_{23}$ & $(2A,3A,6A,8A)$ & 2,040,294 & 587,348 & 0.2879\\
$M_{23}$ & $(2A,4A,4A,8A)$ & 4,075,896 & 1,191,571 & 0.2923\\
$M_{23}$ & $(2A,4A,5A,6A)$ & 5,742,480 & 1,709,427 & 0.2977\\
$M_{23}$ & $(2A,4A,5A,8A)$ & 8,779,060 & 2,630,191 & 0.2996\\
$M_{23}$ & $(2A,4A,6A,6A)$ & 7,814,940 & 2,405,266 & 0.3078\\
$M_{23}$ & $(2A,4A,6A,8A)$ & 11,962,894 & 3,697,305 & 0.3091\\
\noalign{\smallskip}\hline
\end{tabular}
\end{table}
With increasing orbit size $\vert Z \vert$, the ratio $\frac{g_Z}{\vert Z \vert}$ gets closer to $l_{m, G}$.
Thus we may state
\begin{conjecture}
\label{LimitGenus} 
Let $G$ be a finite group with $Z(G)=I$, then
$$\lim\limits_{\vert Z \vert \to\infty} \frac{g_Z}{\vert Z \vert} = l_{m, G}$$
and $g_Z$ increases linearly with $ \vert Z \vert $.
\end{conjecture}
\noindent
The linear growth of $g_Z$ tells us to look for small $B_m$-orbits $Z \subseteq \Sigma^i(C_1,\dots, C_m)$.
\par
\noindent
If $B_m$ acts transitively on $\Sigma^i(C_1,\dots, C_m)$, we have a single $B_m$-orbit $Z$
with $\vert Z \vert = l^i(C_1,\dots, C_m)$.
This motivates to look at the growth of $l^i(C_1,\dots, C_m)$ for increasing dimension $m$.
As there are no non-trivial class vectors with $m=1$ and class vectors with $m=2$ lead to cyclic groups $G$,
we always assume $m \geq 3$.
\begin{proposition}
\label{LowerBoundLiC} 
Let $G$ be a finite group with $Z(G) = I$, $(C_1,\dots, C_m)$ a class vector of $G$
with $m \geq 6$ and $3 \leq n \leq m-3$, then
$$l^i(C_1,\dots, C_m) \geq \vert G \vert \cdot l^i(C_1,\dots, C_n) \cdot l^i(C_{n+1},\dots, C_m).$$
\end{proposition}
\begin{proof}
Without restriction we may assume $l^i(C_1,\dots, C_n) \cdot l^i(C_{n+1},\dots, C_m) > 0$.
For $[\underline{\sigma}] \in S = \Sigma^i(C_1,\dots, C_n)$ and $[\underline{\tau}] \in T = \Sigma^i(C_{n+1},\dots, C_m)$
we have 
$$[\underline{\sigma}, \underline{\tau}] = 
[\sigma_1^{\gamma}, \dots, \sigma_n^{\gamma}, \tau_{n+1}^{\eta}, \dots, \tau_m^{\eta}] \in \Sigma^i(C_1,\dots, C_m)$$
for all $\gamma \in G$ and all $\eta \in G$.
Setting $\underline{\omega}_{\kappa} = [\sigma_1, \dots, \sigma_n, \tau_{n+1}^{\kappa}, \dots, \tau_m^{\kappa}]$
with $\kappa \in G$ and
$$W_{[\underline{\sigma}], [\underline{\tau}]} = \{ [\underline{\omega}_{\kappa}] \mid \kappa \in G \} \subseteq \Sigma^i(C_1,\dots, C_m),$$
we get $\vert W_{[\underline{\sigma}], [\underline{\tau}]} \vert = \vert G \vert$. The set
$$Y_{S,T} = \bigcup_{[\underline{\sigma}] \in S,[\underline{\tau}] \in T} W_{[\underline{\sigma}], [\underline{\tau}]} 
\subseteq \Sigma^i(C_1,\dots, C_m)$$ is a disjoint union and has therefore cardinality
$$\vert Y_{S,T} \vert = \sum_{[\underline{\sigma}] \in S,[\underline{\tau}] \in T} W_{[\underline{\sigma}], [\underline{\tau}]} =
\sum_{[\underline{\sigma}] \in S,[\underline{\tau}] \in T} \vert G \vert = 
\vert G \vert \cdot \vert S \vert \cdot \vert T \vert $$
\par
\noindent
Finally, we get
$$l^i(C_1,\dots, C_m) \geq \vert Y_{S,T} \vert = \vert G \vert \cdot l^i(C_1,\dots, C_n) \cdot l^i(C_{n+1},\dots, C_m).$$
\end{proof}
\noindent
Considering all non-trivial splittings of $(C_1,\dots, C_m)$ and using \cite[I, Corollary 5.6]{RefMM2018} leads to
\begin{corollary}
\label{SumOfLowerBoundLiC} 
Let $G$ be a finite group with $Z(G) = I$, $(C_1,\dots, C_m)$ a class vector of $G$
with $m \geq 6$ and $3 \leq n \leq m-3$, then
$$\lfloor n(C_1,\dots, C_m) \rfloor \geq
l^i(C_1,\dots, C_m) \geq \vert G \vert \cdot \sum_{n=3}^{m-3} l^i(C_1,\dots, C_n) \cdot l^i(C_{n+1},\dots, C_m).$$
\end{corollary}
\noindent
Using Proposition~\ref{LowerBoundLiC} we get
\begin{proposition}
\label{LiCNotZero} 
Let $G$ be a finite group with $Z(G) = I$ and
$l^i(D_1,\dots, D_n) \geq 1$ for all class vectors $(D_1,\dots, D_n)$ of $G$ of length $n$ 
with $n \in \{ n_0, n_0+1,\dots, 2 \cdot n_0-1 \}$ for some integer $n_0 \geq 3$, then 
$$l^i(C_1,\dots, C_m) \geq {\vert G \vert}^{k-1}$$
holds for all class vectors of length $m$ with $m=k \cdot n_0 + l$, $l \in \{0,1,\dots, n_0-1 \}$ and $k \geq 2$.
\end{proposition}
\begin{proof}
Let $m=k \cdot n_0$ with $k \geq 2$ and choose $n=n_0$. Now we can apply Proposition~\ref{LowerBoundLiC} and get
$$l^i(C_1,\dots, C_{k \cdot n_0}) \geq
\vert G \vert \cdot l^i(C_1,\dots, C_{n_0}) \cdot l^i(C_{n_0 + 1},\dots, C_{k \cdot n_0}) \geq$$
$$\vert G^2 \vert \cdot l^i(C_1,\dots, C_{n_0}) \cdot l^i(C_{n_0 + 1},\dots, C_{2 \cdot n_0}) 
\cdot l^i(C_{2 \cdot n_0 + 1},\dots, C_{k \cdot n_0}) \geq ...$$
After $k-1$ steps we get
$$l^i(C_1,\dots, C_{k \cdot n_0}) \geq
{\vert G \vert}^{k-1} \prod_{l=0}^{k-1} {l^i(C_{l \cdot n_0 + 1},\dots, C_{(l+1) \cdot n_0})}$$
with
$l^i(C_{l \cdot n_0 + 1},\dots, C_{(l+1) \cdot n_0}) \geq 1$ for $l=0, \dots, k-1$ by assumption.
\par
\noindent
For $m=k \cdot n_0 + l$ with $l \in \{0,1,\dots, n_0-1 \}$ and $k \geq 2$, choose $n=n_0 + l$. Again, we can
apply Proposition~\ref{LowerBoundLiC} and get $l^i(C_1,\dots, C_{k \cdot n_0 + l}) \geq {\vert G \vert}^{k-1}$
because of
$$l^i(C_1,\dots, C_{k \cdot n_0 + l}) \geq
\vert G \vert \cdot l^i(C_1,\dots, C_{n_0 + l}) \cdot l^i(C_{n_0 + l + 1},\dots, C_{k \cdot n_0 + l})$$
and 
$$l(C_{n_0 + l + 1},\dots, C_{k \cdot n_0 + l}) = k \cdot n_0 + l - (n_0 + l + 1) + 1 = (k-1) \cdot n_0.$$
Here $l(...)$ means the length of a class vector.
\end{proof}
\noindent
Note that Proposition~\ref{LiCNotZero} can only apply to simple groups.
If $G$ is a finite group with a proper normal subgroup $N \vartriangleleft G$, we can choose
a non-trivial class $C$ of $G$ with $C \subseteq N$ and get
$l^i(C, \dots, C) = 0$ for all class vectors $(C, \dots, C)$ of $G$ containing only class $C$.
This was pointed out by G. Malle.
\par
\noindent
If it exists, we have $n_0 \geq 6$ for $M_{23}$ due to $l^i(2A,2A,2A,2A,2A)=0$.
\begin{corollary}
\label{LiCInfinite} 
With the assumptions of Proposition~\ref{LiCNotZero}, the set of class vectors
$(C_1,\dots, C_m)$ with $l^i(C_1,\dots, C_m)=0$ is finite
and
$$\lim\limits_{m \to\infty} l^i(C_1,\dots, C_m) = \infty$$
holds.
\end{corollary}
\noindent
Combining Conjecture~\ref{LimitGenus} and Corollary~\ref{LiCInfinite}, we get 
$$\lim\limits_{m \to\infty} g_Z = \infty$$
for class vectors of length $m$ with transitive $B_m$-action of simple groups satisfying
the assumptions of Proposition~\ref{LiCNotZero}. Therefore the key to small values of $g_Z$
are class vectors of length $m$ with intransitive $B_m$-action.
\par
\noindent
Braid orbit computations for the groups $A_5$, $A_6$ in \cite{RefJam2013},
$L_2(7)$ and $L_2(11)$ in \cite{RefFir2015} and
$A_7$, $A_8$,
$L_2(8)$, $L_2(13)$, $L_2(16)$, $L_3(3)$, $L_3(4)$, $L_3(5)$,
$M_{11}$, $M_{12}$, $M_{22}$, $M_{23}$, $M_{24}$,
$S_5$, $S_6$, $S_7$, $S_8$,
$SO_5(3)$, $S_6(2)$, 
$Sz(8)$, $Sz(32)$
by the author let us expect small braid orbits especially in symmetric class vectors.
Looking at examples shows that symmetric class vectors indeed often
lead to a splitting of $\Sigma^i(C,\dots, C)$ into multiple $B_m$-orbits
while $B_m$ acts transitively in case of non-symmetric class vectors with only few exceptions.
For $m=4$, small rigid $B_4$-orbits occur that cannot be explained by Theorem~\ref{BraidOrbitsOfSize2},
Theorem~\ref{BraidOrbitsOfSize6}, Theorem~\ref{TranslationOfBraidOrbits} or Theorem~\ref{BraidOrbitsOfSize4}.
This is what we are looking for in section \ref{SectionClassVectorsOfDimension4InM23}.
\par
\noindent
In table \ref{tab:Data_SmallRigidB4Orbits} we give examples of such small rigid $B_4$-orbits.
All $B_4$-orbits $Z$ in these examples have the following properties:
\begin{itemize}
\item[]
$\vert Z \vert = k \cdot l, k \in \{1, 2, 3\}, l \in \{120, 144\}$,
\item[]
$\vert Z^{sy} \vert = \vert Z \vert$,
\item[]
$F_{B_4}(F_i,[\underline{\sigma}])=-1, [\underline{\sigma}] \in Z, i=1,2,3$,
\item[]
$\rho^t_4(\eta_{23}) = (2)^{\vert Z^{sy} \vert /2}$,
\item[]
$\rho^t_4(\eta_{234}) = (3)^{\vert Z^{sy} \vert /3}$
\end{itemize}
with the exception of $M_{11}$, $(3A,3A,3A,3A)$ with $\rho^t_4(\eta_{23}) = (1)^{12}(2)^{138}$.
Note that all $B_4$-orbits $Z$ with $\vert Z \vert = 120$ in table \ref{tab:Data_SmallRigidB4Orbits}
have $g_Z = 7$ and $g_{Z^{sy}} = 0$.
\begin{table}[H]
\centering
\footnotesize
\captionsetup{font=footnotesize}
\caption{Some small rigid $B_4$-orbits}
\label{tab:Data_SmallRigidB4Orbits}
\begin{tabular}{lllrl}
\hline\noalign{\smallskip}
$G$ & $C$ & $ \vert Z \vert $      & $g_Z$             & $\rho^t_4(\beta_{1j})$\\
    &     & $ \vert Z^{sy} \vert $ & $g_{Z^{sy}}$ & $\rho^t_4(\beta_3)$\\
\noalign{\smallskip}\hline\noalign{\smallskip}
$A_7$&$(7A_4)$& 120 & 7 & $(1)^8 (2)^{12} (4)^8 (7)^8$\\
     &        & 120 & 0 & $(1)^4 (2)^2 (4)^6 (7)^4 (8)^4 (14)^2$\\
$S_6$&$(6A_4)$& 120 & 7 & $(1)^4 (2)^{12} (4)^8 (5)^{12}$\\
     &        & 120 & 0 & $(1)^2 (2) (4)^6 (5)^6 (8)^4 (10)^3 $\\
$S_8$&$(4D_4)$& 120 & 7 & $(1)^{12} (2)^{12} (7)^{12} $\\
     &        & 120 & 0 & $(1)^2 (2)^5 (4)^6 (7)^6 (14)^3$\\
$L_3(4)$&$(5A_4)$& 240 & 31 & $(2)^{12} (4)^{24} (5)^{24}$\\
        &        & 240 &  3 & $(4)^6 (5)^{12} (8)^{12} (10)^6$\\
$SO_5(3)$&$(4D_4)$& 120 & 7 & $(1)^8 (2)^8 (3)^4 (5)^{12} (6)^4$\\
         &        & 120 & 0 & $(2)^4 (3)^2 (4)^4 (5)^6 (6) (10)^3 (12)^2$\\
$SO_5(3)$&$(6H_4)$& 144 & 19 & $(2)^8 (3)^8 (5)^{16} (6)^4$\\
         &        & 144 &  2 & $(4)^4 (5)^8 (6)^4 (10)^4 (12)^2$\\
$SO_5(3)$&$(6H_4)$& 432 & 25 & $(1)^4 (2)^{32} (3)^{36} (4)^{24} (5)^{32}$\\
         &        & 432 &  1 & $(1)^2 (2) (3)^6 (4)^{16} (5)^8 (6)^{15} (8)^{12} (10)^{12} $\\
$S_6(2)$&$(3C_4)$& 288 & 43 & $(1)^8 (2)^4 (3)^{16} (4)^{16} (6)^8 (7)^{16}$\\
        &        & 288 &  4 & $(2)^4 (3)^8 (4)^2 (6)^4 (7)^8 (8)^8 (12)^4 (14)^4$\\
$S_6(2)$&$(4D_4)$& 360 & 79 & $(1)^4 (2)^{12} (3)^8 (4)^8 (7)^{24} (9)^{12} $\\
        &        & 360 &  8 & $(1)^2 (2) (3)^4 (4)^6 (6)^2 (7)^{12} (8)^4 (9)^6 (14)^6 (18)^3 $\\
$M_{11}$&$(3A_4)$& 288 & 49 & $(1)^4 (2)^8 (3)^8 (5)^{36} (8)^8 $\\
        &        & 288 &  3 & $(1)^2 (2) (3)^4 (4)^4 (5)^6 (6)^2 (10)^{15} (16)^4 $\\
\noalign{\smallskip}\hline
\end{tabular}
\end{table}

\section{The permutation group $\rho_4(B_4)$}\label{SectionTheStructureOfThePermutationGroupRho4B4}

Using GAP \cite{RefGAP}, we tried to compute some properties of the permutation groups $\rho_4(B_4)$ for the
examples in table \ref{tab:Data_SmallRigidB4Orbits}. The results can be seen in table \ref{tab:Data_SmallRigidB4OrbitsAction}.
For the structure descriptions in the last column we use GAP-notation.
Computing more examples for symmetric and non-symmetric class vectors gives rise to
\begin{conjecture}
\label{StructureOfRho4B4}
Let $G > I$ be a finite group with $Z(G)=I$, $(C_1,C_2,C_3,C_4)$ a class vector of length $4$ of $G$,
$Z \subseteq \Sigma^i_4(C_1,C_2,C_3,C_4)$ a $B_4$-orbit with $q = \vert Z \vert > 0$ and
$\rho_4: B_4 \rightarrow S_q$ 
the induced permutation representation of the action of $B_4$ on $Z$, then
$$\rho_4(B_4) \cong C_{n_l}^{k_l} \rtimes (\cdots \rtimes (C_{n_2}^{k_2} 
\rtimes (C_{n_1}^{k_1} \rtimes G_m))\cdots) \leq S_q$$
with $G_m \in \{ A_m, S_m \}$, $q \mid m \cdot n_1 \cdot \dots \cdot n_l$, $m \leq q$,
$l \in \mathbb{N}_0$, $n_1, \dots, n_l \in \{ 2,3 \}$,
$k_1, \dots, k_l \in \mathbb{N}$ and $k_1 = m-1$.
\end{conjecture}
\begin{table}[!htbp]
\centering
\footnotesize
\captionsetup{font=footnotesize}
\caption{Some properties of $\rho_4(B_4)$}
\label{tab:Data_SmallRigidB4OrbitsAction}
\begin{tabular}{lllll}
\hline\noalign{\smallskip}
$G$ & $C$ & $ \vert Z \vert $ & $\vert \rho_4(B_4) \vert$ & Structure $\rho_4(B_4)$\\
\noalign{\smallskip}\hline\noalign{\smallskip}
$A_7$ & $(7A_4)$ & 120 & $2^{57}\cdot 30!$ & $C_2^{29} : (C_2^{29} : A_{30})$\\
$S_6$ & $(6A_4)$ & 120 & $2^{48}\cdot 3^{10}\cdot 10!$ & $C_2^{20} : (C_2^{20} : (C_3^{10} : (C_2^9 : A_{10})))$\\
$S_8$ & $(4D_4)$ & 120 & $2^{48}\cdot 3^{10}\cdot 10!$ & $C_2^{20} : (C_2^{20} : (C_3^{10} : (C_2^9 : A_{10})))$\\
$L_3(4)$ & $(5A_4)$ & 240 & $2^{28}\cdot 3^{10}\cdot 10!$ & $C_2^{10} : (C_2^{10} : (C_3^{10} : (C_2^9 : A_{10})))$\\
$SO_5(3)$ & $(4D_4)$ & 120 & $2^{48}\cdot 3^{10}\cdot 10!$ & $C_2^{20} : (C_2^{20} : (C_3^{10} : (C_2^9 : A_{10})))$\\
$SO_5(3)$ & $(6H_4)$ & 144 & $2^{51}\cdot 9!$ & $?$\\
$SO_5(3)$ & $(6H_4)$ & 432 & $2^{213}\cdot 108!$ & $?$\\
$S_6(2)$ & $(3C_4)$ & 288 & $2^{159}\cdot 18!$ & $?$\\
$S_6(2)$ & $(4D_4)$ & 360 & $2^{148}\cdot 3^{30}\cdot 30!$ & $?$\\
$M_{11}$ & $(3A_4)$ & 288 & $2^{141}\cdot 72!$ & $?$\\
\noalign{\smallskip}\hline
\end{tabular}
\end{table}
\begin{corollary}
\label{StructureOfRho4B4PrimeSize}
If Conjecture~\ref{StructureOfRho4B4} is true and $q \in \mathbb{P}$, then $\rho_4(B_4) \cong G_q$.
\end{corollary}
\begin{proof}
We have $\rho_4(B_4) \cong C_2 \cong S_2$ for $q = 2$. The transitive subgroups of $S_3$ are $A_3 \cong C_3$
and $S_3$ itself, thus $\rho_4(B_4) \cong G_3$ for $q = 3$. If $q$ is a prime with $q > 3$, we get $q \mid m$
and $\rho_4(B_4) \cong G_q$ follows.
\end{proof}
\begin{corollary}
\label{StructureOfRho4B4PrimeSizeSymmetric}
If Conjecture~\ref{StructureOfRho4B4} is true and if $(C,C,C,C)$ is a symmetric class vector with $Z^{H_4} = Z$
and $q \in \mathbb{P}$, then $\rho_4(B_4) \cong A_q$.
\end{corollary}
\begin{proof}
From Corollary~\ref{StructureOfRho4B4PrimeSize} we get $\rho_4(B_4) \cong G_q$.
Now we have a permutation representation
${\rho_4}^{sy}: H_4 \rightarrow S_q$ with ${\rho_4}^{sy} {\mid}_{B_4} = \rho_4$.
From Proposition~\ref{PropertiesOfPermutationRepresentations}(a) follows $\rho_4(B_4) \leq A_q$
and we finally get $\rho_4(B_4) \cong A_q$.
\end{proof}
\begin{example}
\label{L3_5_31A31A31A31A}
For the class vector $(31A,31A,31A,31A)$ of the projective special linear group $L_3(5)$ 
there exists exactly one $B_4$-orbit $Z$ with 
$q = \vert Z \vert = 31$ and $Z^{H_4} = Z$. Using GAP \cite{RefGAP},
we get $\rho_4(B_4) \cong A_{31}$ and $\rho_4(H_4) \cong S_{31}$ for $Z$.
This confirms Corollary~\ref{StructureOfRho4B4PrimeSizeSymmetric}.
\end{example}
\begin{example}
\label{L2_16_17A17A17A17A17A17A}
For the class vector $(17A,17A,17A,17A,17A,17A)$ of $L_2(16)$ there exist exactly two $B_6$-orbits $Z_{40}$
and $Z_{11,511,567}$ with $\vert Z_{40} \vert = 40$ and $\vert Z_{11,511,567} \vert = 11,511,567$.
Using GAP \cite{RefGAP}, we get $\rho_6(B_6) \cong  S_4(3)$ for the action of $B_6$ on $Z_{40}$.
Thus Conjecture~\ref{StructureOfRho4B4} cannot be extended to class vectors of length $6$.
\end{example}

\section{Conclusion}\label{SectionConclusion}

In inverse Galois theory, the action of the pure Hurwitz braid group $B_m$ on classes of generating systems
in $\Sigma^i(C_1,\dots, C_m)$ for a finite group $G$ is a major tool for realising $G$ as Galois group
over rational function fields.
At present, little is known about this action. Here, some generic $B_m$-orbits have been found and the $B_m$-action
on these orbits could be computed completely. Besides Fried's lifting invariant in \cite{RefFri2010} and
the theorem of Conway and Parker in \cite[III, Theorem 6.10]{RefMM2018} and \cite[Corollary 6.11]{RefMM2018},
this is a step towards a better understanding of the $B_m$-action. We can now explain the occurrence of 
certain $B_m$-orbits in $\Sigma^i(C_1,\dots, C_m)$ in several cases. These $B_m$-orbits originate from
fixed points under the $B_m$-action resp. $H_m$-action of finite subgroups of $H_m$.
\par
\noindent
Until now, applying this knowledge to the Mathieu group $M_{23}$ to find suitable $B_m$-orbits
that can be used to realize $M_{23}$ as Galois group over rational function fields with field of constants $\mathbb{Q}$
did not lead to success. One reason is that $M_{23}$ has no $(2,3)$-generating systems.
Thus the problem, whether $M_{23}$ occurs as Galois group over $\mathbb{Q}$ still remains open.
\par
\noindent
We have seen that symmetric class vectors often lead to small $B_m$-orbits, so we should consider them.
For $M_{23}$ and $m>4$, these class vectors are a challenge for currently available computers.

{\it Email address:} {\tt frankdothaefner@web.de}
\end{document}